\documentclass[11pt, letterpaper]{amsart}
\usepackage{ amsfonts,amsmath, amsthm, amssymb, latexsym, epsfig}
\usepackage[all]{xy}
\usepackage{xspace}
\usepackage{comment}
\usepackage{setspace}
\usepackage{enumerate}
\usepackage[mathscr]{eucal}
\RequirePackage{pdfsync}

\usepackage{bbold}

	\advance \topmargin by -\headheight
	\advance \topmargin by -\headsep

	\evensidemargin \oddsidemargin
	\newcommand{\ftn}[3]{ #1 : #2 \rightarrow #3 }
	\newcommand{\setof}[2]{\ensuremath{\left\{ #1 \: : \: #2 \right\}}}
	\newcommand{\norm}[1]{\left\| #1 \right\|}
	\newcommand{\p}{\mathfrak{p}}
	\newcommand{\q}{\mathfrak{q}}
	
	\newcommand{\dirlim}{\displaystyle \lim_{\longrightarrow}}
	\newcommand{\Hom}{\ensuremath{\operatorname{Hom}}}
	\newcommand{\multialg}{\mathcal{M}}
	\newcommand{\id}{\ensuremath{\operatorname{id}}}
	\newcommand{\Z}{\ensuremath{\mathbb{Z}}\xspace}
	\newcommand{\C}{\ensuremath{\mathbb{C}}\xspace}
	\newcommand{\Q}{\ensuremath{\mathbb{Q}}\xspace}
	\newcommand{\R}{\ensuremath{\mathbb{R}}\xspace}
	\newcommand{\N}{\ensuremath{\mathbb{N}}\xspace}
	
	\newcommand{\K}{\ensuremath{\mathbb{K}}\xspace}
        \newcommand{\A}{\mathfrak{A}}
        \newcommand{\M}{\mathsf{M}}
        \newcommand{\ZC}{\mathcal{Z}}

        \theoremstyle{plain}
	\newtheorem{thm}{Theorem}[section]
	\newtheorem{lemma}[thm]{Lemma}
	\newtheorem{theorem}[thm]{Theorem}
	
	\newtheorem{corollary}[thm]{Corollary}

	\theoremstyle{definition}
	\newtheorem{definition}[thm]{Definition}
	
	\newtheorem{remark}[thm]{Remark}
	\newtheorem{notation}[thm]{Notation}

	\numberwithin{equation}{section}
	\numberwithin{figure}{section}

\begin{document}
\title[The Automorphism group of a simple $C^{*}$-algebra]%
	{The Automorphism group of a simple $\mathcal{Z}$-stable $C^{*}$-algebra}	
	
	\author{Ping Wong Ng}
	\address{University of Louisiana at Lafayette \\
	217 Maxim D. Doucet Hall \\
 	P.O.Box 41010 \\
	Lafayette, LA 70504-1010 USA }
	\email{png@louisiana.edu}
	\author{Efren Ruiz}
        \address{Department of Mathematics \\
        University of Hawaii Hilo \\
        200 W. Kawili St. \\
        Hilo, Hawaii 96766 USA}
        \email{ruize@hawaii.edu}
        \date{\today}
	

	
	\begin{abstract}
	We study the automorphism group of a simple, unital, $\mathcal{Z}$-stable $C^{*}$-algebra.  We show that $\overline{\mathrm{Inn}}_{0} ( \mathfrak{A} )$ is a simple topological group and $\frac{ \overline{ \mathrm{Inn} } ( \mathfrak{A} ) }{ \overline{\mathrm{Inn}}_{0} ( \mathfrak{A} ) }$ is isomorphic (as topological groups) to the inverse limit of quotient groups of $K_{1} ( \mathfrak{A} )$, where $\mathfrak{A}$ is a $\mathcal{Z}$-stable $C^{*}$-algebra satisfying the following property:  for every UHF algebra $\mathfrak{B}$, $\mathfrak{A} \otimes \mathfrak{B}$ is a nuclear, separable, simple, tracially AI algebra satisfying the Universal Coefficient Theorem (UCT) of Rosenberg and Schochet.  By the recent results of Lin and Winter, ordered $K$-theory, traces, and the class of the unit is a complete isomorphism invariant for this class of $C^{ *}$-algebras.
	\end{abstract}
        \maketitle

\section{Introduction}

Denote the group of automorphisms of $\mathfrak{A}$ equipped with the topology of pointwise convergence by $\mathrm{Aut}( \mathfrak{A} )$, denote the closure of the group of inner automorphisms of $\mathfrak{A}$ by $\overline{ \mathrm{Inn} } ( \mathfrak{A} )$, and denote the closure of the group of inner automorphisms of $\mathfrak{A}$ whose implementing unitaries are connected to $1_{\mathfrak{A}}$ via a norm continuous path of unitaries by $\overline{ \mathrm{ Inn } }_{0} ( \mathfrak{A} )$.  We then have that $\mathrm{Aut} ( \mathfrak{A} )$ decomposes into the following series of closed normal subgroups
\begin{equation*}
\overline{ \mathrm{Inn} }_{0} ( \mathfrak{A} ) \lhd \overline{ \mathrm{Inn} } ( \mathfrak{A} ) \lhd \mathrm{Aut} ( \mathfrak{A} ).
\end{equation*} 
  
In \cite{EllRorAuto}, Elliott and R{\o}rdam showed that for a simple, unital $C^{*}$-algebra $\mathfrak{A}$ that either is real rank zero, stable rank one and weakly unperforated, or is purely infinite, $\overline{ \mathrm{Inn} }_{0} ( \mathfrak{A} )$ is a simple topological group (no non-trivial closed normal subgroup).  They also showed that $\frac{ \overline{ \mathrm{ Inn } } ( \mathfrak{A} ) }{ \overline{ \mathrm{ Inn } }_{0} ( \mathfrak{A} ) }$ is totally disconnected when $\mathfrak{A}$ is a simple AT algebra with real rank zero.  Hence, when $\mathfrak{A}$ is a simple AT algebra with real rank zero, the results of Elliott and R{\o}rdam give a structure theorem for $\mathrm{Aut} ( \mathfrak{A} )$ since $\mathrm{Aut} ( \mathfrak{A} )$ fits into the following exact sequence
\begin{equation*}
\{ 1 \} \to \overline{ \mathrm{ Inn } } ( \mathfrak{A} ) \to \mathrm{Aut} ( \mathfrak{A} ) \to \mathrm{Aut} ( K_{*} ( \mathfrak{A} ) )_{+,1} \to \{ 1 \}.
\end{equation*}
In their paper, they asked if $\overline{ \mathrm{Inn} }_{0} ( \mathfrak{A} )$ is a simple topological group and if $\frac{ \overline{ \mathrm{ Inn } } ( \mathfrak{A} ) }{ \overline{ \mathrm{ Inn } }_{0} ( \mathfrak{A} ) }$ is totally disconnected for every simple, unital $C^{*}$-algebra $\mathfrak{A}$.  

Recently, the authors in \cite{pr_auto} and \cite{pner_ugrps} proved $\overline{ \mathrm{Inn} }_{0} ( \mathfrak{A} )$ is a simple topological group and $\frac{ \overline{ \mathrm{ Inn } } ( \mathfrak{A} ) }{ \overline{ \mathrm{ Inn } }_{0} ( \mathfrak{A} ) }$ is totally disconnected for all nuclear, separable, simple, tracially AI algebras satisfying the UCT and for all nuclear, purely infinite, separable, simple $C^{*}$-algebras satisfying the UCT.  In this paper, we generalize the results of \cite{pr_auto} and \cite{pner_ugrps}.  We show that $\frac{ \overline{ \mathrm{ Inn } } ( \mathfrak{A} ) }{ \overline{ \mathrm{ Inn } }_{0} ( \mathfrak{A} ) }$ is isomorphic to an inverse limit of discrete abelian groups (similar to the one used by Elliott and R{\o}rdam in \cite{EllRorAuto}) for a nuclear, separable, simple $\mathcal{Z}$-stable $C^{*}$-algebra $\mathfrak{A}$ such that for each supernatural number $\mathfrak{p}$ of infinite type, $\mathfrak{A} \otimes \mathsf{M}_{ \mathfrak{p} }$ is a tracially AI algebra that satisfies the UCT.  
(Here, $\mathsf{M}_{ \mathfrak{p} }$ is the UHF algebra with supernatural
number $\mathfrak{p}$.) 
In fact, we show that $\frac{ \overline{ \mathrm{ Inn } } ( \mathfrak{A} ) }{ \overline{ \mathrm{ Inn } }_{0} ( \mathfrak{A} ) }$ is isomorphic (as topological groups) to the inverse limit of quotient groups of $K_{1} ( \mathfrak{A} )$, where the quotient groups are given the discrete topology and the inverse limit is given the inverse limit topology.  
Consequently, $\frac{ \overline{ \mathrm{ Inn } } ( \mathfrak{A} ) }{ \overline{ \mathrm{ Inn } }_{0} ( \mathfrak{A} ) }$ is totally disconnected.  
Moreover, we show that $\overline{ \mathrm{Inn} }_{0} ( \mathfrak{A} )$ is a simple topological group for any separable, simple, unital, $\mathcal{Z}$-stable $C^{*}$-algebra that is either nuclear and quasidiagonal or is exact and
has a unique tracial state.  

It turns out that a large class of simple $C^{*}$-algebras satisfies the above condition.  We give examples of $C^{*}$-algebras in this class.
\begin{itemize}
\item[(1)] Nuclear, separable, simple, unital, tracially AI algebras which satisfy the UCT.

\item[(2)] Simple, unital, $\mathcal{Z}$-stable AH algebras (see Corollary 11.12 of \cite{hl_asyunit}).

\item[(3)]  The Jiang-Su algebra $\mathcal{Z}$ (\cite{JiangSu}).

\item[(4)]  Simple, unital $\mathcal{Z}$-stable $C^{*}$-algebras which are locally type I with unique tracial state (see Corollary 5.6 of \cite{LinNiuAdvances}; see also Corollary 8.2 of \cite{ww_localelliott}).

\item[(5)] Simple, unital $\mathcal{Z}$-stable, ASH-algebras $\mathfrak{A}$ such that $T (\mathfrak{A}) = S[1] \left( K_{0} (\mathfrak{A}) \right)$, where $S[1] \left( K_{0} ( \mathfrak{A} ) \right)$ is 
the state space of $K_{0}( \mathfrak{A} )$ (see Corollary
5.5 of \cite{LinNiuAdvances}; also see Corollaries 6.3 and 8.3 of \cite{ww_localelliott}).

\item[(6)]  Simple, unital $A\mathbb{T}D$ algebras (an $A\mathbb{T}D$ algebra is an inductive limit
of dimension drop circle algebras; see Definitions 2.3 and 2.4 of \cite{hlzn_range} ; and see Theorem 4.2 of  \cite{hlzn_range}).
\end{itemize}

The paper is organized as follows: In Section \ref{s:innerauto}, we show that for a separable, simple, unital, $\mathcal{Z}$-stable $C^{*}$-algebra
$\mathfrak{A}$ that is either nuclear and quasidiagonal or is exact and has a unique tracial state, $\overline{ \mathrm{Inn} }_{0} ( \mathfrak{A} )$ is a simple topological group.  In Section \ref{pathunitaries}, we introduce the Bott Maps which are defined by Lin in \cite{hl_asyunitinner} and we present some technical results Theorem \ref{t:bottunit}, Corollary \ref{c:bottunit}, and Theorem \ref{t:approxcomm}.  
In Section \ref{results}, we show that the topological group $\frac{ \overline{ \mathrm{ Inn } } ( \mathfrak{A} ) }{ \overline{ \mathrm{ Inn } }_{0} ( \mathfrak{A} ) }$ is isomorphic to the inverse limit of discrete abelian groups, where the inverse limit is given the inverse limit topology.  We also provide a partial structure theorem for $\mathrm{Aut} ( \mathfrak{A} )$.

\section{Simplicity of $\overline{ \mathrm{Inn} }_{0} ( \mathfrak{A} )$} \label{s:innerauto}

We first start with some notation that will be used throughout the paper.  Let $\mathfrak{A}$ be a $C^{*}$-algebra and let $p$ and $q$ be projections in $\mathfrak{A}$.
\begin{itemize}
\item[(1)] If $p$ and $q$ are \emph{Murray-von Neumann equivalent}, i.e., there exists $v \in \mathfrak{A}$ such that $v^{*} v = p$ and $vv^{*} = q$, we write $p \sim q$.

\item[(2)] If there exists $v \in \mathfrak{A}$ such that $v^{*} v = p$ and $vv^{*} \leq q$, we write $p \precsim q$. 

\item[(3)] If there exists $v \in \mathfrak{A}$ such that $v^{*} v = p$, $vv^{*} \leq q$, and $vv^{*} \neq q$, we write $p \precnsim q$. 

\item[(4)] For $a, b \in \mathfrak{A}$, $(a,b) = a b a^{*} b^{*}$.

\item[(5)]  If $\mathfrak{A}$ is unital, then denote the norm closure of the commutator subgroup of $U( \mathfrak{A} )$ by $CU( \mathfrak{A} )$ and denote the norm closure of the commutator subgroup of $U( \mathfrak{A} )_{0}$ by $CU( \mathfrak{A} )_{0}$.
\end{itemize}

\begin{definition}
Let $\mathfrak{p}$ and $\mathfrak{q}$ be supernatural numbers.  Set
\begin{equation*}
\mathcal{Z}_{ \mathfrak{p} , \mathfrak{q} } = \setof{ f \in C\left( [ 0 , 1 ],
 \mathsf{M}_{ \mathfrak{p} } \otimes \mathsf{M}_{ \mathfrak{q} } \right) }{ 
f(0) \in \mathsf{M}_{ \mathfrak{p} } \otimes 1_{ \mathsf{M}_{ 
\mathfrak{q} } } \ \text{and} \ f(1) \in 1_{ \mathsf{M}_{ 
\mathfrak{p} } } \otimes \mathsf{M}_{ \mathfrak{q} } }.   
\end{equation*}
(Here, $\mathsf{M}_{ \mathfrak{p}}$ is the UHF algebra with supernatural
number $\mathfrak{p}$. Similar for $\mathsf{M}_{ \mathfrak{q} }$.)

We shall regard $\mathcal{Z}_{ \mathfrak{p} , \mathfrak{q} }$ (and any tensor product with it) as $C( [ 0 , 1 ] )$-algebra with the obvious central embedding of $C( [ 0 , 1] )$.

\end{definition}

\begin{lemma} Let $\mathfrak{A}$ be a separable, simple, unital  $C^{*}$-algebra
and let $\mathfrak{C}$ be a UHF algebra.  Then $\mathfrak{A} \otimes \mathfrak{C}$ is $\mathcal{Z}$-stable and hence either purely infinite or
stably finite.  Moreover, if $\mathfrak{A} \otimes \mathfrak{C}$ is (stably) finite
then it has the following
properties:
\begin{enumerate}
\item stable rank one 
\item cancellation of projections
\item strict comparison of positive elements when $\A$ is, additionally,
exact  
\item weak unperforation
\item $K_1$-injectivity
\item the (SP) property
\item For every nonzero projection $p \in \mathfrak{A} \otimes \mathfrak{C}$, for every $n \geq 2$,
$p(\mathfrak{A} \otimes \mathfrak{C})p$ contains a unital sub-$C^{*}$-algebra which is 
isomorphic to $\mathsf{M}_n \oplus \mathsf{M}_{n+1}$. 
\item If $p, q$ are nonzero projections in $\mathfrak{A} \otimes \mathfrak{C}$, then there exist
nonzero projections $p', q'$ in $p(\mathfrak{A} \otimes \mathfrak{C}) p$ and 
$q (\mathfrak{A} \otimes \mathfrak{C}) q$ respectively such that 
$p' \sim q'$.  
\end{enumerate}
\label{lem:AotimesUHFProperties} 
\end{lemma}

\begin{proof} These results are contained in  
\cite{BlackKumjRor},
\cite{GongJiangSu}, \cite{JiangSu}, 
\cite{RordamUHFI}, \cite{RordamUHFII}, \cite{MRrrZabs} 
and \cite{mr_dimstablerk}.
\end{proof}

\begin{lemma}  Let $\mathfrak{A}$ be a unital $\mathcal{Z}$-stable $C^{*}$-algebra.
Then $\mathbb{T} \subseteq CU(\mathfrak{A})_0$, i.e., $CU(\mathfrak{A})_0$ contains all scalar unitaries. 
\label{lem:scalarunitaries}  
\end{lemma}
\begin{proof}
Since $\mathfrak{A} \cong \mathfrak{A} \otimes \mathcal{Z}$, it is enough to show that $\mathbb{T} \subseteq CU( \mathfrak{A} \otimes \mathcal{Z} )_{0}$.  Clearly, $CU(1_{\mathfrak{A}} \otimes \mathcal{Z}) = CU(1_{\mathfrak{A}} \otimes \mathcal{Z})_0 \subseteq
CU(\mathfrak{A} \otimes \mathcal{Z})_0$.
By \cite{NgZ},  $CU(1_{\mathfrak{A}} \otimes \mathcal{Z})_0$ contains all scalar unitaries.  Therefore, $\mathbb{T} \subseteq CU( \mathfrak{A} \otimes \mathcal{Z} )_{0}$.   
\end{proof}

The next lemma can be proven using a spectral theory argument.

\begin{lemma}  For every $\epsilon > 0$,
there exists $\delta > 0$ such that
for any unital $C^*$-algebra $\mathfrak{A}$, if
\begin{enumerate}
\item $p_1, p_2, ..., p_n$ are pairwise orthogonal projections in $\mathfrak{A}$,
\item $q_1, q_2, ..., q_n$ are pairwise orthogonal projections in $\mathfrak{A}$,
\item $\alpha_1, \alpha_2, ..., \alpha_n$ are scalars (complex numbers)
with norm one,
\item  $| \alpha_i - \alpha_j | \geq \epsilon$ for $i \neq j$, and
\item  $\| (\alpha_1 p_1 + \alpha_2 p_2 + \cdots + \alpha_n p_n) -
(\alpha_1 q_1 + \alpha_2 q_2 + \cdots + \alpha_n q_n ) \| < \delta$
\end{enumerate}
then $\| p_i - q_i \| < \epsilon$ and
$p_i \sim q_i$ in $\mathfrak{A}$ for $1 \leq i \leq n$.
\label{lem:FinSpectUnitAppr}
\end{lemma}

\begin{lemma}\label{l:isoz}
There exists a $*$-isomorphism $\ftn{ \Phi }{ \mathcal{Z} }{ \mathcal{Z} \otimes \mathcal{Z} }$ and there exists a sequence of unitaries $\{ u_{n} \}_{ n = 1}^{ \infty }$ in $U( \mathcal{Z} \otimes \mathcal{Z} )_{0}$ such that for all $a \in \mathcal{Z}$,
\begin{align*}
\lim_{ n \to \infty } \| \Phi ( a ) - u_{n} ( a \otimes 1_{ \mathcal{Z} } ) 
u_{n}^{*} \| = 0. 
\end{align*}
\end{lemma}

\begin{proof}
The result follows from Theorems 7.6 and 8.7 of \cite{JiangSu}.
\end{proof}

\begin{lemma}  Let $\mathfrak{B}$ be a simple, separable, 
unital $C^{*}$-algebra and let $G$ be a closed normal subgroup of 
$U(\mathfrak{B} \otimes \mathcal{Z} \otimes \mathcal{Z} )_0$ that 
properly contains $\mathbb{T}$.  
Then $G$ contains 
$CU(1_{\mathfrak{B} \otimes \mathcal{Z} } \otimes \mathcal{Z})_0$.   
\label{lem:containU(Z)}
\end{lemma}

\begin{proof}
By Theorem 3 of \cite{GongJiangSu},  
$\mathfrak{B} \otimes \mathcal{Z} \otimes \mathcal{Z}$ 
($\cong \mathfrak{B} \otimes \mathcal{Z}$) is either
purely infinite or stably finite. 
If $\mathfrak{B} \otimes \mathcal{Z} \otimes \mathcal{Z}$
is purely 
infinite, then by Theorem 2.4 of \cite{EllRorAuto},   
$G = U(\mathfrak{B} \otimes \mathcal{Z} \otimes \mathcal{Z} )_0$; and thus,
$G$ contains $CU(1_{\mathfrak{B} \otimes \mathcal{Z} } \otimes \mathcal{Z})_0$.  Hence, we may assume that 
$\mathfrak{B} \otimes \mathcal{Z} \otimes \mathcal{Z}$
(and hence $\mathfrak{B}$; see Lemma 3.3 of \cite{GongJiangSu}) 
is stably finite.

Set $\mathfrak{A} = \mathfrak{B} \otimes \mathcal{Z}$.  Let $u$ be an element of $G \setminus \mathbb{T}$ and $\ftn{ \Phi }{ \mathcal{Z} }{ \mathcal{Z} \otimes \mathcal{Z} }$ be the $*$-isomorphism given in Lemma \ref{l:isoz}.  Then by Lemma \ref{l:isoz}, there exist $w \in \mathfrak{A}$ and a sequence of unitaries $\{ u_{n} \}_{ n = 1}^{ \infty } \subseteq U( \mathfrak{A} \otimes \mathcal{Z} )_{0}$ such that $(\id_{ \mathfrak{B} } \otimes \Phi )( w ) = u$ and 
\begin{align*}
\lim_{ n \to \infty } \| u_{n}^{*} ( \id_{ \mathfrak{B} } \otimes 
\Phi )( w ) u_{n} - w \otimes 1_{ \mathcal{Z} } \| = 0. 
\end{align*}
Since $\id_{ \mathfrak{B} } \otimes \Phi$ is an isomorphism and since $u \in U( \mathfrak{A} \otimes \mathcal{Z} )_{0} \setminus \mathbb{T}$, $w \in U( \mathfrak{A} )_{0} \setminus \mathbb{T}$.  Since $G$ is a closed normal subgroup of $U( \mathfrak{A} \otimes \mathcal{Z} )_{0}$, $u_{n}^{*} u u_{n} = u_{n}^{*} ( \id_{ \mathfrak{B} } \otimes \Phi )( w ) u_{n} \in G$ which implies that $w \otimes 1_{ \mathcal{Z} } \in G$.  Hence, $G$ contains a unitary of the form $x = w \otimes 1_{\mathcal{Z}}$ where $w \in U(\mathfrak{A})_0 \setminus \mathbb{T}$.
   
   Let $\p$, $\q$ be relatively prime supernatural numbers of infinite
type.
By Theorem 3.4 of \cite{RordamWinter}, $\mathcal{Z}$ is a C*-inductive limit
$\mathcal{Z} = \overline{\bigcup_{n=1}^{\infty} \mathcal{Z}_n}$
where $\mathcal{Z}_n \cong \mathcal{Z}_{\p, \q}$ for all $n \geq 1$ and where the connecting
maps are unital and injective.
Consider an arbitrary building block $\mathcal{Z}_N$.
We will prove that $G$ contains a nonscalar unitary in 
$CU(1_{\mathfrak{A}} \otimes
\mathcal{Z}_N) = CU(1_{\mathfrak{A}} \otimes \mathcal{Z}_N)_0$.

Now $x \in U(\mathfrak{A} \otimes \mathcal{Z}_N)_0$; in particular, 
$x = w \otimes 1_{\mathcal{Z}_N} \in \mathfrak{A} \otimes \mathcal{Z}_N \cong \mathfrak{A} \otimes \mathcal{Z}_{\p, \q}$.  Since $\mathbb{T} \subseteq G$, multiplying $x$ by a scalar if necessary, we may assume that $1$ is in the spectrum of $w$.  Hence, $1$ will be in the spectrum of $x$.  Since $w$ is not in $\mathbb{T}$, the spectrum $w$ and $x$ contains a point other than $1$.

Case 1:   Suppose that the spectrum of $w$ contains a point
$\alpha \neq -1, 1, i, -i$.  To proceed, recall the following matrix computation:
\[
\left[
\begin{array}{cc}
0 & 1 \\
1 & 0 
\end{array}
\right]
=
\left[
\begin{array}{cc}
\frac{1}{\sqrt{2}} & \frac{1}{\sqrt{2}} \\
\frac{1}{\sqrt{2} } & -\frac{1}{\sqrt{2}}
\end{array}
\right]
\left[
\begin{array}{cc}
1 & 0 \\
0 & -1
\end{array}
\right]
\left[
\begin{array}{cc}
\frac{1}{\sqrt{2}} & \frac{1}{\sqrt{2}} \\
\frac{1}{\sqrt{2}} & -\frac{1}{\sqrt{2}}
\end{array}
\right].  
\]
Let $\{ u(t) \}_{t \in [0,1] }$ be the continuous path of 
unitaries in $\mathsf{M}_{2}$ given by 
\begin{equation}
u(t) = 
\left[
\begin{array}{cc}
\frac{1}{\sqrt{2}} & \frac{1}{\sqrt{2}} \\
\frac{1}{\sqrt{2}} & -\frac{1}{\sqrt{2}}
\end{array}
\right]
\left[  
\begin{array}{cc}
1 & 0 \\
0 & \exp(i \pi t)    
\end{array}
\right]
\left[
\begin{array}{cc}
\frac{1}{\sqrt{2}} & \frac{1}{\sqrt{2}} \\
\frac{1}{\sqrt{2}} & -\frac{1}{\sqrt{2}}
\end{array}
\right]
\label{equ:p_1p_2homotopy1} 
\end{equation}
for all $t \in [0,1]$.
Hence, 
$u(0) = 1_{\mathsf{M}_{2}}$ and
$u(1) = \left[ \begin{array}{cc} 0 & 1 \\ 1 & 0 \end{array} \right]$.
Also, for 
$\beta, \gamma \in \mathbb{T}$, 
if $\{ v(t) \}_{t \in [0,1] }$ is the
continuous path of unitaries in $\mathsf{M}_{2}$ that is given by 
\begin{equation}
v(t) = u(t) \left[ \begin{array}{cc} \beta & 0 \\ 0 & \gamma 
\end{array} \right] u(t)^* \left[ \begin{array}{cc} \overline{\beta}
 & 0 \\ 0 & \overline{\gamma} 
\end{array} \right]   
\label{equ:p_1p_2homotopy2} 
\end{equation}
for all $t \in [0,1]$, 
then $v(0) = 1_{\mathsf{M}_{2}}$ and 
$v(1) = \left[ \begin{array}{cc} \overline{\beta} \gamma & 0 \\
0 & \beta \overline{\gamma} \end{array} \right]$. 
Moreover, by a direct computation, we get that 
if $\beta \neq \pm \gamma$ then 
for all $t \in (0,1)$,  the eigenvalues of $v(t)$ are
distinct and are complex conjugates of each other;  also, for $0 < s, t < 1$
with $s \neq t$,  the set of eigenvalues of $v(s)$ is
different from the set of eigenvalues of $v(t)$.    

Let $h_1, h_2 : [0,1] \rightarrow \mathbb{T}$
be the unique continuous functions such that  
$h_1(0) = h_2(0) = 1, 
h_1(1) = \alpha^2, h_2(1) = \overline{\alpha}^2$, 
and 
$\{ h_1, h_2 \}$ are the eigenfunctions of the continuous path of 
unitaries 
$\{ u(t) \mathrm{diag}(\alpha, \overline{\alpha}) u(t)^* 
\mathrm{diag}(\overline{\alpha}, \alpha) \}_{t \in [0,1]}$, i.e., for all $t \in [0,1]$, $\{ h_1(t), h_2(t) \}$ is the set of
(distinct when $0 < t < 1$) eigenvalues of $u(t) \mathrm{diag}(\alpha, \overline{\alpha}) u(t)^* 
\mathrm{diag}(\overline{\alpha}, \alpha)$.  Define two continuous functions $g_1, g_2 : [0,1] \rightarrow \mathbb{R}$ by
\[
g_1(t) =
\begin{cases}
\alpha^2 & \makebox{ if } t \in \left[\frac{2}{5}, \frac{3}{5}\right] \\
1  & \makebox{ if } t \in \left[0, \frac{1}{5}\right] \cup \left[\frac{4}{5}, 1\right] \\
h_1(5t - 1) &  \makebox{ if } t \in \left[\frac{1}{5}, \frac{2}{5}\right] \\
h_1(-5t + 4) & \makebox{ if } t \in \left[\frac{3}{5}, \frac{4}{5}\right]  
\end{cases}
\] 

\[
g_2(t) =
\begin{cases}
\overline{\alpha}^2 & \makebox{ if } t \in \left[\frac{2}{5}, \frac{3}{5}\right] \\
1  & \makebox{ if } t \in \left[0, \frac{1}{5}\right] \cup \left[\frac{4}{5}, 1\right] \\
h_2(5t - 1) &  \makebox{ if } t \in \left[\frac{1}{5}, \frac{2}{5}\right] \\
h_2(-5t + 4) & \makebox{ if } t \in \left[\frac{3}{5}, \frac{4}{5}\right].  
\end{cases}
\]

Claim 1:  For every $\epsilon > 0$,
there exist pairwise orthogonal nonzero projections
$r_1, r_2 \in C[0,1] \otimes \mathfrak{A}  \otimes \mathsf{M}_{\p} \otimes \mathsf{M}_{\q}$
such that if $p_1, p_2$ are nonzero projections 
in $r_1 (C[0,1] \otimes \mathfrak{A} \otimes 
\mathsf{M}_{\p} \otimes \mathsf{M}_{\q}) r_1, 
r_2 (C[0,1] \otimes \mathfrak{A} \otimes \mathsf{M}_{\p} \otimes \mathsf{M}_{\q}) r_2$ respectively
with $p_1 \sim p_2$, then there exists a unitary $w' \in G$ such that the
following hold:
\begin{enumerate}
\item[(a)] $w' (1 - (p_1 +p_2)) = (1 - (p_1 + p_2)) w' = 1 - (p_1 + p_2)$   
in $C[0,1] \otimes \A \otimes \M_{\p} \otimes \M_{\q}$ and 
\item[(b)] $\| w'(t) - (g_1(t) p_1 (t) + g_2(t) p_2(t) + 
1 - (p_1(t) + p_2(t)) ) \| < \epsilon$ for all $t \in [0,1]$.    
\end{enumerate}
    
  Towards proving Claim 1, let $\epsilon > 0$ be given.   
For simplicity, we may assume that $\epsilon < 1/2$.  Plug $\min \left\{ \frac{ |1-\alpha| }{ 2} , \frac{ |1 - \overline{\alpha}| }{ 2 }, 
\frac{ |\alpha - \overline{\alpha}| }{ 2} , \frac{ \epsilon }{ 100 } \right\}$
into Lemma \ref{lem:FinSpectUnitAppr} to get a positive real number
$\delta'$.  Let $\delta = \min \left\{ \frac{ \epsilon }{ 1000 }, \frac{ \delta' }{ 1000 }, \frac{ |1 - \alpha|}{1000},
\frac{ |1 - \overline{\alpha}| }{ 1000 }, \frac{ |\alpha - \overline{\alpha} | }{ 1000 } \right\}$.
Contracting $\delta > 0$ if necessary, we may assume that for all 
$\gamma_1, \gamma_2 \in \mathbb{T}$, if $| \gamma_1 - \gamma_2 | < 100 \delta$,
then $| \gamma_1^2 - \gamma_2^2 | < \frac{ \epsilon }{ 100 }$.   
By Lemma \ref{lem:AotimesUHFProperties}, there exist nonzero projections
$r'_1, s'_1$ in 
$\mathfrak{A} \otimes \mathsf{M}_{\p} \otimes \mathsf{M}_{\q}$ such that the following hold:
\begin{itemize}
\item[(1)] $r'_1$ is contained in the hereditary sub-$C^{*}$-algebra of 
$\mathfrak{A} \otimes \mathsf{M}_{\p} \otimes \mathsf{M}_{\q}$ generated
by $f(w \otimes 1_{\mathsf{M}_{\p} \otimes \mathsf{M}_{\q}})$, 
where $f$ is a continuous real-valued function on 
$\mathbb{C}$ with $0 \leq f \leq 1$, $f(1) = 1$ and
$f$ vanishes outside of a small precompact open neighbourhood $O_1$ of $1$. 
\item[(2)] $s'_1$ is contained in the hereditary sub-$C^{*}$-algebra of 
$\mathfrak{A} \otimes \mathsf{M}_{\p} \otimes \mathsf{M}_{\q}$ that is generated by 
$g( w \otimes 1_{\mathsf{M}_{\p} \otimes \mathsf{M}_{\q}})$, where $g$  
is a continuous real-valued function on $\mathbb{C}$ with $0 \leq g \leq 1$,
$g(\alpha) = 1$ and $g$ vanishes outside of a small precompact neighbourhood
$O_2$ of $\alpha$.
\item[(3)] $\overline{O_1} \cap \overline{O_2} = \emptyset$.  
\item[(4)] $r'_1 \sim s'_1$ in 
$\mathfrak{A} \otimes \mathsf{M}_{\p} \otimes \mathsf{M}_{\q}$.
\item[(5)] For all $s , t \in \{ r'_1 (w \otimes 1_{\mathsf{M}_{\p} \otimes \mathsf{M}_{\q}}), 
(w \otimes 1_{\mathsf{M}_{\p} \otimes \mathsf{M}_{\q}}) r'_1, r'_1 (w \otimes 
1_{\mathsf{M}_{\p} \otimes \mathsf{M}_{\q}})  r'_1, r'_1 \}$, 
\begin{equation*}
\| s - t \| < \delta. 
\end{equation*} 
\item[(6)] For all $s, t \in \{ s'_1 (w \otimes 1_{\mathsf{M}_{\p} \otimes \mathsf{M}_{\q}}), 
(w \otimes 1_{\mathsf{M}_{\p} \otimes \mathsf{M}_{\q}}) s'_1, s'_1 (w \otimes 
1_{\mathsf{M}_{\p} \otimes \mathsf{M}_{\q}})  s'_1, \alpha s'_1 \}$, 
\begin{equation*}
\| s - t \| < \delta. 
\end{equation*}
\end{itemize}
   
Define $r_1, s_1 \in C[0,1] \otimes \mathfrak{A} \otimes \mathsf{M}_{\p} \otimes 
\mathsf{M}_{\q}$ by $r_1 = r'_1 \otimes 1_{C[0,1]}$ and $s_1 = s'_1 \otimes 1_{C[0,1]}$.
Let $p_1, q_1$ be projections in $\mathfrak{A} \otimes \mathsf{M}_{\p} \otimes \mathsf{M}_{\q}$
that satisfy the hypotheses of the statement of 
Claim 1.  
From the definition of $r_1, s_1$ and the properties of 
$r'_1, s'_1$, we know that the following must hold for $p_1, q_1$:
\begin{itemize}
\item[(7)] $p_1$ is contained in the hereditary sub-$C^{*}$-algebra of 
$C[0,1] \otimes \mathfrak{A} \otimes \mathsf{M}_{\p} \otimes \mathsf{M}_{\q}$ generated
by $f(x)$, and $q_1$ is contained in the hereditary sub-$C^{*}$-algebra of
$C[0,1] \otimes \mathfrak{A} \otimes \mathsf{M}_{\p} \otimes \mathsf{M}_{\q}$ generated 
by $g(x)$, where $f, g$ are the functions in the definition
of $r'_1, s'_1$.   
\item[(8)] For all $s,t \in \{ p_1 x, x p_1, p_1 x p_1 , p_1 \}$,
\begin{equation*}
\| s - t \| < \delta. 
\end{equation*} 
\item[(9)] For all $s, t \in \{ q_1 x, x q_1, q_1 x q_1, \alpha q_1 \}$,  
\begin{equation*}
\| s -t \| < \delta.   
\end{equation*}
\end{itemize}
Since $p_1$ and $q_1$ are orthogonal projections and  since $p_1 \sim q_1$,
there exists $v_1 \in U(C[0,1] \otimes \mathfrak{A} \otimes \mathsf{M}_{\p} \otimes \mathsf{M}_{\q})_0$ such that 
$v_1 p_1 v_1^* = v_1^* p_1 v_1 = q_1$ and 
$v_1 ( 1 - (p_1 + q_1)) = ( 1 - (p_1 + q_1)) v_1 = 
1 - (p_1 + q_1)$.   We may assume that $p_1 v_1 = v_1^* q_1$. 
   
By (7)--(9) and the definition of $v_1$, we get $\| x^* v_1 x v_1^* p_1 - \alpha p_1 \|  < 2 \delta$, $\| x^* v_1 x v_1^* q_1 - \overline{\alpha} q_1 \|   <  2 \delta$, and
\begin{align*}
  \| x^* v_1 x v_1^* ( 1- (p_1 + q_1)) - 
( 1- (p_1 + q_1)) \| <  6 \delta.   
\end{align*}
The above inequalities implies that  
\begin{equation}
\| x^* v_1 x v_1^*  - (\alpha p_1 + \overline{\alpha} q_1 
+  1- (p_1 + q_1)) \| < 10 \delta.   
\label{equ:troll}
\end{equation}
Recall that if $a, b$ are elements in a unital $C^{*}$-algebra such 
that $a$ is invertible and $\| a - b \| < \frac{ 1 }{  \| a^{-1} \| }$ 
then $b$ is also invertible in the $C^{*}$-algebra.  From this, 
(\ref{equ:troll}) and the definition of $\delta$,   
it follows that the spectrum of $x^* v_1 x v_1^*$ is contained in 
three pairwise disjoint open balls with centres $1, \alpha$ and 
$\overline{\alpha}$.  Since the spectrum is a compact set, we may assume
that the closures of the three open balls are also pairwise disjoint.
In particular, we can take each open ball to have radius
$10 \delta$.
Hence, there exist pairwise disjoint self-adjoint
partial isometries $x_1, y_1, z_1 
\in C[0,1] \otimes \mathfrak{A} \otimes \mathsf{M}_{\p} \otimes \mathsf{M}_{\q}$ 
and there exist pairwise disjoint projections $c_1, d_1, e_1 \in 
C[0,1] \otimes \mathfrak{A} \otimes \mathsf{M}_{\p} \otimes \mathsf{M}_{\q}$ such that the following
hold:
\begin{itemize}
\item[(10)] $x^* v_1 x v_1^* = x_1 + y_1 + z_1$.     
\item[(11)] $x_1, y_1, z_1$ are elements of the open balls about
$\alpha, \overline{\alpha}, 1$ (respectively), all with
radius $10 \delta$.  (Of course, we are really applying the continuous
functional calculus to $x^* v_1 x v_1^*$.)   
\item[(12)] $x_1^{*} x_1 = x_1 x_1^* = c_1$, $y_1^{*} y_1 = y_1 y_1^{*} = d_1$ and
$z_1^* z_1 = z_1 z_1^* = e_1$.   
\item[(13)] $c_1 + d_1 + e_1 = 1_{C[0,1] \otimes \mathfrak{A} \otimes \mathsf{M}_{\p}  \otimes   
\mathsf{M}_{\q}}$.  
\end{itemize}
By (11) and (12), we have that
\begin{equation}
\| x^* v_1 x v_1^* - ( \alpha c_1 + \overline{\alpha} d_1 + e_1 ) \|
< 10 \delta. 
\label{equ:Roberts}
\end{equation}
Together with (\ref{equ:troll}),  
we get  
\begin{equation} 
\| \alpha p_1 + \overline{\alpha} q_1 + 1 - (p_1 + q_1)
- ( \alpha c_1 + \overline{\alpha} d_1 + e_1 ) \| < 20 \delta.  
\label{equ:Niu}
\end{equation}
Hence, by (\ref{equ:Niu}), the definition of $\delta$ and 
Lemma \ref{lem:FinSpectUnitAppr},   
we have that $c_1, d_1, e_1$ is Murray-von Neumann equivalent and
close to $p_1, q_1, 1 - (p_1 + q_1)$ respectively.   

Let $w_1 \in U(C[0,1] \otimes \mathfrak{A} \otimes \mathsf{M}_{\p} \otimes \mathsf{M}_{\q})_0$ be a 
unitary such that 
$w_1 c_1 w_1^* = p_1$, $w_1 d_1 w_1^* = q_1$ and 
$w_1 e_1 w_1^* = 1 - (p_1 + q_1)$.  We can choose $w_1$ to be close to $1$.  Then, by   (\ref{equ:Roberts}),  
\begin{equation}
\| w_1 x^* v_1 x v_1^* w_1^* - (\alpha p_1 + \overline{\alpha} q_1 + 
1 - (p_1 + q_1)) \| < 10 \delta. 
\label{equ:mitacshit}
\end{equation}
Let $x_1$ be the unitary in $U(C[0,1] \otimes \mathfrak{A} \otimes \mathsf{M}_{\p} 
\otimes \mathsf{M}_{\q})_0$  that is given by 
$$x_1 = (w_1 x^* v_1 x v_1^* w_1^*) v_1 (w_1 x^* v_1 x v_1^* w_1^*)^* 
v_1^*.$$ 
By (10)--(13), (\ref{equ:mitacshit}), and the definitions
of $v_1$, $w_1$, $x_1$ and $\delta$, the following
hold:
\begin{itemize}
\item[(14)] $x_1 p_1 = p_1 x_1 = p_1 x_1 p_1$ and 
$\| x_{1} p_{1} - \alpha^{2} p_{1} \| < 20 \delta$. 
   
\item[(15)] $x_1 q_1 = q_1 x_1 = q_1 x_1 q_1$ 
and $\| x_{1} q_{1} - \overline{ \alpha }^{2} q_{1} \|  < 20 \delta$. 
 
\item[(16)] $x_1 (1 - (p_1 + q_1)) = (1 - (p_1 + q_1)) x_1 = 
1 - (p_1 + q_1)$.   
\end{itemize}

We now define elements $v_2, w_2, v_3 \in U(\mathfrak{A} \otimes \mathcal{Z}_{\p, \q})_0$.  To simplify notation, assume that      
$p_1, q_1, v_1, w_1 \in 1_{C[0,1]} \otimes \mathfrak{A} 
\otimes \mathsf{M}_{\p} \otimes
\mathsf{M}_{\q}$.  The general case is similar.  
To define $v_2$, we first consider the norm-continuous path of 
unitaries $\{ u(t) \}_{t \in [0,1]}$ in $(p_1 + q_1)(\mathfrak{A} \otimes \mathsf{M}_{\p} \otimes \mathsf{M}_{\q}) (p_1 + q_1)$ given by (\ref{equ:p_1p_2homotopy1}) with the standard system of matrix units for $\mathsf{M}_2$ as 
$e_{1,1} = p_1$, $e_{2,2} = q_1$, $e_{1,2} = p_1 v_1 = v_1^* q_1$.  The last equality follows from the definition of $v_1$.  Define $v_{2}$ by
\[
v_2(t) = 
\begin{cases} 
v_1  & \makebox{ if } t \in \left[\frac{ 2 }{ 5 }, \frac{3}{ 5 } \right] \\
1    & \makebox{ if } t \in \left[0, \frac{1}{5} \right] \cup \left[\frac{4}{5}, 1\right] \\
u(5t - 1) + 1 - (p_1 + q_1) & \makebox{ if } t \in \left[\frac{1}{5}, \frac{2}{5} \right] \\
u(-5t + 4) + 1 - (p_1 + q_1) & \makebox{ if } t \in \left[\frac{3}{5}, 
\frac{4}{5} \right].  
\end{cases}
\] 
Note that $v_2 \in U(\mathfrak{A} \otimes \mathcal{Z})_0$.  Let $w_2, v_3 \in U(\mathfrak{A} \otimes \mathcal{Z}_{\p, \q})_0$ be unitaries
chosen so that 
$w_2(t) = w_1$ for $t \in \left[\frac{1}{5}, \frac{4}{5} \right]$, $w_2(0) = w_2(1) = 1$, 
$v_3(t) = v_1$ for $t \in \left[\frac{1}{5}, \frac{4}{5}\right]$, and
$v_3(0) = v_3(1) = 1$.  

Since $x \in G$ and $G$ is a normal subgroup of $U(\mathfrak{A} \otimes \mathcal{Z}_{\p,
\q} )_{0}$,  
$$w_3 = (w_2 x^* v_3 x v_3^* w_2^*) v_2 
(w_2 x^* v_3 x v_3^* w_2^*)^* v_2^*$$  
is an element of $G$.  Hence, it follows, by the definitions above, that the
following hold: 
\begin{itemize}
\item[(17)] $w_3 (1 - (p_1 + q_1)) = (1 - (p_1 + q_1)) w_3 
= 1 - (p_1 + q_1)$. 
\item[(18)] $\| w_3 - (\alpha p_1 + \overline{\alpha} q_1   + 1 - (p_1 + 
q_1)) v_2 (\alpha p_1 + \overline{\alpha} q_1   + 1 - (p_1 + 
q_1))^* v_2^* \| < 20 \delta$. 
\item[(19)] The eigenvalues of $(\alpha p_1 + \overline{\alpha} q_1   
+ 1 - (p_1 + 
q_1)) v_2(t) (\alpha p_1 + \overline{\alpha} q_1   + 1 - (p_1 + 
q_1))^* v_2^*(t)$ are given by $\{ g_1(t), g_2(t) \}$ for all $t \in [0,1]$. 
\end{itemize}
Since $g_1(t) \neq g_2(t)$ for $t \in \left[ \frac{1}{5}, \frac{4}{5} \right]$ and 
$g_1(t) = g_2(t) = 1$ for $t \in \left[ 0,\frac{1}{5} \right] \cup \left[ \frac{4}{5}, 1 \right]$, 
there exist two projections $p_2, q_2 \in C[0,1] \otimes \mathfrak{A} \otimes
\mathsf{M}_{\p} \otimes \mathsf{M}_{\q}$ such that 
$$(\alpha p_1 + \overline{\alpha} q_1   
+ 1 - (p_1 + 
q_1)) v_2 (\alpha p_1 + \overline{\alpha} q_1   + 1 - (p_1 + 
q_1))^* v_2^* = g_1 p_2 + g_2 q_2 + 1 - (p_1 + q_1).$$
Indeed, $p_2, q_2$ can be chosen to be
 projections inside the copy of $\mathsf{M}_2$ that 
is generated by $\{ p_1, q_1, v_1 p_1 \}$; and they can be chosen to satisfy
$p_1 \sim p_2$, $q_1 \sim q_2$.  Hence, there exists a unitary
$v_4 \in U( \mathfrak{A} \otimes \mathcal{Z}_{\p, \q} )_{0}$ such that 
$$v_4 (\alpha p_1 + \overline{\alpha} q_1   
+ 1 - (p_1 + 
q_1)) v_2 (\alpha p_1 + \overline{\alpha} q_1   + 1 - (p_1 + 
q_1))^* v_2^*v_4^* = g_1 p_1 + g_2 q_1 + 1 - (p_1 + q_1)$$
and 
$$v_4 (1 - (p_1 + q_1)) = (1 - (p_1 + q_1)) v_4 = 
1 - (p_1 + q_1).$$

Set $w' = v_4 w_3 v_4^*$.  Then $w' \in G$ is a unitary that satisfies the statement of Claim 1.
This completes the proof of Claim 1. 

Claim 2:  Suppose that $r, s$ are pairwise orthogonal projections
in $C[0,1] \otimes \mathfrak{A} \otimes \mathsf{M}_{\p} \otimes \mathsf{M}_{\q}$ which
are Murray-von Neumann equivalent in $C[0,1] \otimes \mathfrak{A} \otimes \mathsf{M}_{\p}
\otimes \mathsf{M}_{\q}$.  Then 
$$g_1 r + g_2 s + 1 - (r + s) \in G.$$

    We will show that $g_1 r + g_2 s + 1 - (r + s)$ can be approximated
arbitrarily close by elements of $G$.  Let $\epsilon > 0$ be given.  Plug this $\epsilon$ into Claim 1 to
get nonzero orthogonal projections $r_1, r_2 \in C[0,1] \otimes \mathfrak{A} \otimes 
\mathsf{M}_{\p} \otimes \mathsf{M}_{\q}$.  Decompose $r, s$ into pairwise orthogonal projections 
$$r = r_{1, 1} + r_{1, 2} + .... + r_{1, n}$$
and 
$$s = s_{1, 1} + s_{1, 2} + .... + s_{1, n}$$ 
so that the following hold: 
\begin{itemize}
\item[(i)] $r_{1,j}$ and $s_{1,j}$ 
are projections in $C[0,1] \otimes \mathfrak{A} \otimes 
\mathsf{M}_{\p} \otimes \mathsf{M}_{\q}$ for all $j \geq 1$.   
\item[(ii)] $r_{1,j} \sim s_{1,j}$ in $C[0,1] \otimes \mathfrak{A} \otimes 
\mathsf{M}_{\p} \otimes \mathsf{M}_{\q}$ for all $j \geq 1$.  
\item[(iii)] $r_{1,j} \precsim r_1 \sim r_2$ 
in $C[0,1] \otimes \mathfrak{A} \otimes 
\mathsf{M}_{\p} \otimes \mathsf{M}_{\q}$
for all $j \geq 1$.   
\end{itemize}
Since $C[0,1] \otimes \mathfrak{A} \otimes 
\mathsf{M}_{\p} \otimes \mathsf{M}_{\q}$ has cancellation of projections and since
$G$ is a normal subgroup of $U(\mathfrak{A} \otimes \mathcal{Z})_0$, it follows, by
Claim 1, that for $j \geq 1$, there exists 
a unitary $w_{1,j} \in G$ such that the following conditions
hold:
\begin{itemize}
\item[(iv)] $w_{1,j} ( 1 - (r_{1,j} + s_{1,j})) = 
( 1 - (r_{1,j} + s_{1,j})) w_{1,j} = 
1 - (r_{1,j} + s_{1,j})$  for all $j \geq 1$ and
\item[(v)] $\| w_{1,j} - (g_1 r_{1,j} + g_2 s_{1,j} + 1 - (r_{1,j} + 
s_{1,j}) ) \| < \epsilon$ for all $j \geq 1$.  
\end{itemize}
Set $w_4 = w_{1,1} w_{1,2} ... w_{1,n}$.
Then $$\| w_4 - (g_1 r + g_2 s + 1 - (r + s)) \| < \epsilon.$$
Since $\epsilon > 0$ is arbitrary and since $G$ is closed,
$$g_1 r + g_2 s + 1 - (r + s) \in G$$ 
as required.
This completes the proof of Claim 2.

    We now complete the proof that 
$CU(1_{\mathfrak{A}} \otimes \mathcal{Z})_0 \subseteq G$ for Case 1.
Choose two nonzero orthogonal
projections $r, s \in C[0,1] \otimes 1_{\mathfrak{A}} 
\otimes \mathsf{M}_{\p} \otimes \mathsf{M}_{\q}$ such that
$r \sim s$ in $C[0,1] \otimes 1_{\mathfrak{A}}
\otimes \mathsf{M}_{\p} \otimes \mathsf{M}_{\q}$.  
By Claim 2, 
$w_5 = g_1 r + g_2 s + 1 - (r + s) \in G$.
Note that $w_5$ is also a unitary in $CU(1_{\mathfrak{A}} \otimes \mathcal{Z})$ which is not
a scalar multiple of the identity. 
By \cite{NgZ},  $CU(\mathcal{Z})/ \mathbb{T}$ is a simple topological group.
Hence, $CU(1_{\mathfrak{A}} \otimes \mathcal{Z}) \subseteq G$ as required.  

Case 2:  Suppose that the spectrum of $w$ (and hence the spectrum of $x$) is contained
in $\{ -1, 1, i, -i \}$.  Recall that $w$ is not in $\mathbb{T}$ and its spectrum contains $1$.
For simplicity, let us assume that the spectrum of $w$ is 
$\{ 1, \alpha \}$ where $\alpha \in \{ -1, i, -i \}$.  (The proofs
for the other cases are similar.)  Hence, there exist nonzero orthogonal projections $p, q \in \mathfrak{A}$
such that $w = p + \alpha q$.
Therefore, $x =  p \otimes 1_{\mathcal{Z}_{\p, \q}}
+ \alpha q \otimes 1_{\mathcal{Z}_{\p, \q}}$.   

Let $r, s \in  1_{\mathfrak{A}} \otimes 
\mathsf{M}_{\p} \otimes \mathsf{M}_{\q}$  
be nonzero orthogonal projections such that 
$r$ is orthogonal to $p$, $s$ is orthogonal to $q$, 
$r \sim s \precnsim p \otimes
1_{\mathsf{M}_{\p} \otimes \mathsf{M}_{\q}}$ and 
$s \precnsim q \otimes 1_{\mathsf{M}_{\p} \otimes \mathsf{M}_{\q}}$ 
in $\mathfrak{A} \otimes
\mathsf{M}_{\p} \otimes \mathsf{M}_{\q}$.   
Then there exist projections $r', s' \in \mathfrak{A} \otimes \mathsf{M}_{\p} 
\otimes \mathsf{M}_{\q}$ and $v \in \mathfrak{A} \otimes \mathsf{M}_{\p} 
\otimes \mathsf{M}_{\q}$ such that 
$r \sim r' \lneq p \otimes 1_{\M_{\p} \otimes \M_{\q}}$, $s \sim s' \lneq q \otimes 1_{\M_{\p} \otimes \M_{\q}}$, and $v^* v = r'$ and $v v^* = s'$.  

Let $\{ u(t) \}_{t \in [0,1]}$ be the norm-continuous
path of unitaries in $\mathsf{M}_2$ as in (\ref{equ:p_1p_2homotopy1})
except that the canonical system of matrix units is taken to be
$e_{1,1} = r'$, $e_{2,2} = s'$ and $e_{2, 1} = v$.  Define the norm-continuous path of unitaries 
$\{ w(t) \}_{t \in [0,1] }$ in $U(\mathfrak{A} \otimes \M_{\p} \otimes \M_{\q})$ by 
\[
w(t) = 
\begin{cases}
u(2t) + 1 - (r' + s') \otimes 1_{C[0,1]} & \makebox{ if } t \in \left[0,\frac{1}{2} \right] \\
u(-2t + 2) + 1 - (r' + s') \otimes 1_{C[0,1]} & \makebox{ if } t \in 
\left[ \frac{1}{2}, 1 \right]. 
\end{cases}
\]
Note that $w \in U(\A \otimes \ZC_{\p, \q})_0$.  
Since $G$ is a closed normal subgroup of $U(\mathfrak{A}
 \otimes \mathcal{Z})_0$,
$w' = w x w^* x^*$ is an element of $G \cap U(\A \otimes \ZC_{\p, \q})_0$.
Moreover, we have that $w'$ is not a scalar multiple of the identity,  
$w' (1 - (r' + s') \otimes 1_{C[0,1]}) = 
(1 - (r' + s') \otimes 1_{C[0,1]}) w' = 1 - (r' + s') \otimes
1_{C[0,1]}$ in $C[0,1] \otimes \A \otimes \M_{\p} \otimes \M_{\q}$
and $w'(0) = w'(1) = 1$.

   Let $f_3, g_3 : [0,1] \rightarrow \mathbb{T}$ be two continuous 
functions such that 
for all $t \in [0,1]$, $\{ f_3(t), g_3(t) \}$ are the eigenvalues
of $(r' + s') w'(t) (r' + s')$.
Note that $f_3(0) = f_3(1) = g_3(0) = g_3(1) = 1$.
Set $w'' = f_3 r' + g_3 s' + 1 - (r' + s') \otimes 1_{C[0,1]}$.
Then $w'' \in U(\A \otimes \ZC_{\p, \q})$.  Note that $f_3 r' + g_3 s' \in C[0,1] \otimes C^*(r', s', v) \cong
C[0,1] \otimes \M_2$.  Hence, by \cite{ThomsenCircleAlgebras},  
$w'$ and $w''$ are approximately unitarily equivalent in
$C[0,1] \otimes \A \otimes \M_{\p} \otimes \M_{\q}$.
Moreover, since $w'(0) = w'(1) = w''(0) = w''(1) = 1$, 
we can choose the implementing unitaries to also have unit value at the
endpoints and be elements of
$U(\A \otimes \mathcal{Z}_{\p, \q})_0$.  
Hence, since $G$ is a closed normal subgroup of $U(\A \otimes \ZC)_0$,
$w'' \in G$.  

Since  
$f_3(t) = f_3(1 - t)$, $g_3(t) = g_3(1-t)$ for all $t \in [0,1]$,   
$w''' = f_3 r  + g_3 s  + 1 - 
(r + s) \otimes 1_{C[0,1]} \in U(1_{\A} \otimes \ZC_{\p, \q})_0$.
Note also that $w'''$ is not a scalar multiple of the identity.  
Since $r \sim r'$, $s \sim s'$, $r$ is orthogonal to $r'$, $s$ is orthogonal to $s'$, and
$w'''(0) = w''(0) = w'''(1) = w''(1) = 1$,   
we have that $w''$ and $w'''$ are approximately unitarily 
where we can choose the unitaries to have unit value at endpoints and be elements in $U(\A \otimes \ZC)_0$.
Hence, since $G$ is a closed normal subgroup of $U(\A \otimes \ZC)_0$,
we must have that $w''' \in G$, i.e., $G$ contains an element 
of $U(1_{\A} \otimes \ZC)_0$ which is not a scalar multiple of the
identity.  By \cite{NgZ}, $U(1_{\A} \otimes \ZC)_0/\mathbb{T}$ is 
a simple topological group.  Hence, 
$U(1_{\A} \otimes \ZC)_0 \subseteq G$ as required.  This completes the
proof for Case 2 and hence, all the cases.     
\end{proof}

Recall that for a $C^*$-algebra $\mathfrak{C}$ and for elements
$a, b \in \mathfrak{C}$,
$(a,b)$ is defined to be $(a,b) = a b a^* b^*$.

\begin{lemma}
Consider the supernatural numbers $\p = 2^{\infty}$ and $\q = 3^{ \infty }$.  Let 
$\mathfrak{A}$ be a simple unital $C^{*}$-algebra.  Let $G$ be a closed normal 
subgroup of $U(\mathfrak{A} \otimes \mathcal{Z}_{\p, \q})_0$ that contains $CU(1_{\mathfrak{A}} \otimes \mathcal{Z}_{\p, \q})_0$ and let $u_i, v_i : [0,1] \rightarrow U(\mathfrak{A} \otimes 1_{\mathcal{Z}_{\p,\q}})_0$
($1 \leq i \leq n$) be norm-continuous paths.  Define $w$ by 
\begin{align*}
w = \prod_{i=1}^n (u_i, v_i) = (u_1, v_1) (u_2, v_2) ... (u_n , v_n).
\end{align*}
Note that $w \in CU(C[0,1] \otimes \mathfrak{A} \otimes 1_{\M_{\p} \otimes
\M_{\q}})_0 \subseteq CU(\A \otimes \ZC_{\p, \q})_0.$  If $w(0) = 1$, then $w \in G$.    
\label{lem:leftendpoint}
\end{lemma}

\begin{proof}

   Since $G$ is a closed subset of $U(\mathfrak{A} \otimes \mathcal{Z}_{\p, \q})_0$, we may assume that there exists $\delta > 0$ 
($\delta < 1$) such that for all $t \in [0, 2\delta)$,
$w(t) = 1$.  
By definition of $\q$,
\[
\mathsf{M}_{\q} = \overline{\bigcup_{j=1}^{\infty} \mathsf{M}_{3^j}}
\]
with connecting maps defined by  
$a \mapsto \mathrm{diag}(a, a, a)$.

As an intermediate step, we will work inside $C[0,1] \otimes \mathfrak{A} \otimes \mathsf{M}_3$.  Let $x_1, x_{2,i}, x_{3,i}, x_{4,i} \in C[0,1] \otimes \mathfrak{A} \otimes \mathsf{M}_3$ be given as
follows:

\[
x_1 = \left[ \begin{array}{ccc} 0 & 1 & 0\\ 1 & 0 & 0
\\ 0 & 0 & -1 \end{array} \right] 
\quad
x_{2,i} = \left[ \begin{array}{ccc} u_i & 0 & 0 \\ 0 & v_i^* & 0 \\
0 & 0 & 1\end{array} \right] 
\quad 
x_{3,i} = \left[ \begin{array}{ccc} 1 & 0 & 0 \\ 0 & v_i & 0 \\
0 & 0 & 1 \end{array} \right] 
\quad x_{4,i} = \left[ \begin{array}{ccc} 1 & 0 & 0 \\ 0 & u_i & 0 \\
0 & 0 & 1 \end{array} \right]. 
\]
Hence, in $C[0,1] \otimes \mathfrak{A} \otimes \mathsf{M}_3$, we have that 
\[
\left[
\begin{array}{ccc} 
u_i v_i  & 0 & 0\\
0 & u_i^* v_i^* & 0 \\
0 & 0 & 1
\end{array}
\right]
= x_{3,i} x_{2,i} x_1 x_{2,i}^* x_1^{*} x_{3,i}^* 
= x_{3,i} (x_{2,i}, x_1 ) x_{3,i}^*
\]
and 
\[
\left[
\begin{array}{ccc} 
u_i^* v_i^*  & 0 & 0\\
0 & u_i v_i & 0 \\
0 & 0 & 1
\end{array}
\right]  
= x_{3,i}^* x_{2,i}^* x_1 x_{2,i} x_1^{*} x_{3,i} 
= x_{3,i}^* (x_{2,i}^*, x_1 ) x_{3,i}.  
\]
Hence, 
\begin{equation} \label{equ:firsttwo}
\begin{aligned}
\left[
\begin{array}{ccc} 
(u_i,  v_i)  & 0 & 0\\
0 & (u_i, v_i)  & 0 \\
0 & 0 & 1
\end{array}
\right]
 &=  x_{3,i} x_{4,i} [ x_{3,i} (x_{2,i}, x_1 ) x_{3,i}^* ]
[x_{3,i}^* (x_{2,i}^*, x_1 ) x_{3,i}]
x_{4,i}^* x_{3,i}^* \\
\\
&=  x_{3,i} x_{4,i}  x_{3,i} (x_{2,i}, x_1 ) {x_{3,i}^*}^2 
(x_{2,i}^*, x_1 ) x_{3,i}
x_{4,i}^* x_{3,i}^*.  
\end{aligned}
 \end{equation}
 Let $y_1, y_{2,i}, y_{3,i}, y_{4,i} \in C[0,1] \otimes \mathfrak{A} \otimes \mathsf{M}_3$ be given
as follows:
\begin{align*}
y_1 = \left[
\begin{array}{ccc}
-1 & 0 & 0 \\ 0 & 0 & 1 \\ 0 & 1 & 0 
\end{array} \right]
\quad
y_{2,i} = \left[
\begin{array}{ccc}
1 & 0 & 0 \\ 0 &  u_i & 0  \\ 0 & 0 & v_i^* 
\end{array} \right] \quad
y_{3,i} = \left[
\begin{array}{ccc}
1 & 0 & 0 \\ 0 &  1 & 0  \\ 0 & 0 & v_i 
\end{array} \right]
\quad 
y_{4,i} = \left[
\begin{array}{ccc}
1 & 0 & 0 \\ 0 &  1 & 0  \\ 0 & 0 & u_i 
\end{array} \right]. 
\end{align*}
Hence, in $C[0,1] \otimes \mathfrak{A} \otimes \mathsf{M}_3$, we have that
\[
 \left[
\begin{array}{ccc}
1 & 0 & 0 \\ 0 &  v_iu_i & 0  \\ 0 & 0 & u_i^* v_i^* 
\end{array} \right]
= y_1 y_{2,i}^* y_1^{*} y_{2,i} = (y_1, y_{2,i}^*)   
\]
and 
\[
\left[
\begin{array}{ccc}
1 & 0 & 0 \\ 0 &  u_i^*v_i^* & 0  \\ 0 & 0 & u_i v_i 
\end{array} \right]
=  y_{3,i}^* y_{2,i}^* y_1 y_{2,i} y_1^{*} y_{3,i}   
= y_{3,i}^* (y_{2,i}^*, y_1) y_{3,i}. 
\]
Hence,
\begin{equation} \label{equ:Duty}
 \begin{aligned}
 &\left[ \begin{array}{ccc} (u_i, v_i) & 0 & 0 \\ 0 & (u_i, v_i) & 0 \\
0 & 0 & (u_i, v_i) \end{array} \right] \\ 
&\qquad = y_{3,i} y_{4,i} 
\left[ \begin{array}{ccc} (u_i, v_i) & 0 & 0 \\ 0 & u_i v_i u_i^* v_i^*
 & 0 \\
0 & 0 & 1 \end{array} \right] (y_1, y_{2,i}^*) [ y_{3,i}^* (y_{2,i}^*, y_1) 
y_{3,i}] y_{4,i}^* y_{3,i}^*.    
\end{aligned}
\end{equation}
From this and (\ref{equ:firsttwo}), 
we have in $C[0,1] \otimes \mathfrak{A} \otimes \mathsf{M}_3$  that 
\begin{equation} \label{equ:Joy} 
\begin{aligned}
 w &=  \prod_{i=1}^n \left[ \begin{array}{ccc}  (u_i, v_i) & 0 & 0 \\
0 & (u_i, v_i) & 0 \\ 0 & 0 & (u_i, v_i) \end{array} \right] \\
&=  \prod_{i=1}^n   
y_{3,i} y_{4,i} 
\left[ \begin{array}{ccc} (u_i, v_i) & 0 & 0 \\ 0 & (u_i, v_i)
 & 0 \\
0 & 0 & 1 \end{array} \right] (y_1, y_{2,i}^*) [ y_{3,i}^* (y_{2,i}^*, y_1) 
y_{3,i}] y_{4,i}^* y_{3,i}^*.    
\end{aligned}
\end{equation}
 
   We now use the above to manufacture elements in 
$CU(\mathfrak{A} \otimes \mathcal{Z}_{\p, \q})_0$.  Let $h : [0, 1] \rightarrow [0, \delta]$ be the continuous 
function that
is given by 
$$h(t) = \begin{cases}
 t & \makebox{  for  } t \in [0, \delta ] \\
\delta & \makebox{  for  } t \in [\delta, 1].  
\end{cases} $$ 
For $1 \leq i \leq n$, $2 \leq j \leq 4$, 
let $x'_1, y'_1, x'_{j,i}, y'_{j,i}  : 
[0,1] \rightarrow U(\mathfrak{A} \otimes \mathsf{M}_3)_0$
be continuous functions such that 
the following hold:
\begin{enumerate}
\item[(i)] $x'_1(t) = x_1(t)$,  $y'_1(t) = y_1(t)$,
$x'_{j,i}(t) = x_{j,i}(t)$ and 
$y'_{j,i}(t) = y_{j,i}(t)$  for $t \in [\delta, 1]$.    
\item[(ii)] $x'_1(0) = y'_1(0) = x'_{j,i}(0) = y'_{j,i}(0) = 1$.  
\item[(iii)] $x'_1(t), y'_1(t), x'_{j,i}(t), y'_{j,i}(t) 
\in \A \otimes \M_3$ for $t \in (0, \delta]$. 
\item[(iv)] $x'_1, y'_1, x'_{j,i}, y'_{j,i} \in U(\A \otimes \ZC_{\p, \q})_0$.
\item[(v)] 
$x'_1, y'_1, x'_1 \circ h, y'_1 \circ h \in CU(1_{\mathfrak{A}} \otimes \mathcal{Z}_{\p, \q}
)_0$.  
\end{enumerate}
(Note that the determinants of $x_1$ and $y_1$ are both one.) 
Next, we replace $x_1, y_1, x_{j,i}, y_{j,i}$ with 
$x'_1, y'_1, x'_{j,i}, y'_{j,i}$ respectively in the expressions in
(\ref{equ:firsttwo}) and (\ref{equ:Duty}) -- except at one occurence
of $y_1$.    
We now proceed with the details. 

  By (i)--(v), we have that $x'_1 \in CU(1_{\mathfrak{A}} \otimes \mathcal{Z}_{\p, \q})_0
\subseteq G$.  Since $G$ is a normal subgroup of $U(\mathfrak{A} \otimes \mathcal{Z}_{\p, \q})_0$,
\[
z_i =  
x'_{3,i} x'_{4,i}  x'_{3,i} (x'_{2,i}, x'_1 ) 
{{x'_{3,i}}^*}^2  ({x'_{2,i}}^*, x'_1 ) x'_{3,i}
{x'_{4,i}}^* {x'_{3,i}}^* \in G. 
\]
Note also by (\ref{equ:firsttwo}) and (i)--(v), that 
\[
\left[ \begin{array}{ccc} (u_i (t), v_i(t)) & 0 & 0 \\
0 & (u_i (t), v_i(t)) & 0 \\ 0 & 0 & 1 \end{array} \right]
= z_i(t) 
\]
for $t \in [\delta, 1]$.  Also by (i)--(v), since $y'_1 \in G$ and since
$G$ is a normal subgroup of $U(\mathfrak{A} \otimes \mathcal{Z}_{\p, \q})_0$,  
\[
z'_i = 
{y'}_{3,i} y'_{4,i} z_i (y'_1, {y'}_{2,i}^*) [
{y'_{3,i}}^* ({y'_{2,i}}^*, y'_1) {y'}_{3,i}] {y'}_{4,i}^* 
{y'}_{3,i}^* \in G.  
\]
Again, by (\ref{equ:Duty}) and (i)--(v), 
\[ 
\left[
\begin{array}{ccc} (u_i(t), v_i(t)) & 0 & 0 \\ 0 & (u_i(t) , v_i(t)) 
& 0 \\
0 & 0 & (u_i(t), v_i(t)) \end{array} \right]  = z'_i(t)
\]
for $t \in [\delta, 1]$.

   Now, recall that we have the condition
$$1 = w(t) = \prod_{i=1}^n (u_i(t), v_i(t))$$
for $t \in [0, \delta]$. To obtain this, we replace the
last occurrence of $y'_1$ in $z'_n$ by another element
$y''_1 \in G$.  
We define $y''_1$ as follows:
\[
y''_1(t) =
y_1(t)
\]
for  $t \in [\delta, 1]$, and for $t \in [0, \delta ]$, set $y''_{1} ( t )$ to be 

\vspace*{4ex}

${y'}_{2,n}^*(t) {y'}_1^*(t)  {y'}_{2,n}(t) y'_{3,n}(t) (y'_1(t), {y'}_{2,n}^*(t))^* z_n(t)^* {y'}_{4,n}^*(t) 
{y'}_{3,n}^*(t) 
 \left[ \prod_{i=1}^{n-1} z'_i(t) \right]^* 
\linebreak
{y'}_{3,n}(t) 
y'_{4,n}(t) {y'}_{3,n}(t)^*.$

\vspace*{4ex}

 (The complicated definition over $[0, \delta]$
is to ensure that 

\vspace*{2ex}

$\left[ \prod_{i=1}^{n-1} z'_i(t) \right] 
{y'}_{3,n}(t) {y'}_{4,n}(t) z_n(t) 
(y'_1(t), {y'}_{2,n}^*(t) ) 
[{y'}_{3,n}^*(t) [{y'}_{2,n}^*(t)  y'_1(t)
{y'}_{2,n}(t) 
\linebreak
y''_1(t)]
{y'}_{3,n}(t)] {y'}_{4,n}^*(t) 
{y'}_{3,n}^*(t) = 1$

\vspace*{2ex}

for $t \in [0, \delta]$.)  From (\ref{equ:Duty}) and (i.)--(v.), we have that
 $y''_1 \in C[0,1] \otimes \mathfrak{A} \otimes \mathsf{M}_{\p} \otimes
\mathsf{M}_{\q}$, $y''_1(t) \in \mathfrak{A} \otimes \mathsf{M}_3$ for 
$t \in (0, 1]$ and   
$y''_1(0) = 1$.  Thus, $y''_1 \in U(\mathfrak{A} \otimes \mathcal{Z}_{\p, \q})_0$.  

    We now prove that $y''_1 \in G$.
Note that by (i)--(v) and by an argument similar to that
used to show that
$z_i, z'_i \in G$, we have that 
$z_i \circ h, z'_i \circ h, x'_1 \circ h, y'_1 \circ h \in G$.  Note that $y''_{1}$ is equal to

\vspace*{4ex}

$({y'}_{2,n}^*\circ h) ({y'}_1^* \circ h)  ({y'}_{2,n}\circ h) (y'_{3,n}\circ h)
(y'_1\circ h, {y'}_{2,n}^*\circ h)^* (z_n^* \circ h) ({y'}_{4,n}^* \circ h) 
({y'}_{3,n}^* 
\circ
\linebreak
 h) \left[ \prod_{i=1}^{n-1} (z'_i \circ h) 
 \right]^*
({y'}_{3,n} \circ h) 
(y'_{4,n} \circ h)({y'}_{3,n}^* \circ h)$.  

\vspace*{4ex}

\noindent Since $G$ is a normal subgroup of $U(\mathfrak{A} \otimes \mathcal{Z}_{\p, \q})_0$,
we have that $y''_1 \in G$.  Since $G$ is a normal subgroup of 
$U(\mathfrak{A} \otimes \mathcal{Z}_{\p, \q})_0$ and  
\begin{align*}
w = \left[ \prod_{i=1}^{n-1} z'_i \right]
{y'}_{3,n} {y'}_{4,n} z_n
(y'_1, {y'}_{2,n}^* )
[{y'}_{3,n}^* [{y'}_{2,n}^*  y'_1
{y'}_{2,n} y''_1]
{y'}_{3,n}] {y'}_{4,n}^*
{y'}_{3,n}^*
\end{align*}
we have that $w \in G$ as required. 
\end{proof}

\begin{lemma}
Consider the supernatural numbers $\p = 2^{ \infty }$ and $\q = 3^{ \infty }$.  Let
$\mathfrak{A}$ be a simple unital $C^{*}$-algebra and let $G$ be a closed normal
subgroup of $U(\mathfrak{A} \otimes \mathcal{Z}_{\p, \q})_0$ that contains $CU(1_{\mathfrak{A}} \otimes \mathcal{Z}_{\p, \q})_0$.  Let $u_i, v_i : [0,1] \rightarrow U(\mathfrak{A} \otimes 1_{\mathcal{Z}_{\p,\q}})_0$
($1 \leq i \leq n$) be norm-continuous paths and let
\begin{align*}
w = \prod_{i=1}^n (u_i, v_i) = (u_1, v_1) (u_2, v_2) ... (u_n , v_n).
\end{align*}
Note that $w\in CU(C[0,1] \otimes \mathfrak{A} \otimes 1_{\M_{\p} \otimes \M_{\q}})_0
\subseteq CU(\A \otimes \ZC_{\p, \q})_0.$  If $w(1) = 1$, then $w \in G$.

\label{lem:rightendpoint}
\end{lemma}

\begin{proof}

The proof is similar (actually slightly easier) to that of
Lemma \ref{lem:leftendpoint}.
\end{proof}

\begin{lemma}  
Consider the supernatural numbers $\p = 2^{ \infty }$ and $\q = 3^{ \infty }$.  Let
$\mathfrak{A}$ be a simple unital $C^{*}$-algebra and let $G$ be a closed normal
subgroup of $U(\mathfrak{A} \otimes \mathcal{Z}_{\p, \q})_0$ that contains $CU(1_{\mathfrak{A}} \otimes \mathcal{Z}_{\p, \q})_0$.  Let $u_i, v_i : [0,1] \rightarrow U(\mathfrak{A} \otimes 1_{\mathcal{Z}_{\p,\q}})_0$
($1 \leq i \leq n$) be norm-continuous paths and let

\begin{align*}
w = \prod_{i=1}^n (u_i, v_i) = (u_1, v_1) (u_2, v_2) ... (u_n , v_n).
\end{align*}
Note that $w\in CU(C[0,1] \otimes \mathfrak{A} \otimes 1_{\M_{\p} \otimes \M_{\q}})_0
\subseteq CU(\A \otimes \ZC_{\p, \q})_0.$  Then $w \in G$. 
\label{lem:fullendpoints}
\end{lemma}
\begin{proof}
Decompose $w$ into a product of unitaries $w = w' w''$ such that 
$w'$ satisfies the hypotheses of 
Lemma \ref{lem:leftendpoint} and 
$w''$ satisfies the hypotheses of
Lemma \ref{lem:rightendpoint}. 
(Sketch of argument:  For all $i$,
since $u_i(1) \in U(\mathfrak{A} \otimes 1_{\mathcal{Z}_{\p,\q}})_0$, 
there exists a norm continuous path
$u'_i : [0,1] \rightarrow U(\mathfrak{A} \otimes 1_{\mathcal{Z}_{\p,\q}})_0$
such that $u'_i(0) = 1$ and $u'_i(1) 
= u_i(1)$.
Let $w' = \prod_{i=1}^n (u'_i, v_i)$ and 
$w'' = w'^* w$.)  Now apply Lemmas \ref{lem:leftendpoint} and \ref{lem:rightendpoint}.   
\end{proof}

\begin{lemma}
Consider the supernatural numbers $\p = 2^{ \infty }$ and $\q = 3^{ \infty }$.  Let
$\mathfrak{A}$ be a simple unital $C^{*}$-algebra and let $G$ be a closed normal
subgroup of $U(\mathfrak{A} \otimes \mathcal{Z}_{\p, \q})_0$ that contains $CU(1_{\mathfrak{A}} \otimes \mathcal{Z}_{\p, \q})_0$.
Then $CU(\mathfrak{A} \otimes 1_{\mathcal{Z}_{\p, \q}})_0 \subseteq G$.

\label{lem:Aotimes1}
\end{lemma}
\begin{proof}
Say that $u, v \in 
U(\mathfrak{A} \otimes 1_{\mathcal{Z}_{\p, \q}})_0$.  Thus, $u, v$ are constant functions from $[0,1]$ to  
$U(\mathfrak{A} \otimes 1_{\M_{\p} \otimes \M_{\q}})_0$ (always taking the
constant values $u(0)$, $v(0)$ respectively).  It follows from
Lemma \ref{lem:fullendpoints} that 
$(u, v) \in G$.
Thus, since $G$ is a group, the commutator subgroup of
$U(\mathfrak{A} \otimes 1_{\mathcal{Z}_{\p, \q}})_0$ is contained 
in $G$. Hence, since $G$ is closed,
$CU(\mathfrak{A} \otimes 1_{\mathcal{Z}_{\p, \q}})_0 \subseteq G$.
\end{proof}

\begin{theorem}  Let $\mathfrak{A}$ be an exact, separable, simple, unital, $\mathcal{Z}$-stable
$C^{*}$-algebra.  Suppose that $G$ is a closed normal subgroup of $U(\mathfrak{A})_0$
that properly contains $\mathbb{T}$.  Then $CU(\mathfrak{A})_0 \subseteq G$.
\label{thm:BIGCONTAINS}
\end{theorem}

\begin{proof}
Since $\mathfrak{A}$ is $\mathcal{Z}$-stable, $\mathfrak{A} \cong \mathfrak{A} \otimes \mathcal{Z} \cong \mathfrak{A} \otimes \ZC \otimes \ZC$.  Hence, it is enough to prove the theorem with $\mathfrak{A}$ replaced by $\mathfrak{A} \otimes \mathcal{Z} \otimes \ZC$.  By Lemma \ref{lem:containU(Z)}, we have that  
$CU(1_{\mathfrak{A} \otimes \mathcal{Z} } \otimes \ZC)_0 \subseteq G$.  It suffices to prove the following:  Let $u, v \in U(\A \otimes \ZC \otimes \ZC)_0$
be given.  Then $(u, v) \in G$.

By Lemma \ref{l:isoz}, there is a $*$-isomorphism $\Phi : \A \otimes \mathcal{Z} \rightarrow \A \otimes \ZC \otimes \ZC$  
such that $\Phi$ is approximately unitarily equivalent to 
the map $\A \otimes \ZC \rightarrow \A \otimes \ZC \otimes \ZC: a \mapsto a \otimes 1_{\ZC}$,
where the unitaries come from $U(\A \otimes \ZC \otimes \ZC)_0$.  Let $u', v' \in U(\A \otimes \ZC)_0$ be unitaries such that 
$\Phi(u') = u$ and $\Phi(v') = v$.
Hence, $u' \otimes 1_{\ZC_{\p, \q}}$, $v' \otimes 1_{\ZC_{\p, \q}}$
are approximately unitarily equivalent to $u$, $v$ respectively, 
with implementing unitaries in $U(\A \otimes \ZC \otimes \ZC)_0$.
  
    Let $\p, \q$ be the supernatural numbers that are given by 
$\p = 2^{\infty}$ and $\q = 3^{\infty}$.  By Theorem 3.4 of \cite{RordamWinter}, $\mathcal{Z}$ is a $C^{*}$-inductive limit
$\mathcal{Z} = \overline{\bigcup_{n=1}^{\infty} \mathcal{Z}_n}$
where $\mathcal{Z}_n \cong \mathcal{Z}_{\p, \q}$ for all $n \geq 1$ and 
where the connecting
maps are unital and injective.
By Lemma \ref{lem:Aotimes1}, we have that 
$(u' \otimes 1_{\ZC_{\p, \q}}, v' \otimes 1_{\ZC_{\p, \q}}) \in G$.  Since $(u' \otimes 1_{\ZC_{\p, \q}}, v' \otimes 1_{\ZC_{\p, \q}})$ is 
approximately unitarily equivalent to $(u,v)$, with unitaries coming
from $U(\A \otimes \ZC \otimes \ZC)_0$ and since $G$ is a closed normal subgroup of $U(\A \otimes \ZC \otimes \ZC)_0$,
$(u, v) \in G$ as required. 
\end{proof}

\begin{lemma} Let $\mathfrak{A}$ be an exact, unital, stably finite, $\mathcal{Z}$-stable $C^{*}$-algebra.
For every $\epsilon >0$, there exists $\delta > 0$
such that for every self-adjoint element $a \in \mathfrak{A}$ such that
$|\tau(a)| < \delta$ for all $\tau \in T(\mathfrak{A})$,
$$\mathrm{dist}(e^{ia}, CU(\mathfrak{A})_0) = \inf \{ \| e^{ia} - u \| : u \in CU(\mathfrak{A})_0 \}
< \epsilon.$$
\label{lem:CUApproximation} 
\end{lemma}

\begin{proof}
   This follows from \cite{ThomsenExactSequence}  
which gives a topological group isomorphism:
\[
\Delta : U(\mathfrak{A})_0 / CU(\mathfrak{A})_0 \rightarrow \mathrm{Aff}(T(\mathfrak{A})) / \overline{K_0(\mathfrak{A}})
\]
where $\Delta$ is the determinant map.   
Note that for a self-adjoint $a \in \mathfrak{A}$,  
$\Delta ( [ e^{i 2 \pi a} ] )  = a + \overline{K_0(\mathfrak{A}})$.   
\end{proof}

\begin{lemma} Let $\mathfrak{A}$ be an exact, simple, unital, stably finite $C^{*}$-algebra 
and let $\mathfrak{C}$ be a UHF-algebra.  Then for every $\epsilon > 0$, for every $N \geq 1$,
for every nonzero  selfadjoint element 
$a \in \mathfrak{A} \otimes \mathfrak{C}$ 
and nonzero projection $p \in \mathfrak{A} \otimes \mathfrak{C}$, 
there exists a self-adjoint  element $c \in p \mathfrak{A} p$  
such that 
$$| \tau(a)  -  N \tau(c) | < \epsilon$$
for all $\tau \in T(\mathfrak{A} \otimes \mathfrak{C})$. 
\label{lem:UHFSelfAdjointTracialApproximation}  
\end{lemma}

\begin{proof}
  We have that $\mathfrak{C}$ can be realized as an inductive limit
$$\mathfrak{C} = \overline{\bigcup_{k=1}^{\infty} \mathsf{M}_{n_k}}$$
where the connecting maps are diagonal maps
$$\mathsf{M}_{n_k} \rightarrow \mathsf{M}_{n_{k+1}} : c \mapsto
\bigoplus^{ \frac{ n_{k+1} }{ n_k }} c.$$
Note that $n_k \rightarrow  \infty$.  Hence, $\mathfrak{A} \otimes \mathfrak{C}$ can be realized as an inductive limit
$$\mathfrak{A} \otimes \mathfrak{C} = \overline{\bigcup_{k=1}^{\infty} \mathsf{M}_{n_k}(\mathfrak{A})}$$
where the connecting maps are diagonal maps
$$\mathsf{M}_{n_k}(\mathfrak{A})  \rightarrow \mathsf{M}_{n_{k+1}}(\mathfrak{A})  : c \mapsto
\bigoplus^{ \frac{ n_{k+1} }{ n_k }} c.$$
Note that the connecting map divides $c$ up into $\frac{ n_{k+1} }{ n_k }$ pairwise
orthogonal Cuntz equivalent pieces, each with $\frac{ n_k }{ n_{k+1} }$th trace
of $c$ for any tracial state on $\mathfrak{A} \otimes \mathfrak{C}$.

    The result follows from     
the nature of the connecting maps for the inductive limit decomposition of
$\mathfrak{A} \otimes \mathfrak{C}$ and since $\mathfrak{A} \otimes \mathfrak{C}$ has strict comparison for
positive elements (see Lemma \ref{lem:AotimesUHFProperties}). 
\end{proof}

\begin{lemma}  Let $\mathfrak{A}$ be a unital $C^{*}$-algebra and 
let $a, b \in \mathfrak{A}$ be self-adjoint elements.  Then 
$$e^{ia} e^{ib} e^{-i(a + b)} \in CU(\mathfrak{A})_0.$$ 
\label{lem:CampbellBaker}
\end{lemma}

\begin{proof}

This follows immediately from the formula

\begin{eqnarray*}
 \exp(ia) \exp(ib) \exp(-i(a +b)) = \lim_{n \rightarrow \infty} \exp(ia) \exp(ib) \left[ 
\exp\left(\frac{-ia}{n}\right) \exp\left(\frac{-ib}{n} 
\right) \right]^n.
\end{eqnarray*}
\end{proof}

The next lemma is motivated by the main result of 
\cite{PopaQuasidiagonal}.  However, we note that unlike
\cite{PopaQuasidiagonal}, we need to assume that our $C^*$-algebra
is nuclear.   In fact, it is not clear
to us whether if $\mathfrak{A}$ is an arbitrary 
unital simple separable
quasidiagonal $C^*$-algebra and $\mathfrak{C}$ is a UHF-algebra then
$\mathfrak{A} \otimes \mathfrak{C}$ has the ``sufficiently many projections"
condition of Popa.  (See \cite{PopaQuasidiagonal} (1.1), 
(2.1), (2.1') and (2.1'').)

\begin{lemma}  Let $\mathfrak{A}$ be a nuclear, separable, simple, unital $C^{*}$-algebra, and let $\mathfrak{C}$ be a UHF-algebra.  Then $\mathfrak{A}$ is quasidiagonal if and only if $\mathfrak{A} \otimes \mathfrak{C}$
has the Popa property; i.e.,  
for every $\epsilon > 0$ and for every finite subset $\mathcal{F} \subseteq \mathfrak{A}
\otimes \mathfrak{C}$, there exists a nonzero finite dimensional 
sub-$C^{*}$-algebra $\mathfrak{D}$ of $\mathfrak{A} \otimes \mathfrak{C}$ 
with unit $p = 1_{\mathfrak{D}}$ such that
for every $a \in \mathcal{F}$, the following hold:
\begin{enumerate}
\item[i.]  $\| pa - ap \| < \epsilon$ and  
\item[ii.] $pap$ is within $\epsilon$ of an element
of $\mathfrak{D}$.       
\end{enumerate}
\label{lem:PopaProperty}
\end{lemma}

\begin{proof}
  By Theorem 1 of \cite{VoiculescuQuasidiagonal}, 
if $\mathfrak{A} \otimes \mathfrak{C}$ has the Popa property then $\mathfrak{A} \otimes \mathfrak{C}$
is quasidiagonal
(see, for example, \cite{NBrownQuasidiagonal} the argument
after Theorem 12.1.); and  hence, 
$\mathfrak{A}$ is quasidiagonal.  This completes
the proof of the ``if" direction.

   We now prove the ``only if" direction.
Suppose that $\mathfrak{A}$ is quasidiagonal.  
Since $\mathfrak{C}$ is a UHF-algebra,  it can be expressed as an inductive limit
$$\mathfrak{C} = \overline{\bigcup_{k=1}^{\infty} \mathsf{M}_{n_k}}$$
where the connecting maps 
are defined by $c \mapsto \mathrm{diag}(c,c,...,c)$ (each $c$ being repeated
$\frac{ n_{k+1} }{ n_k} $ times).  
Hence, $\mathfrak{A} \otimes \mathfrak{C}$ is an inductive limit
$$\mathfrak{A} \otimes \mathfrak{C} = \overline{\bigcup_{k=1}^{\infty} \mathsf{M}_{n_k}(\mathfrak{A})}$$
where the connecting maps are diagonal maps. 
It suffices to prove the Popa property for finite subsets of the
building blocks $\mathsf{M}_{n_k}(\mathfrak{A})$.     
Let $K \geq 1$ be given.  
Let $\epsilon >0$ be given and let $\mathcal{F} \subseteq \mathsf{M}_{n_K}(\mathfrak{A})$ be a 
finite set.
We may assume that the elements of $\mathcal{F}$ have norm less than or equal to
one.  Let $\multialg(\mathsf{M}_{n_K}(\mathfrak{A}) \otimes \K)$ be the multiplier algebra of the stabilization
of $\mathsf{M}_{n_K}(\mathfrak{A})$. (Note that $\mathsf{M}_{n_K}(\mathfrak{A}) \otimes \K \cong
\mathfrak{A} \otimes \K$.)  Define $\phi : \mathsf{M}_{n_K}(\mathfrak{A}) \rightarrow \multialg(\mathsf{M}_{n_K}(\mathfrak{A}) \otimes \K)$ by $\phi (a) = a \otimes 1_{\multialg(\K)}$ and let $\psi' : \mathsf{M}_{n_K}(\mathfrak{A}) 
\rightarrow \multialg(\K) = \mathbb{B}( \mathcal{H})$  
be any unital (and hence essential) $*$-homomorphism.    
Set $\psi = 1_{\mathsf{M}_{n_K}(\mathfrak{A})} 
\otimes \psi' : \mathsf{M}_{n_K}(\mathfrak{A}) 
\rightarrow \multialg(\mathsf{M}_{n_K}(\mathfrak{A}) \otimes \K)$.  Hence, 
$\phi, \psi : \mathsf{M}_{n_K} \otimes
\mathfrak{A} \rightarrow \multialg(\mathsf{M}_{n_K}(\mathfrak{A}) \otimes \K)$ are injective, unital $*$-homomorphisms.  Note that $\phi$ and $\psi$ are both full $*$-homomorphisms.
Hence, since $\mathfrak{A}$ is nuclear, it follows, by \cite{ElliottKucerovsky}, that
there exists a unitary 
$u \in \multialg(\mathsf{M}_{n_K}(\mathfrak{A}) \otimes \K)$  such that
for every $a \in \mathsf{M}_{n_K} \otimes \mathfrak{A}$,
\begin{equation}
\phi(a) - u \psi(a) u^* \in \mathsf{M}_{n_K}(\mathfrak{A}) \otimes \K.  
\label{equ:Voiculescu}
\end{equation}

   Since $\mathsf{M}_{n_K}(\mathfrak{A})$ is quasidiagonal, 
$\psi'(\mathsf{M}_{n_K}(\mathfrak{A}))$ is a quasidiagonal collection of 
operators on $\mathcal{H}$; i.e.,  
there exists an increasing sequence $\{ p_n \}_{ n = 1}^{ \infty }$ of finite rank operators
on $\mathcal{H}$ such that 
\begin{itemize}
\item[i.] $p_n \rightarrow 1_{\mathbb{B}(\mathcal{H})}$ in the
strong operator topology and
\item[ii.] $\| p_n \psi'(a) - \psi'(a) p_n \| \rightarrow 0$
for all $a \in \mathsf{M}_{n_K} \otimes \mathfrak{A}$.
\end{itemize}
 This and (\ref{equ:Voiculescu}) implies that there exists an integer $N \geq 1$ such that 
\begin{enumerate}
\item $\| u (p_m - p_n) u^* \phi(a) - \phi(a) u (p_m - p_n) u^* \| 
< \frac{ \epsilon }{ 100 }$
for all $m > n \geq N$ and for all $a \in \mathcal{F}$    
\item $u(p_m - p_n) u^* \phi(a) u(p_m - p_n) u^*$ is within $\frac{ \epsilon }{ 100 }$
of an element of the finite dimensional $C^{*}$-algebra
$u(p_m - p_n) \mathbb{B}( \mathcal{H}) (p_m - p_n) u^*$ for all 
$a \in \mathcal{F}$ and all $m > n \geq N$.   
\end{enumerate}

   Now let $\{ e_{i,j} : 1 \leq i,j < \infty \}$ be a system of matrix units
for $\K$. 
Note that $1_{\multialg(\mathsf{M}_{n_K}(\mathfrak{A}) \otimes \K)} = \sum_{i=1}^{\infty} 1 \otimes 
e_{i,i}$ where   
the sum converges in the strict topology on $\multialg(\mathsf{M}_{n_K}(\mathfrak{A}) \otimes \K)$.   
Hence, 
for all $a \in \mathsf{M}_{n_K}(\mathfrak{A})$, 
$$\phi(a) = a \otimes 1 = \sum_{i=1}^{\infty} a \otimes e_{i,i}$$
where the sum converges in the strict topology on $\multialg(\mathsf{M}_{n_K}(\mathfrak{A}) \otimes
\K)$.  
Also, note that for every $n \geq 1$, 
\[
\lim_{M \rightarrow \infty} \left( 
\sum_{i=1}^M 1 \otimes e_{i,i} \right) u p_n u^*  
= u p_n u^*.    
\]
Hence, (1)--(2) and the definition of $\phi$
implies that there exists an $M \geq 1$ and
there exists a nonzero projection
$r \in \mathsf{M}_M(\mathsf{M}_{n_K}(\mathfrak{A}))$  (which, inside $\multialg(\mathsf{M}_{n_K}(\mathfrak{A}) \otimes \K)$, is 
close to $u(p_m - p_n)u^*$ for some integers $m > n$)   
such that the following hold:
\begin{itemize}
\item[(i)] $M = \frac{ n_L }{ n_K }$ for some $L \geq K$.    
\item[(ii)] $\| (\bigoplus^M a) r - r (\bigoplus^M a) \| < \epsilon$ for all
$a \in \mathcal{F}$.  
\item[(iii)] There exists a finite dimensional sub-$C^{*}$-algebra 
$\mathfrak{D} \subseteq \mathsf{M}_M( \mathsf{M}_{n_K}(\mathfrak{A}))$ such that 
the unit of $\mathfrak{D}$ is $1_{\mathfrak{D}} = r$.
($\mathfrak{D}$ will be ``close to" $u (p_m - p_n) \mathbb{B}( \mathcal{H} ) (p_m - p_n) 
u^*$ for some positive integers $m > n$.) 
\item[(iv)] For every $a \in \mathcal{F}$, 
$r \left( \bigoplus^M a \right) r$ is within $\epsilon$ of an element 
of $\mathfrak{D}$. 
\end{itemize}
     
Now since, in the inductive limit decomposition of $\mathfrak{A} \otimes \mathfrak{C}$, the
connecting map 
$\mathsf{M}_{n_K}(\mathfrak{A}) \rightarrow \mathsf{M}_{n_L}(\mathfrak{A})$ has the form
$c \mapsto \bigoplus^M c$ (a diagonal map), we are done.   
\end{proof}

\begin{lemma}  Let $\mathfrak{A}$ be a quasidiagonal, nuclear, separable, simple, unital
$C^{*}$-algebra.  Let  $a \in \mathfrak{A}$ be a self-adjoint element and 
$\mathcal{F} \subseteq \mathfrak{A}$ be a finite set.  Then for every $\epsilon > 0$, there exists a unitary 
$u \in CU(\mathfrak{A} \otimes \mathcal{Z})_0$ such that 
\[
\| (e^{ia} \otimes 1_{\mathcal{Z}}) (b \otimes 1_{\mathcal{Z}}) (e^{-ia} \otimes 1_{\mathcal{Z}}) 
- u (b \otimes 1_{\mathcal{Z}}) u^* \| < \epsilon
\] 
\label{lem:Inn(A)Aotimes1}
for all $b \in \mathcal{F}$.
\end{lemma}

\begin{proof}

   We may assume that every element of $\mathcal{F}$ has norm less than or
equal to one.  Let $\p, \q$ be the relatively prime supernatural numbers given by
$\p = 2^{\infty}$ and $\q = 3^{\infty}$.  By Theorem 3.4 of \cite{RordamWinter}, $\mathcal{Z}$ is a $C^{*}$-inductive limit
$\mathcal{Z} = \overline{\bigcup_{n=1}^{\infty} \mathcal{Z}_n}$
where $\mathcal{Z}_n \cong \mathcal{Z}_{\p, \q}$ for all $n \geq 1$ and the connecting
maps are unital and injective.

    By Lemma \ref{lem:PopaProperty}, both
$\mathfrak{A} \otimes \mathsf{M}_{\p}$ and $\mathfrak{A} \otimes \mathsf{M}_{\q}$ have the Popa 
property.    
Hence, let $\mathfrak{D}_1 \subseteq \mathfrak{A} \otimes \mathsf{M}_{\p}$ and
$\mathfrak{D}_2 \subseteq \mathfrak{A} \otimes \mathsf{M}_{\q}$ be nonzero finite dimensional simple sub-$C^{*}$-algebras
with units $e_1= 1_{\mathfrak{D}_1}$ and $e_2= 1_{\mathfrak{D}_2}$ such that the following statements
hold:
\begin{itemize}
\item[(1)] $\| c e_1 - e_1 c \| 
< \frac{ \epsilon }{ 100 }$ 
for all $c \in (\mathcal{F} \otimes 1_{\mathsf{M}_{\p}}) \cup \{ a \otimes 1_{\mathsf{M}_{\p}} \} 
\cup \{ e^{ita} \otimes 1_{\mathsf{M}_{\p}} : t \in [0,1] \}$.    
\item[(2)] $e_1 c e_1$ is within $\frac{ \epsilon }{ 100 }$ of an element of 
$\mathfrak{D}_1$ for all
$c \in (\mathcal{F} \otimes 1_{\mathsf{M}_{\p}}) \cup \{ a \otimes 1_{\mathsf{M}_{\p}} \} 
\cup \{ e^{ita} \otimes 1_{\mathsf{M}_{\p}} : t \in [0,1] \}$.
\item[(3)] $\| c e_2 - e_2 c \| 
< \frac{ \epsilon }{ 100 }$
for all $c \in (\mathcal{F} \otimes 1_{\mathsf{M}_{\q}}) \cup \{ a \otimes 1_{\mathsf{M}_{\q}} \} 
\cup \{ e^{ita} \otimes 1_{\mathsf{M}_{\q}} : t \in [0,1] \}$.
\item[(4)] $e_2 c e_2$ is within $\frac{ \epsilon }{ 100 }$ of an element of
$\mathfrak{D}_2$ for all
$c \in (\mathcal{F} \otimes 1_{\mathsf{M}_{\q}}) \cup \{ a \otimes 1_{\mathsf{M}_{\q}} \} 
\cup \{ e^{ita} \otimes 1_{\mathsf{M}_{\q}} : t \in [0,1] \}$.
\end{itemize}
Plug $\frac{ \epsilon }{ 100 }$ into Lemma \ref{lem:CUApproximation} to get 
a positive real number $\delta > 0 $.  Now let $\{ e_{i,j} \}_{1 \leq i,j \leq m }$ be a system of 
matrix units for $\mathfrak{D}_1$.   
By Lemma \ref{lem:UHFSelfAdjointTracialApproximation},
let $d \in e_{1,1} (\mathfrak{A} \otimes \mathsf{M}_{\p}) e_{1,1}$ be a self-adjoint 
element such that 
$$| \tau(a) - m \tau(d) | < \frac{ \delta }{ 100 }$$
for all $\tau \in T(\mathfrak{A} \otimes \mathcal{Z})$.
Consider the element $d_1 \in e_1 (\mathfrak{A} \otimes \mathsf{M}_{\p}) e_1$ 
given by $d_1 = \bigoplus^m d$.  (The $i$th copy of $d$ sits 
inside $e_{i,i}$.)   
Then $d_1$ is a self-adjoint element of $\mathfrak{A} \otimes \mathsf{M}_{\p}$
such that the following statements hold:

\begin{enumerate}
\item[(a)] $d_1$ commutes with every element of $\mathfrak{D}_1$ and   
$e^{-isd_1}$ commutes with every element of $\mathfrak{D}_1 \oplus (1 - e_1) (\mathfrak{A}
\otimes \mathsf{M}_{\p} ) (1 - e_1)$ for all $s \in [0,1]$.  
Hence, by (1)--(4),   
we have that $\| e^{-isd_1} c - c e^{-i sd_1} \| < \frac{ 6 \epsilon }{ 100 }$ 
for all $s \in [0,1]$ and for
all $c \in (\mathcal{F} \otimes 1_{\mathsf{M}_{\p}}) \cup \{ e^{ita} \otimes 1_{\mathsf{M}_{\p}} 
: t \in [0,1] \}$.  
\item[(b)] $| \tau(a) - \tau(d_1) | < \frac{ \delta }{ 100 }$ for all 
$\tau \in T(\mathfrak{A})$.  Hence, 
$| \tau((1-t)a) - \tau((1-t)d_1) | < \frac{ \delta }{ 100 }$ for all
$\tau \in T(\mathfrak{A})$ and for all $t \in [0,1]$.  
Hence, by our choice of $\delta$, by   
Lemma \ref{lem:CampbellBaker} and Lemma \ref{lem:CUApproximation},
the map $[0,1] \rightarrow U(\mathfrak{A} \otimes \mathcal{Z})_0$ defined by $t \mapsto 
(e^{i(1-t)a} \otimes 1)e^{-i(1 - t)d_1}$ is an element of 
$U(\mathfrak{A} \otimes \mathcal{Z}_{\p, \q})_0 \subseteq  U(\mathfrak{A} \otimes \mathcal{Z})_0$ which is 
within $\frac{ \epsilon }{ 100 }$ of an element $u_1$ 
of $CU(\mathfrak{A} \otimes \mathcal{Z})_0$.     
\end{enumerate}
 
    By a similar argument, we can find a self-adjoint element
$d_2 \in e_2(\mathfrak{A} \otimes \mathsf{M}_{\q}) e_2$
satisfying the following statements:
\begin{enumerate}
\item[(i)] $\| e^{-isd_2} c - c e^{-i sd_2} \| < \frac{ 6 \epsilon }{ 100 }$ 
for all $s \in [0,1]$ and for all 
 $c \in (\mathcal{F} \otimes 1_{\mathsf{M}_{\q}}) \cup \{ e^{ita} \otimes 1_{\mathsf{M}_{\q}} 
: t \in [0,1] \}$.
\item[(ii)] The map $[0,1] \rightarrow U(\mathfrak{A} \otimes \mathcal{Z})_0$ defined by $t \mapsto 
(e^{ita} \otimes 1)e^{-itd_2}$ is an element of 
$U(\mathfrak{A} \otimes \mathcal{Z}_{\p, \q})_0 \subseteq  U(\mathfrak{A} \otimes \mathcal{Z})_0$ which is
within $\frac{ \epsilon }{ 100 }$ of an element $u_2$
of $CU(\mathfrak{A} \otimes \mathcal{Z})_0$.
\end{enumerate}
 
   By (a)--(b) and (i)--(ii), we have 
that  
\begin{align*}
&\| (e^{ia} \otimes 1_{\mathcal{Z}}) (b \otimes 1_{\mathcal{Z}}) (e^{-ia} \otimes 1_{\mathcal{Z}})
- u_1 u_2  (b \otimes 1_{\mathcal{Z}}) u_2^* u_1^* \| \\ 
& = \| (e^{i(1-t)a} e^{ita}
 \otimes 1_{\mathcal{Z}}) (e^{-i(1-t)d_1} e^{-i t d_2} e^{it d_2}  e^{i(1-t) d_1}) 
 (b \otimes 1_{\mathcal{Z}}) (e^{-it a} e^{-i(1-t)a} \otimes 1_{\mathcal{Z}})
 \\
 & \quad - u_1 u_2  (b \otimes 1_{\mathcal{Z}}) u_2^* u_1^* \| \\
& \leq \| (e^{i(1-t)a} e^{ita}
 \otimes 1_{\mathcal{Z}}) (e^{-i(1-t)d_1} e^{-i t d_2} e^{i t d_2}
 e^{i(1-t) d_1}) 
 (b \otimes 1_{\mathcal{Z}}) (e^{-it a} e^{-i(1-t)a} \otimes 1_{\mathcal{Z}})  \\
& \quad - (e^{i(1-t)a} \otimes 1) e^{-i(1-t)d_1}(e^{ita}
 \otimes 1) e^{-itd_2} 
 (b \otimes 1_{\mathcal{Z}}) e^{itd_2} (e^{-it a} \otimes 1) e^{i(1-t) d_1} \\
& \qquad (e^{-i(1-t)a} \otimes 1) \| \\ 
&\quad + \|   (e^{i(1-t)a} \otimes 1) e^{-i(1-t)d_1}(e^{ita}
 \otimes 1) e^{-itd_2} 
 (b \otimes 1_{\mathcal{Z}}) e^{itd_2} (e^{-it a} \otimes 1) e^{i(1-t) d_1}\\
& \qquad (e^{-i(1-t)a} \otimes 1) \\
& \quad - u_1 u_2  (b \otimes 1_{\mathcal{Z}}) u_2^* u_1^* \| \\ 
&\quad <  24 \epsilon/100 + 4 \epsilon/100 \\
&\quad <  \epsilon
\end{align*}
for all $b \in \mathcal{F}$.  Hence, taking $u = u_1 u_2$, we are done.
\end{proof}

\begin{lemma}
Let $\mathfrak{A}$ be a quasidiagonal, nuclear, separable, simple, unital $C^{*}$-algebra.
Let $v$ be a unitary in $U(\mathfrak{A})_0$ and let 
$\mathcal{F} \subseteq \mathfrak{A}$ be a finite set.  Then for every $\epsilon > 0$, there exists a unitary $u \in CU(\mathfrak{A} \otimes 
\mathcal{Z})_0$ such that
\[
\| (v \otimes 1_{\mathcal{Z}}) (b \otimes 1_{\mathcal{Z}}) (v^* \otimes 1_{\mathcal{Z}})
- u (b \otimes 1_{\mathcal{Z}}) u^* \| < \epsilon
\]
for all $b \in \mathcal{F}$.
\label{lem:NuclearQuasidiagonalPreCase}
\end{lemma}

\begin{proof}
Since $v \in U( \mathfrak{A} )_{0}$, $v$ has the form $v = e^{ia_1} e^{ia_2} ... e^{ia_n}$ where
$a_1, a_2, .., a_n$ are self-adjoint elements of $\mathfrak{A}$.  The rest of
the proof is the same as that of Lemma \ref{lem:Inn(A)Aotimes1}.    
\end{proof}

\begin{lemma} Let $\mathfrak{A}$ be a quasidiagonal, nuclear, separable, simple, unital $\mathcal{Z}$-stable $C^{*}$-algebra.
Let $v$ be a unitary in $U(\mathfrak{A})_0$ and let $\mathcal{F} \subseteq \mathfrak{A}$ be a finite set.  Then for every $\epsilon > 0$, there exists a unitary 
$u \in CU(\mathfrak{A})_0$ such that 
\[
\| v c v^* - u c u^* \| < \epsilon
\] 
\label{lem:NuclearQuasidiagonalCase}
for all $c \in \mathcal{F}$.
\end{lemma}

\begin{proof}
   Since $\mathfrak{A}$ is $\mathcal{Z}$-stable, $\mathfrak{A} \cong \mathfrak{A} \otimes \mathcal{Z} \cong  \mathfrak{A} \otimes \mathcal{Z} \otimes \mathcal{Z}$.  Hence, it is enough to prove the theorem with $\mathfrak{A}$ replaced by $\mathfrak{A} \otimes \mathcal{Z} \otimes \mathcal{Z}$.  By Lemma \ref{l:isoz} and  Lemma 4.1 of \cite{NgZ}, there exists a $*$-isomorphism $\Phi : \mathfrak{A} \otimes \mathcal{Z} \rightarrow \mathfrak{A} \otimes \mathcal{Z} \otimes \mathcal{Z}$ which is
approximately unitarily equivalent to the natural inclusion map
$\mathfrak{A}  \otimes \mathcal{Z} \rightarrow \mathfrak{A} \otimes \mathcal{Z} \otimes \mathcal{Z}$ defined by $b \mapsto b \otimes 1_{\mathcal{Z}}$, where we
can choose the unitaries to be in $CU(\mathfrak{A} \otimes \mathcal{Z} \otimes \mathcal{Z} )_0$.  Hence, we may assume that the elements of $\mathcal{F} \cup \{ v \}$ are all
inside $\mathfrak{A} \otimes \mathcal{Z} \otimes 1_{\mathcal{Z}}$.   
The result then follows from Lemma \ref{lem:NuclearQuasidiagonalPreCase}. 
\end{proof}

\begin{lemma} Let $\mathfrak{A}$ be an exact, separable, simple, unital, $\mathcal{Z}$-stable
$C^{*}$-algebra with unique tracial state.  
Let $v$ be a unitary in $U(\mathfrak{A})_0$.  Then there exists a unitary $u \in CU(\mathfrak{A})_0$ such that 
\[
v a v^* = u a u^*
\] 
\label{lem:UniqueTraceCase}  
for all $a \in \mathfrak{A}$.
\end{lemma}

\begin{proof}
Let $\tau$ be the unique tracial state of $\mathfrak{A}$.  Since $v \in U( \mathfrak{A} )_{0}$, $v$ has the form 
$v = e^{ia_1} e^{ia_2} ... e^{ia_n}$ where $a_1, a_2, ..., a_n \in \mathfrak{A}$
are self-adjoint elements.  By Lemma \ref{lem:CUApproximation} and Lemma \ref{lem:CampbellBaker}, and
since $e^{i \tau(a_1 + a_2 + ... + a_n)}1$ (being scalar) commutes with
every element of $\mathfrak{A}$,  the unitary 
$$u = e^{-i \tau(a_1 + a_2 + ... + a_n)} v$$
satisfies the requirements of the lemma.  
\end{proof}

\begin{theorem}\label{t:simpletop} Let $\mathfrak{A}$ be an exact, separable, simple, unital $\mathcal{Z}$-stable
$C^{*}$-algebra.  Suppose that either 
\begin{enumerate}
\item $\mathfrak{A}$ is nuclear and quasidiagonal, or  
\item $\mathfrak{A}$ has unique tracial state.   
\end{enumerate}

Then we have the following:
\begin{enumerate}
\item[(a)] $CU(\mathfrak{A})_0/ \mathbb{T}$ is a simple topological group.  
\item[(b)] Every automorphism in $\overline{\mathrm{Inn}}_{0}(\mathfrak{A})$ can be realized
using unitaries in $CU(\mathfrak{A})_0$.
\item[(c)] $\overline{\mathrm{Inn}}_{0}(\mathfrak{A})$ is a simple topological group.   
\end{enumerate} 
\end{theorem}

\begin{proof}
(a) and (b) follow from Lemmas \ref{lem:NuclearQuasidiagonalCase},
\ref{lem:UniqueTraceCase} and Theorem \ref{thm:BIGCONTAINS}.    

(c) follows from (a) and (b), arguing as in the proof of Theorem 3.2(b) of \cite{pner_ugrps} 
(which is a modification of the argument of 
Corollary 2.5 in \cite{EllRorAuto}).  For the convenience of the reader, we provide
the (short) argument: 
Let $G$ be a nontrivial closed normal subgroup
of $\overline{\mathrm{Inn}}_{0}(\mathfrak{A})$.  Let
$H = \{ u \in CU(\mathfrak{A})_0 : \mathrm{Ad}(u) \in G \}$.
Then $H$ is a closed normal subgroup of $CU(\mathfrak{A})_0$ such that
$H$ contains all scalar unitaries. Since $G$ is nontrivial, let
$\beta \in G$ be different from the identity automorphism.
Hence, there exists $v \in CU(\mathfrak{A})_0$ such that 
$v^* \beta(v) \notin \mathbb{C} 1$.
Since $(\mathrm{Ad}(v))^{-1} \beta \mathrm{Ad}(v) \beta^{-1} = 
\mathrm{Ad}(v^* \beta(v))$, 
it follows that $v^* \beta(v) \in H$.
Therefore, by (a), $H = CU(\mathfrak{A})_0$. Hence, by (b),  
$G = \overline{\mathrm{Inn}}_{0}(\mathfrak{A})$.
\end{proof}

\section{Bott Maps and Continuous path of unitaries}\label{pathunitaries}

\subsection{$K$-theory with $\Z_{n}$ coefficient}  For $n \in \N$, define $\ftn{ \theta_{n} }{ C(S^{1}) }{ C( S^{1} ) }$ by $\theta_{n} ( x ) = x^{n}$.  By identifying $C_{0} ( 0,1 )$ with
\begin{align*}
\setof{ f \in C( S^{1} ) }{ f( 1 ) = 0 }  \subseteq C( S^{1} )
\end{align*}
then $\theta_{n} \vert_{ C_{0} ( 0,1 ) }$ is a homomorphism from $C_{0} ( 0,1 )$ to $C_{0} ( 0,1 )$.  Denote by $C_{n}$ the mapping cone of $\theta_{n} \vert_{ C_{0} ( 0,1 ) }$.  Set $C_{0} = \C$.  For a $C^{*}$-algebra $\mathfrak{A}$, define $\underline{K} ( \mathfrak{A} )$ by
\begin{align*}
\underline{K} ( \mathfrak{A} ) = \bigoplus_{ i = 0}^{1} \bigoplus_{ n = 0}^{ \infty } K_{i} ( \mathfrak{A} ; \Z_{n} )
\end{align*}
where $K_{i} ( \mathfrak{A} ; \Z_{n} ) = K_{i} ( \mathfrak{A} \otimes C_{n} )$.  For $C^{*}$-algebras $\mathfrak{A}$ and $\mathfrak{B}$, a homomorphism from $\underline{K} ( \mathfrak{A} )$ to $\underline{K} ( \mathfrak{B} )$ is a collection of group homomorphisms $( \phi_{n}^{0}, \phi_{n}^{1} )_{ n = 0 }^{ \infty }$, where $\ftn{ \phi_{n}^{i} }{ K_{i} ( \mathfrak{A} ; \Z_{n} ) }{ K_{i} ( \mathfrak{B} ; \Z_{n} ) }$, satisfying the Bockstein operations of \cite{multcoeff}.

\subsection{Embeddings}
\begin{itemize}
\item[(1)] For each $n \in \N$, let $\ftn{ j_{n} }{ \mathfrak{A} \otimes C_{0} ( (0,1)^{2} ) }{ \mathfrak{A} \otimes C_{n} }$ be the natural inclusion.

\item[(2)] For a $C^{*}$-algebra $\mathfrak{A}$, let $\ftn{ j_{ \mathfrak{A} } }{ \mathfrak{A} \otimes C_{0} ( 0, 1 ) }{ \mathfrak{A} \otimes C(S^{1}) }$ be the canonical embedding that sends $a \otimes f$, to $a \otimes f$, using the identification of $C_{0} ( 0 ,  1 )$ as a sub-$C^{*}$-algebra of $C( S^{1} )$.

\item[(3)] Let $\mathfrak{A}$ be a $C^{*}$-algebra.  Define $\ftn{ \iota_{ \mathfrak{A} } }{ \mathfrak{A} }{ \mathfrak{A} \otimes C( S^{1} ) }$ by $\iota_{ \mathfrak{A} } ( a ) = a \otimes 1_{ C(S^{1}) }$. 

\item[(4)] Let $m, n \in \N$ with $0 < m \leq n \in \N$, let $\mathfrak{A}_{1} , \dots, \mathfrak{A}_{m}$, $\mathfrak{B}_{1}, \dots, \mathfrak{B}_{n}$ be unital $C^{*}$-algebras, let $\iota_{1} , \dots, \iota_{m} \in \{ 1, \dots, n \}$ be pairwise distinct numbers, and let $\ftn{ \alpha_{j} }{ \mathfrak{A}_{j} }{ \mathfrak{B}_{ \iota_{j} } }$ be a $*$-homomorphism for each $j$.  Then these $*$-homomorphisms induce a $*$-homomorphism 
\begin{equation*}
\ftn{ \alpha = \alpha_{1} \otimes \alpha_{2} \otimes \cdots \otimes \alpha_{m} }{ \mathfrak{A}_{1} \otimes \cdots \otimes \mathfrak{A}_{m} }{ \mathfrak{B}_{ \iota_{1} } \otimes \cdots \otimes \mathfrak{B}_{ \iota_{m} } }.
\end{equation*}
The composition of this map with the canonical unital embedding 
\begin{equation*}
\ftn{ \iota }{ \mathfrak{B}_{ \iota_{1} } \otimes \cdots \otimes \mathfrak{B}_{ \iota_{m} }  } { \mathfrak{B}_{1} \otimes \cdots \otimes \mathfrak{B}_{n} }
\end{equation*}
will be denoted by $\alpha^{ [ \iota_{1} \dots \iota_{m} ] }$.  Note that $\iota$ may be expressed as $\id^{ [ \iota_{1} \dots \iota_{m} ] }$.
\end{itemize}

\subsection{Partial maps on $\underline{K} ( \mathfrak{A} )$}\label{maps}  Let $\mathfrak{A}$ be a $C^{*}$-algebra.  Let $\mathfrak{A}^{\dagger}$ be $\mathfrak{A}$ if $\mathfrak{A}$ is unital and the unitization of $\mathfrak{A}$ if $\mathfrak{A}$ is not unital. We will denote the set of projections and unitaries in $\bigcup_{ m = 1}^{ \infty } ( \mathfrak{A} \otimes \mathsf{M}_{m} )^{ \dagger }$ by $\mathbb{P}_{0} ( \mathfrak{A} )$ and we will denote the set of projections and unitaries in 
\begin{align*}
  \bigcup_{ m = 1}^{ \infty } \bigcup_{ n = 0}^{ \infty }  (\mathfrak{A} \otimes \mathsf{M}_{m} \otimes C_{n} )^{ \dagger} 
\end{align*}
by $\mathbb{P} ( \mathfrak{A} )$.  

Let $\mathcal{P}$ be a subset of $\mathbb{P} ( \mathfrak{A} )$ and $\ftn{ \psi }{ \mathfrak{A} }{ \mathfrak{B} }$ be a contractive, completely positive linear map, we now define $\underline{K} ( \psi ) \vert_{ \mathcal{P } }$ as in Section 2.3 of \cite{hl_apphomtop}.  We use $\underline{K} ( \psi ) \vert_{ \mathcal{P} }$ instead of the notation $[ \psi ] \vert_{ \mathcal{P} }$ used in \cite{hl_apphomtop}.

\begin{definition}
Let $\mathfrak{A}$ and $\mathfrak{B}$ be $C^{*}$-algebras and let $\ftn{ \psi }{ \mathfrak{A} }{ \mathfrak{B} }$ be a linear map.  Let $\epsilon > 0$ and $\mathcal{S} \subseteq \mathfrak{A}$.  Then $\psi$ is \textbf{\emph{$\mathcal{S}$-$\epsilon$-multiplicative}} if 
\begin{equation*}
\norm{ \psi ( ab ) - \psi ( a ) \psi ( b ) } < \epsilon
\end{equation*}
for all $a,b \in \mathcal{S}$. 
\end{definition}

The following lemmas are well-known and can be found in Section 2.5 of \cite{booklin}

\begin{lemma}\label{l:almostunit}
For each $\epsilon > 0$, there exists $\delta > 0$ such that the following holds: Suppose $\mathfrak{A}$ is a unital $C^{*}$-algebra.   
\begin{itemize}
\item[(1)] If $x \in \mathfrak{A}$ with $\norm{ x^{*} x - 1_{ \mathfrak{A} } } < \delta$ and $\norm{ x  x^{*} - 1_{ \mathfrak{A} } } < \delta$, then there exists a unitary $u \in \mathfrak{A}$ such that $\norm{ x - u } < \epsilon$.  Moreover, if $\epsilon < \frac{1}{2}$ and if $u_{1}$ and $u_{2}$ are unitaries such that $\norm{ x - u_{i} } < \epsilon$, then $[ u_{1} ] = [ u_{2} ]$ in $K_{1} ( \mathfrak{A} )$.  

\item[(2)]  If $x \in \mathfrak{A}$ is a self-adjoint element with $\norm{ x^{2} - x } < \delta$, then there exists a projection $p$ in $\mathfrak{A}$ such that $\norm{ x - p } < \epsilon$.  Moreover, if $\epsilon < \frac{1}{2}$ and if $p_{1}$ and $p_{2}$ are projections such that $\norm{ x - p_{i} } < \epsilon$, then $[ p_{1} ] = [ p_{2} ]$ in $K_{0} ( \mathfrak{A} )$.
\end{itemize}
\end{lemma}

\begin{lemma}\label{l:almostmult}
Let $\mathfrak{A}$ be a $C^{*}$-algebra and let $\mathfrak{C}$ be a nuclear $C^{*}$-algebra.  Let $\epsilon > 0$ and let $\mathcal{F}$ be a finite subset of $\mathfrak{A} \otimes \mathfrak{C}$.  Then there exist $\delta > 0$ and a finite subset $\mathcal{G}$ of $\mathfrak{A}$ such that if $\ftn{\psi}{ \mathfrak{A} }{ \mathfrak{B} }$ is a contractive, linear map that is $\mathcal{G}$-$\delta$-multiplicative, then $\ftn{ \psi \otimes \id_{ \mathfrak{C} } }{ \mathfrak{A} \otimes \mathfrak{C}  }{ \mathfrak{B} \otimes  \mathfrak{C}  }$ is $\mathcal{F}$-$\epsilon$-multiplicative.
\end{lemma}

\begin{proof}
Set $m_{1} = \max \setof{ \| x \| }{ x \in \mathcal{F} }$.  Note that each element of $\mathfrak{A} \otimes \mathfrak{C}$ can be approximated by a finite linear combinations of elementary tensors, $a \otimes c$, where $a \in \mathfrak{A}$ and $c \in \mathfrak{C}$.  Thus, there exists $N \in \N$ and there exist $a_{1}, \dots, a_{N} \in \mathfrak{A}$ and $c_{1}, \dots c_{N} \in \mathfrak{C}$ such that for each $x \in \mathcal{F}$, there exist a finite subset $\mathcal{S}_{x} \subseteq \setof{ k \in \N }{ k \leq N }$ and functions $\ftn{ k_{x}, n_{x} }{ \mathcal{S} }{ \setof{ k \in \N }{ k \leq N } }$ such that  
\begin{align*}
\norm{ x -  \sum_{ i \in \mathcal{S}_{x}  }  a_{k_{x}(i)} \otimes c_{ n_{x} (i) } }  < \min \left\{ \frac{\epsilon}{5( m_{1} +1) } , 1 \right\}.
\end{align*}
Note that  $\norm{  \sum_{ i \in \mathcal{S}_{x}  }  a_{k_{x}(i)} \otimes c_{ n_{x} (i) } } \leq 1 + m_{1}$.

Define $M, m_{2} \in \R$ and $\mathcal{G} \subseteq \mathcal{A}$ as follows:  
\begin{align*}
M &= \max \setof{ | \mathcal{S}_{x} | }{ x \in \mathcal{F} } \\
m_{2} &= \setof{ \| c_{i} \| }{ i \in \N , 1 \leq i \leq N } \\
\mathcal{G} &= \setof{ a_{i} }{ i \in \N , 1 \leq i \leq N }.
\end{align*} 
Set $\delta = \frac{ \epsilon }{ 5 M^{2}( m_{2} + 1 )^{2} }$.  Suppose $\ftn{ \psi }{ \mathfrak{A} }{ \mathfrak{B} }$ is a contractive, linear map such that $\psi$ is $\mathcal{G}$-$\delta$-multiplicative.  If $x,y \in \setof{ a_{i} }{ i \in \N , 1 \leq i \leq N }$ and $s,t \in \setof{ c_{i} }{ i \in \N , 1 \leq i \leq n }$, then 
\begin{align*}
\norm{ ( \psi \otimes \id_{ \mathfrak{C} } )( x \otimes s ) ( \psi \otimes \id_{ \mathfrak{C} } )( y \otimes t ) - ( \psi \otimes \id_{ \mathfrak{C} } )( xy \otimes st ) } 
&= \norm{ \psi ( x ) \psi ( y ) \otimes st - \psi ( xy ) \otimes st } \\
&\leq \norm{ \psi ( x ) \psi (y) - \psi ( x y ) } m_{2}^{2} \\
&\leq \delta m_{2}^{2}.
\end{align*}
Let $S_{1}, S_{2}$ be subsets of $\setof{ i \in \N }{ i \leq N }$ and $\ftn{ k_{i}, n_{i} }{ \mathcal{S}_{i} }{ \setof{ i \in \N }{ i \leq N } }$ be functions.  Set $y_{1} = \sum_{ i \in \mathcal{S}_{1} } a_{k_{1}(i) } \otimes c_{n_{1}(i) }$ and $y_{2} = \sum_{ i \in \mathcal{S}_{2} } a_{k_{2}(i) } \otimes c_{n_{2}(i) }$
\begin{align*}
\norm{ ( \psi \otimes \id_{ \mathfrak{C} } ) \left(  y_{1} \right) ( \psi \otimes \id_{ \mathfrak{C} } ) \left( y_{2} \right) - ( \psi \otimes \id_{ \mathfrak{C} } ) \left( y_{1}y_{2} \right) }  &< \sum_{ i \in \mathcal{S}_{1} , j \in \mathcal{S}_{2} }  \delta m_{2}^{2} \\
&< \frac{ \epsilon }{ 5 }.
 \end{align*}

Let $x_{1}, x_{2} \in \mathcal{F}$.  Then there exist subsets $S_{1}, S_{2}$  of $\setof{ i \in \N }{ i \leq N }$ and functions $\ftn{ k_{i}, n_{i} }{ \mathcal{S}_{i} }{ \setof{ i \in \N }{ i \leq N } }$ such that
\begin{align*}
\norm{ x_{i} -  y_{i}}  < \min \left\{ \frac{\epsilon}{5( m_{1} +1) } , 1 \right\}
\end{align*}
where $y_{j} = \sum_{ i \in \mathcal{S}_{j}  }  a_{k_{j}(i)} \otimes c_{ n_{j} (i) }$.  Therefore, 
\begin{align*}
&\norm{ (\psi \otimes \id_{\mathfrak{C}} ) ( x_{1} ) (\psi \otimes \id_{\mathfrak{C}} )  ( x_{2} ) - (\psi \otimes \id_{\mathfrak{C}} )  ( x_{1} x_{2} ) }  \\
&\quad \leq \norm{ (\psi \otimes \id_{\mathfrak{C}} ) ( x_{1} - y_{1} ) } \norm{ (\psi \otimes \id_{\mathfrak{C}} ) ( x_{2} ) } + \norm{ (\psi \otimes \id_{\mathfrak{C}} )  ( y_{1} ) } \norm{ (\psi \otimes \id_{\mathfrak{C}} ) ( x_{2} - y_{2} ) }  \\
&\qquad  +  \norm{ (\psi \otimes \id_{\mathfrak{C}} )  (  y_{1}) (\psi \otimes \id_{\mathfrak{C}} ) ( y_{2} ) - (\psi \otimes \id_{\mathfrak{C}} )  ( y_{1} y_{2} ) } + \norm{ (\psi \otimes \id_{\mathfrak{C}} )  ( y_{1} ( y_{2} - x_{2} ) ) } \\
&\qquad   + \norm{ (\psi \otimes \id_{\mathfrak{C}} ) ( (y_{1} - x_{1} ) x_{2} ) } \\
&\quad \leq \norm{ x_{1} - y_{1} } \norm{  x_{2}  } + \norm{  y_{1} } \norm{ x_{2} - y_{2} } + \norm{ (\psi \otimes \id_{\mathfrak{C}} )  (  y_{1}) (\psi \otimes \id_{\mathfrak{C}} ) ( y_{2} ) - (\psi \otimes \id_{\mathfrak{C}} )  ( y_{1} y_{2} ) } \\
&\qquad  + \norm{  y_{1} } \norm{ y_{2} - x_{2} } + \norm{ y_{1} - x_{1} }\norm{ x_{2}  } \\
&\quad < \epsilon.
\end{align*}
\end{proof}

Using Lemma \ref{l:almostmult} and arguing as in Remark 4.5.1 and 6.1.1 of \cite{booklin} we get the following lemma.  See also Section 2.3 of \cite{hl_apphomtop}.

\begin{lemma}\label{l:partialmap}
Let $\mathfrak{A}$ be a $C^{*}$-algebra.  Let $\mathcal{P}_{0}$ be a finite set of projections in $\bigcup_{ m = 1}^{ \infty } ( \mathfrak{A} \otimes \mathsf{M}_{m})^{\dagger}$, let $\mathcal{P}_{1}$ be a finite set of unitaries in $\bigcup_{ m = 1}^{ \infty } ( \mathfrak{A} \otimes \mathsf{M}_{m} )^{\dagger}$.  Then there exist $\delta > 0$ and a finite subset $\mathcal{F}$ of $\mathcal{A}$ such that the following holds:  Suppose $\mathfrak{B}$ is a $C^{*}$-algebra and $\ftn{ \psi }{ \mathfrak{A} }{ \mathfrak{B} }$ is a contractive, completely positive, linear map such that $\psi$ is $\mathcal{F}$-$\delta$-multiplicative.  Then there exist a finite set $\mathcal{G}_{0}$ of projections in $\bigcup_{ m = 1}^{ \infty } (\mathfrak{B} \otimes \mathsf{M}_{m})^{\dagger}$ and a finite set of unitaries $\mathcal{G}_{1}$ of $\bigcup_{ m = 1}^{ \infty } ( \mathfrak{B} \otimes \mathsf{M}_{m} )^{ \dagger }$ such that 
\begin{itemize}
\item[(1)] for each $p \in \mathcal{P}_{0}$, there exists $e(p) \in \mathcal{G}_{0}$ such that 
\begin{align*}
\norm{ \psi ( p ) - e(p) } < \frac{1}{2}
\end{align*}
and if $p_{1}, p_{2} \in \mathcal{P}_{0}$ such that $[ p_{1} ] = [ p_{2} ]$ in $K_{0} ( \mathfrak{A} )$, then $[ e(p_{1}) ] = [ e( p_{2} ) ]$ in $K_{0} ( \mathfrak{B} )$ and if $p_{1}, p_{2}, p_{1} \oplus p_{2} \in \mathcal{P}_{0}$, then $[ e( p_{1} \oplus p_{2} ) ] = [ e( p_{1} ) ] + [ e( p_{2} ) ]$ in $K_{0} ( \mathfrak{B} )$.

\item[(2)] for each $u \in \mathcal{P}_{1}$, there exists $v(u) \in \mathcal{G}_{1}$ such that 
\begin{align*}
\norm{ \psi ( u ) - v(u) } < \frac{1}{2}
\end{align*}
and if $u_{1}, u_{2} \in \mathcal{P}_{1}$ such that $[ u_{1} ] = [ u_{2} ]$ in $K_{1} ( \mathfrak{A} )$, then $[ v( u_{1} ) ] = [ v ( u_{2} ) ]$ in $K_{1} ( \mathfrak{B} )$ and  if $u_{1}, u_{2}, u_{1} \oplus u_{2} \in \mathcal{P}_{1}$, then $[ v( u_{1} \oplus u_{2} ) ] = [ v( u_{1} ) ] + [ v( u_{2} ) ]$ in $K_{1} ( \mathfrak{B} )$
\end{itemize}
\end{lemma}

\begin{definition}
Let $\mathfrak{A}$ be a $C^{*}$-algebra and let $\mathcal{P}$ be a finite subset of $\mathbb{P} ( \mathfrak{A} )$.  By Lemma \ref{l:partialmap}, there exist a finite subset $\mathcal{F}$ of $\mathfrak{A}$ and $\delta > 0$ such that the following holds:  Suppose $\mathfrak{B}$ is a $C^{*}$-algebra and $\ftn{ \psi }{ \mathfrak{A} }{ \mathfrak{B} }$ is a contractive, complete positive, linear map such that $\psi$ is $\mathcal{F}$-$\delta$-multiplicative.  Then there exists a finite subset $\mathcal{Q}$ of $\mathbb{P} ( \mathfrak{A} )$ such that for each projection $p \in \mathcal{P} \cap ( \mathfrak{A} \otimes \mathsf{M}_{m} \otimes C_{n} )^{\dagger}$ and unitary $u \in \mathcal{P} \cap ( \mathfrak{A} \otimes \mathsf{M}_{m} \otimes C_{n} )^{\dagger}$, there exist a projection $e(p) \in \mathcal{Q} \cap ( \mathfrak{B} \otimes \mathsf{M}_{m} \otimes C_{n} )^{ \dagger } $ and a unitary in $v(u) \in \mathcal{Q} \cap ( \mathfrak{A} \otimes \mathsf{M}_{m} \otimes C_{n} )^{ \dagger } $ such that 
\begin{align*}
\norm{ ( \psi \otimes \id_{ \mathsf{M}_{m} } \otimes \id_{ C_{n} } ) ( p ) - e(p) }  < \frac{1}{2} \quad \text{and} \quad \norm{ ( \psi \otimes \id_{ \mathsf{M}_{m} } \otimes \id_{ C_{n} } )( u ) - v(u) }  < \frac{1}{2}.
\end{align*}
Moreover,
\begin{itemize}
\item[(1)] If $p_{1}, p_{2}$ are projections in $\mathcal{P}$ and $[ p_{1} ] = [ p_{2} ]$ in $\underline{K} ( \mathfrak{A} )$, then $[ e( p_{1} ) ] = [ e ( p_{2} ) ]$ in $\underline{K} ( \mathfrak{B} )$.

\item[(2)] If $u_{1}, u_{2}$ are unitaries in $\mathcal{P}$ and $[ u_{1} ] = [ u_{2} ]$ in $\underline{K} ( \mathfrak{A} )$, then $[ v( u_{1} ) ] = [ v( u_{2} ) ]$ in $\underline{K} ( \mathfrak{B} )$.

\item[(3)] If $p_{1}, p_{2}, p_{1} \oplus p_{2}$ are projections in $\mathcal{P} $, then $[ e( p_{1} \oplus p_{2} ) ] = [ e( p_{1} ) ] + [ e(p_{2}) ]$ in $\underline{K} ( \mathfrak{B} )$

\item[(4)] If $u_{1}, u_{2}, u_{1} \oplus u_{2}$ are unitaries in $\mathcal{P}$, then $[ v( u_{1} \oplus u_{2} ) ] = [ v( u_{1} ) ] +  [ v(u_{2}) ]$ in $\underline{K} ( \mathfrak{B} )$. 
\end{itemize}

Let $\overline{ \mathcal{P} }$ be the image of $\mathcal{P}$ in $\underline{K} ( \mathfrak{A} )$.  Define the \emph{\textbf{partial map on $\underline{K} ( \mathfrak{A} )$ induced by $\psi$,  $\ftn{ \underline{K} ( \psi ) \vert_{ \mathcal{P} } }{ \overline{ \mathcal{P} } }{ \underline{K} ( \mathfrak{B} ) }$,}} by 
\begin{align*}
\underline{K} ( \psi ) \vert_{ \mathcal{P} } ( x ) =
\begin{cases}
 [ e(p) ], &\text{if $x = [ p ]$ for some projection $p$ in $\mathcal{P}$} \\
[ v(u) ], &\text{if $x = [ u ]$ for some unitary $u$ in $\mathcal{P}$}
\end{cases}.
\end{align*}  
By Lemma \ref{l:partialmap} and the above remarks, we have that $\underline{K} (\psi ) \vert_{ \mathcal{P} }$ is a well-defined map on $\overline{ \mathcal{P} }$.  Moreover, as in Section 6.1.1 of \cite{booklin}, by enlarging $\mathcal{F}$ and decreasing the size of $\delta$ if necessary, $\underline{K}( \psi ) \vert_{ \mathcal{P} }$ can be extended to a well-defined group homomorphism from the sub-group $G( \overline{\mathcal{P}} )$ of $\underline{K} ( \mathcal{A} )$ generated by $\overline{ \mathcal{P} }$.  

Throughout the paper, we will abuse notation and denote this extension by $\underline{K}( \psi ) \vert_{ \mathcal{P} }$.
\end{definition}

\begin{remark}
\begin{itemize}
\item[(1)] Note that if $p \in \mathcal{P}$ and $[p] \in K_{i} ( \mathfrak{A} ; \Z_{n} )$, then $\underline{K} ( \psi ) \vert_{ \mathcal{P} } ( [ p ] )$ is an element of $K_{i} ( \mathfrak{B} ; \Z_{n} )$.  Hence, $\underline{K} ( \psi ) \vert_{ \mathcal{P } }$ is group homomorphism from $G( \overline{ \mathcal{P} } ) \cap K_{*} ( \mathfrak{A} ; \Z_{n} )$ to $K_{*} ( \mathfrak{B} ; \Z_{n} )$.  We will denote this homomorphism by $K_{*} ( \psi ; \Z_{n} ) \vert_{ \mathcal{P} }$.  When $n = 0$, we simply write $K_{*} ( \psi ) \vert_{ \mathcal{P} }$.

\item[(2)]  Throughout the paper, for $\mathcal{P} \subseteq \mathbb{P} ( \mathfrak{A} )$, when we write $\underline{K} ( \psi ) \vert_{ \mathcal{P} }$ we mean that $\ftn{ \psi }{ \mathfrak{A} }{ \mathfrak{B} }$ is a contractive, complete positive, linear map such that $\psi$ is $\mathcal{F}$-$\delta$-multiplicative for some $\delta > 0$ and some finite subset $\mathcal{F}$ of $\mathbb{P} ( \mathfrak{A} )$, so that we have a well-defined homomorphism from $G( \overline{ \mathcal{P} } )$ to $\underline{ K } ( \mathfrak{B} )$.  
\end{itemize}
\end{remark}

\begin{lemma}\label{l:closemaps}
Let $\mathfrak{A}$ be a nuclear $C^{*}$-algebra.  Let $\mathcal{P}$ be a finite subset of $\mathbb{P} ( \mathfrak{A} )$.  Then there exist a finite subset $\mathcal{F}$ of $\mathfrak{A}$ and $\delta > 0$ such that if $\mathfrak{B}$ is a $C^{*}$-algebra and $\ftn{ \psi_{1}, \psi_{2} }{ \mathfrak{A} }{ \mathfrak{B} }$ are contractive, completely positive, linear maps such that $\psi_{i}$ is $\mathcal{F}$-$\delta$-multiplicative and 
\begin{align*}
\norm{ \psi_{1} ( a ) -\psi_{2} ( a ) } < \delta
\end{align*}
for all $a \in \mathcal{F}$, then $\underline{K} ( \psi_{i} ) \vert_{ \mathcal{P} } $ is well-defined and 
\begin{align*}
\underline{K} ( \psi_{1} ) \vert_{ \mathcal{P} } = \underline{K} ( \psi_{2} ) \vert_{ \mathcal{P} }. 
\end{align*}
\end{lemma}

\begin{proof}
Choose a finite subset $\mathcal{F}_{1}$ of $\mathfrak{A}$ and $\delta_{1} > 0$ such that for any contractive, completely, positive, linear map $\ftn{ \psi }{ \mathfrak{A} }{ \mathfrak{B} }$ that is also $\mathcal{F}_{1}$-$\delta_{1}$-multiplicative, then $\underline{K} ( \psi ) \vert_{ \mathcal{P} }$ is well-defined.  Note that we can choose $\delta$ small enough with $\delta < \delta_{1}$ and a finite subset $\mathcal{F}$ of $\mathfrak{A}$ large enough with $\mathcal{F}_{1} \subseteq \mathcal{F}$ such that if $\ftn{ \psi }{ \mathfrak{A} }{ \mathfrak{B} }$ is a contractive, completely positive, linear map that is $\mathcal{F}$-$\delta$-multiplicative, then for each $p \in \mathcal{P}$, then 
\begin{align*}
\norm{ ( \psi \otimes \id_{ \mathsf{M}_{m} } \otimes \id_{C_{n} } )( p ) - x ( p ) } < \frac{1}{2}
\end{align*}
where $x ( p )$ is element of $\mathbb{P} ( \mathfrak{B} )$ given in Lemma \ref{l:partialmap}.  

Let $\ftn{ \psi_{1}, \psi_{2} }{ \mathfrak{A} }{ \mathfrak{B} }$ be contractive, completely positive, linear maps such that $\psi_{i}$ is $\mathcal{F}$-$\delta$-multiplicative.  Let $p \in \mathcal{P}$.  Let $x_{i} ( p )$ be the element in $\mathbb{P} ( \mathfrak{B} )$ given by Lemma \ref{l:partialmap}.  By the definition of $\underline{K} ( \psi_{i} ) \vert_{ \mathcal{P} }$, $\underline{K}( \psi_{i} ) \vert_{ \mathcal{P} } ( [ p ] ) = [ x_{i} ( p ) ]$.  Since
\begin{align*}
\norm{ ( \psi_{i} \otimes \id_{ \mathsf{M}_{m} } \otimes \id_{ C_{n} } )( p ) - x_{i} ( p ) } < \frac{1}{2}
\end{align*}
we have that 
\begin{align*}
\norm{ x_{1} ( p ) - x_{2} ( p ) } < 1.
\end{align*}
Therefore, $[ x_{1} ( p ) ] = [ x_{2} ( p ) ]$ in $\underline{K} ( \mathfrak{B} )$.  Hence, $\underline{K}( \psi_{1} ) \vert_{ \mathcal{P} } ( [ p ] ) = \underline{K}( \psi_{2} ) \vert_{ \mathcal{P} } ( [ p ] )$.
\end{proof}

\subsection{Bott Maps} We now define the Bott maps as in \cite{hl_asyunit}.  Let $\epsilon > 0$ and $\mathcal{F}$ be a finite subset of $\mathfrak{A} \otimes C(S^{1})$.  Then there exist $\delta > 0$ and a finite subset $\mathcal{G}$ of $\mathfrak{A}$ such that the following holds: if $\ftn{h}{ \mathfrak{A} }{ \mathfrak{B} }$ is a $*$-homomorphism and $u$ is a unitary in $\mathfrak{B}$ such that 
\begin{equation*}
\norm{  h ( a ) u - u h ( a ) } < \delta
\end{equation*}  
for all $a \in \mathcal{G}$, then the contractive, completely positive, linear map $\ftn{ \varphi_{h, u } }{ \mathfrak{A} \otimes C( S^{1} ) } { \mathfrak{B} }$ defined by $\varphi_{h,u} ( a \otimes f ) = h ( a ) f( u )$ is $\mathcal{F}$-$\epsilon$-multiplicative.  

Note that for every finite subset $\mathcal{F}$ of $\mathfrak{C}$ and $\epsilon > 0$, there exist $\delta > 0$ and a finite subset $\mathcal{G}$ of $\mathfrak{A}$ such that the following holds:  if $\ftn{ \psi }{ \mathfrak{C} }{ \mathfrak{A} \otimes C( S^{1} ) }$ and $\ftn{ h }{ \mathfrak{A} }{ \mathfrak{B} }$ are homomorphisms and $u$ is unitary in $\mathfrak{B}$ such that 
\begin{align*}
\norm{ h(a) u - u h(a) } < \delta
\end{align*}
for all $a \in \mathcal{G}$, then the contractive, completely positive, linear map $\ftn{ \varphi_{h, u } \circ \psi }{ \mathfrak{C} } { \mathfrak{B} }$ is $\mathcal{F}$-$\epsilon$-multiplicative.

\begin{definition}
Note that by the K{\"u}nneth Formula \cite{uct}, the embedding $\ftn{ j_{ \mathfrak{A} } }{ \mathfrak{A} \otimes C_{0} ( 0 , 1 ) } { \mathfrak{A} \otimes C( S^{1} ) }$ induces injective group homomorphisms
\begin{align*}
\ftn{ K_{0} ( j_{ \mathfrak{A} } ) }{ K_{0} ( \mathfrak{A} \otimes C_{0} ( 0 , 1 ) ) }{ K_{0} ( \mathfrak{A} \otimes C( S^{1} ) ) } \\
\ftn{ K_{1} ( j_{ \mathfrak{A} } ) }{ K_{1} ( \mathfrak{A} \otimes C_{0} ( 0 , 1 ) ) }{ K_{1} ( \mathfrak{A} \otimes C( S^{1} ) ) }.
\end{align*}
Using Bott periodicity to identify $K_{0} ( \mathfrak{A} \otimes C_{0} (0 , 1 ) )$ with $K_{1} ( \mathfrak{A} )$ and $K_{1} ( \mathfrak{A} \otimes C_{0} ( 0 , 1 ) )$ with $K_{0} ( \mathfrak{A} )$, we obtain injective group homomorphisms
\begin{align*}
\ftn{ \beta_{ \mathfrak{A} }^{(0)} }{ K_{0} ( \mathfrak{A} ) }{ K_{1} ( \mathfrak{A} \otimes C( S^{1} ) ) } \\
\ftn{ \beta_{ \mathfrak{A} }^{(1)} }{ K_{1} ( \mathfrak{A} ) }{ K_{0} ( \mathfrak{A} \otimes C( S^{1} ) ) }.
\end{align*}
Using Bott periodicity again, we obtain injective group homomorphisms
\begin{align*}
\ftn{ \beta_{ \mathfrak{A} , k }^{(0)} }{ K_{0} ( \mathfrak{A} ; \Z_{k} ) }{ K_{1} ( \mathfrak{A} \otimes C( S^{1} ) ; \Z_{k} ) } \\
\ftn{ \beta_{ \mathfrak{A}, k }^{(1)} }{ K_{1} ( \mathfrak{A} ; \Z_{k} ) }{ K_{0} ( \mathfrak{A} \otimes C( S^{1} ) ; \Z_{k} ) }.
\end{align*}

Let $\mathcal{P}$ be a finite subset of $\mathbb{P} ( \mathfrak{A} )$.  By the above remarks and Section \ref{maps}, there exist $\delta > 0$ and a finite subset $\mathcal{F}$ of $\mathfrak{A}$ such that if $\ftn{ h }{ \mathfrak{A} }{ \mathfrak{B} }$ is a $*$-homomorphism and $u$ is a unitary in $\mathfrak{B}$ with 
\begin{equation*}
\norm{ h( a ) u - u h ( a ) } < \delta
\end{equation*}
for all $a \in \mathcal{F}$, then $\underline{K} ( \varphi_{h,u} \circ j_{ \mathfrak{A} } ) \vert_{ \mathcal{P} }$ is well-defined, which we will denote as
\begin{equation*}
\mathrm{Bott} ( h , u ) \vert_{ \mathcal{P} }.
\end{equation*}
In particular, $\mathrm{bott}_{1} ( h, v ) \vert_{ \mathcal{P} } = K_{0} ( \varphi_{h,u} ) \circ \beta_{ \mathfrak{A} }^{(1)} \vert_{ \mathcal{P} }$ and $\mathrm{bott}_{0} ( h, v ) \vert_{ \mathcal{P} } = K_{1} ( \varphi_{h,u} ) \circ \beta_{ \mathfrak{A} }^{(0)} \vert_{ \mathcal{P} }$.  If $u h(a) = h(a) u$ for all $a \in \mathfrak{A}$, then $\mathrm{Bott} ( h, u )$ is well-defined on $\underline{K} ( \mathfrak{A} )$. 
\end{definition}

\begin{lemma}\label{l:prodbottmap}
Let $\mathfrak{A}$ be a unital $C^{*}$-algebra and let $\mathcal{P}$ be a finite subset of $\mathbb{P}( \mathfrak{A} )$.  Then there exist $\delta > 0$ and a finite subset $\mathcal{F}$ of $\mathfrak{A}$ such that the following holds:  if $\ftn{ h }{ \mathfrak{A} }{ \mathfrak{B} }$ is a $*$-homomorphism and $u_{1} , \dots, u_{n}$ are unitaries in $\mathfrak{B}$ with
\begin{equation*}
\norm{ h ( a ) u_{i} - u_{i} h(a) } < \frac{ \delta }{ n }
\end{equation*}
for all $a \in \mathcal{F}$ and for all $i$, then $\mathrm{Bott} ( h, u_{i} ) \vert_{ \mathcal{P} }$ and $\mathrm{Bott} ( h, u_{1} \cdots u_{n} ) \vert_{ \mathcal{P } }$ are well-defined and 
\begin{equation*}
\mathrm{Bott} ( h, u_{1} \cdots u_{n} ) \vert_{ \mathcal{P} } = \sum_{ i = 1}^{n} \mathrm{Bott} ( h, u_{i} ) \vert_{ \mathcal{P} }.
\end{equation*}

Moreover, if $\mathrm{Bott} ( h , u_{1} ) \vert_{ \mathcal{P} } = \mathrm{Bott} ( h, u_{2} ) \vert_{ \mathcal{P} }$, then $\mathrm{Bott} ( h, u_{1} u_{2}^{*} ) \vert_{ \mathcal{P} } = 0$.
\end{lemma}

\begin{proof}
Let $\delta > 0$ and $\mathcal{F}$ be a finite subset of $\mathfrak{A}$ such that if $\ftn{ h }{ \mathfrak{A} }{ \mathfrak{B} }$ is a $*$-homomorphism and $u$ is a unitary in $\mathfrak{B}$ such that 
\begin{align*}
\norm{ h(a) u - u h(a) } < \delta
\end{align*}
for all $a \in \mathcal{F}$, then $\mathrm{Bott} ( h, u ) \vert_{ \mathcal{P} }$ is well-defined.

Suppose $h$ is a $*$-homomorphism and $u_{1}, \dots, u_{n}$ are unitaries in $\mathfrak{B}$ such that $\norm{ h(a) u_{i} - u_{i} h(a) } < \frac{ \delta }{ n }$ for all $a \in \mathcal{F}$ and for all $i$.  Then 
\begin{align*}
&\norm{ h(a) u_{1} \cdots u_{n} - u_{1} \cdots u_{n} h(a) } \\
&\quad \leq \norm{ h(a) u_{1} \cdots u_{n} - u_{1} h(a) u_{2} \dots u_{n} } \\
&\qquad + \sum_{ i = 2}^{ n - 1 } \norm{ u_{1}u_{2} \cdots u_{i-1} h(a) u_{i} u_{i+1} \cdots u_{n} - u_{1} \cdots u_{i } h(a) u_{i+1} \cdots u_{n} }  \\
&\qquad + \norm{ u_{1} \cdots u_{n-1} h(a) u_{n} - u_{1} \cdots u_{n} h(a) } \\
&\quad = \sum_{ i = 1}^{n} \norm{ h(a) u_{i} - u_{i} h(a) } \\
&\quad < \delta
\end{align*}
for all $a \in \mathcal{F}$.  Hence, by the choice of $\delta$ and $\mathcal{F}$, $\mathrm{Bott} ( h, u_{i} ) \vert_{ \mathcal{P} }$ and $\mathrm{Bott} ( h, u_{1} \cdots u_{n} ) \vert_{ \mathcal{P } }$ are well-defined.  Moreover, enlarging $\mathcal{F}$ and decreasing the size of $\delta$ if necessary, we get that  
\begin{equation*}
\mathrm{Bott} ( h, u_{1} \cdots u_{n} ) \vert_{ \mathcal{P} } = \sum_{ i = 1}^{n} \mathrm{Bott} ( h, u_{i} ) \vert_{ \mathcal{P} }.
\end{equation*}
\end{proof}

\subsection{Some technical results}

\begin{lemma}\label{l:extmaplinear}
Let $\mathfrak{A}$ be a unital $C^{*}$-algebra.  Let $\epsilon > 0$, $\mathcal{F}_{1}$ be a finite subset of $\mathfrak{A}$, and $\mathcal{F}_{2}$ be a finite subset of $\mathfrak{A} \otimes C( S^{1} )$.  Then there exist $\delta > 0$, a finite subset $\mathcal{G}_{1}$ of $\mathfrak{A}$, and a finite subset $\mathcal{G}_{2}$ of $\mathfrak{A} \otimes C( S^{1} )$ such that the following holds: if $\ftn{ \psi }{ \mathfrak{A} \otimes C( S^{1} ) }{ \mathfrak{A} }$ is a contractive, completely positive, linear map and $u$ is a unitary in $\mathfrak{A}$ such that 
\begin{itemize}
\item[(1)] $\psi$ is $\mathcal{G}_{2}$-$\delta$-multiplicative;

\item[(2)] $\norm{ \psi ( a \otimes 1_{ C( S^{1} ) } ) - a } < \delta$ for all $a \in \mathcal{G}_{1}$; and 

\item[(3)] $\norm{ \psi ( 1_{ \mathfrak{A} } \otimes z ) - u } < \delta$
\end{itemize}
where $z$ is the function on the circle that sends $\xi$ to $\xi$, then 
\begin{itemize}
\item[(i)] $\norm{ u a - a u } < \epsilon$ for all $a \in \mathcal{F}_{1}$ and

\item[(ii)] $\norm{ \psi ( x ) - \varphi_{ \id_{ \mathfrak{A} } , u } ( x ) } < \epsilon$ for all $x \in \mathcal{F}_{2}$.
\end{itemize}
\end{lemma}

\begin{proof}
Set $m_{1} = \max\setof{ \| a \| }{ a \in \mathcal{F}_{1} }$ and set $\mathcal{G}_{1} = \mathcal{F}_{1} \cup \{ 1_{ \mathfrak{A} } \}$.  Note that $\setof{ a \otimes 1_{ C( S^{1} )  }}{ a \in \mathfrak{A} }$ and $\setof{1_{ \mathfrak{A} } \otimes z^{k} }{ k \in \Z }$ generate $\mathfrak{A} \otimes C( S^{1} )$.  Therefore, there exist $N \in \N$ and a finite subset of $\mathfrak{A}$, $\setof{ a_{i} \in \mathfrak{A} }{ i \in \Z , | i | \leq N }$, such that for each $x \in \mathcal{F}_{2}$, there exist a finite subset $\mathcal{S}_{x}$ of $\Z$ and functions $\ftn{ k_{x}, m_{x} }{ \mathcal{S}_{x} }{  \setof{ k \in \Z }{ | k | \leq N } }$ such that 
\begin{align*}
\norm{ x - \sum_{ i \in \mathcal{S}_{x} } a_{ k_{x}(i) } \otimes z^{m_{x}(i) } } < \frac{ \epsilon }{ 4 }.
\end{align*}

Define $M$ and $m_{2}$ as follows:
\begin{align*}
M = \max\setof{| S_{x} | }{ x \in \mathcal{F}_{2} } \quad \text{and} \quad m_{2} = \max \setof{ \| a_{i} \| }{ i \in \Z , | i | \leq N }.
\end{align*}
Set $\delta = \min \left\{ \frac{ \epsilon}{ 2m_{1} + 4 } , \frac{ \epsilon }{ 2M ( (2N-1)m_{2} +2 ) } \right\}.$  Let 
\begin{align*}
G_{2} &= \left(\setof{ a \otimes 1_{ C( S^{1} ) }}{ a \in \mathcal{F}_{1} } \cup 
\setof{ a_{i} \otimes 1_{ C( S^{1} ) }}{ i \in \Z , | i | \leq N } \right) \\
	&\quad  \cup \left(\setof{ a_{i} \otimes z^{j} }{ i, j \in \Z, | i | , | j | \leq N } \cup
\setof{ 1_{ \mathfrak{A} } \otimes z^{k} }{ k \in \Z , | k | \leq N } \right).
\end{align*}  
Note that for each $a \in \mathcal{F}_{1}$
\begin{equation}
\begin{aligned}
\norm{ u a - a u } &\leq \norm{ ( u - \psi ( 1_{ \mathfrak{A} } \otimes z ) ) a } \\
			&\qquad + \norm{ \psi ( 1_{ \mathfrak{A} } \otimes z ) ( a - \psi ( a \otimes 1_{ C( S^{1} ) } ) ) } \\
&\qquad + \norm{ \psi ( 1_{ \mathfrak{A} } \otimes z ) \psi ( a \otimes 1_{ C( S^{1} ) }) - \psi ( a \otimes z ) } \\
&\qquad + \norm{ \psi ( a \otimes z ) - \psi ( a \otimes 1_{ C( S^{1} ) })\psi ( 1_{ \mathfrak{A} } \otimes z ) } \\
&\qquad + \norm{ \psi ( a \otimes 1_{ C( S^{1} ) }) ( \psi ( 1_{ \mathfrak{A} } \otimes z ) - u ) } \\
&\qquad + \norm{ ( \psi ( a \otimes 1_{ C( S^{1} ) }) - a ) u } \\
&< ( 2 m_{1} + 4 ) \delta \\
&< \epsilon.
\end{aligned}
\end{equation}
Since for each $k \in \Z \setminus \{ 0 \}$ with $| k | \leq N$, 
\begin{align}
\norm{ \psi ( 1_{ \mathfrak{A} } \otimes z^{k} ) - u^{k} } &\leq ( 2 N - 1 ) \delta 
\end{align}
and
\begin{align}
\begin{aligned}
\norm{ \psi ( 1_{ \mathfrak{A} } \otimes z^{0}  ) - u^{0} } &= \norm{ \psi ( 1_{ \mathfrak{A} } \otimes 1_{ C( S^{1} ) } ) - 1_{ \mathfrak{A} } } \\
										&< \delta
\end{aligned}
\end{align}
we have that for each $i, k \in \Z$ with $| i | \leq N$ and $| k | \leq N$,
\begin{equation}
\begin{aligned}
\norm{ \psi ( a_{i} \otimes z^{k} ) - \varphi_{ \id_{ \mathfrak{A} } , u } ( a_{i} \otimes z^{k} ) } &= \norm{ \psi ( a_{i} \otimes z^{k} ) - a_{i}u^{k} } \\
&\leq \norm{ \psi ( a_{i} \otimes z^{k} ) - \psi ( a_{i} \otimes 1_{ C( S^{1} ) })\psi ( 1_{ \mathfrak{A} } \otimes z^{k} ) }  \\
&\qquad + \norm{ \psi ( a_{i} \otimes 1_{ C( S^{1} ) }) ( \psi ( 1_{ \mathfrak{A} } \otimes z^{k} ) - u^{k} ) } \\
&\qquad + \norm{ ( \psi ( a_{i} \otimes 1_{ C( S^{1} ) }) - a_{i} ) u^{k} } \\
&< ( (2 N -1) m_{2} + 2 ) \delta.
\end{aligned}
\end{equation}
Therefore, for each $x \in \mathcal{F}_{2}$,
{\small
\begin{align*}
\norm{ \psi ( x ) - \varphi_{ \id_{ \mathfrak{A} }, u } ( x )  } &\leq \norm{ \psi ( x ) - \psi \left( \sum_{ i \in \mathcal{S}_{x} } a_{ k_{x}(i) } \otimes z^{ m_{x}(i) } \right) } \\
&\qquad + \norm{ \psi \left( \sum_{ i \in \mathcal{S}_{x} } a_{ k_{x}(i) } \otimes z^{ m_{x}(i) } \right) - \varphi_{\id_{ \mathfrak{A} }, u } \left( \sum_{ i \in \mathcal{S}_{x} } a_{k_{x}(i)} \otimes z^{m_{x}(i)} \right) } \\
&\qquad + \norm{ \varphi_{\id_{ \mathfrak{A} }, u } \left( \sum_{ i \in \mathcal{S}_{x} } a_{ m_{x}(i) } \otimes z^{k_{x}(i)} \right) - \varphi_{ \id_{ \mathfrak{A} }, u } ( x ) } \\
&\leq \frac{ \epsilon }{ 2 } + M ( (2N-1)m_{2}+2) \delta \\
&< \epsilon.
\end{align*}
}
\end{proof}

The following lemma is the result of the proof of Lemma 4 of \cite{pr_auto}.  In \cite{pr_auto}, we did not state the lemma as below since we were not concerned with the bott maps.

\begin{lemma}\label{l:extendmap}
Let $\mathfrak{A}$ and $\mathfrak{B}$ be unital $C^{*}$-algebras satisfying the UCT.  Suppose $\ftn{ \varphi }{\underline{K} ( \mathfrak{A} ) }{ \underline{K} ( \mathfrak{B} ) }$ is a homomorphism such that $\varphi \vert_{ K_{0} ( \mathfrak{A} ) }$ is positive.  Suppose $\ftn{ \gamma }{ K_{0} ( \mathfrak{A} ) }{ K_{1} ( \mathfrak{B} ) }$ is a group homomorphism.  Then there exists $\ftn{ \alpha }{\underline{K} ( \mathfrak{A} \otimes C(S^{1}) ) }{ \underline{K} ( \mathfrak{B} ) }$ such that 
\begin{itemize}
\item[(1)] $\alpha \vert_{ K_{0} ( \mathfrak{A} \otimes C(S^{1}) ) }$ is positive;

\item[(2)] $\alpha \circ \underline{K} ( \iota_{ \mathfrak{A} } ) = \varphi$;

\item[(3)] $\alpha \circ \beta_{\mathfrak{A}}^{(1)} = 0$; and

\item[(4)] $\alpha \circ \beta_{ \mathfrak{A} }^{(0)} = \gamma$.
\end{itemize} 
\end{lemma}

\begin{lemma}\label{l:liftkthy}
Let $\mathfrak{A}$ be an infinite dimensional, simple, unital $C^{*}$-algebra satisfying the UCT.  Suppose $\mathfrak{A}$ is a tracially AI algebra.  Then for each $\epsilon > 0$, for each finite subset $\mathcal{P}$ of $\mathbb{P} ( \mathfrak{A} )$, and for each finite subset $\mathcal{F}$ of $\mathfrak{A} \otimes C(S^{1})$, there exists a finitely generated subgroup $\mathcal{G}$ of $K_{0} ( \mathfrak{A} )$ containing $[ 1_{ \mathfrak{A} } ]$ such that the following holds: for every unital, simple, tracially AI algebra $\mathfrak{B}$, if $\ftn{ \gamma }{ \mathcal{G} }{ K_{1} ( \mathfrak{B} ) }$ is a homomorphism and $\ftn{ \varphi }{ \mathfrak{A} }{ \mathfrak{B} }$ is a unital $*$-homomorphism, then there exists a contractive, completely positive, linear map $\ftn{ \psi }{ \mathfrak{A} \otimes C( S^{1} ) }{ \mathfrak{B} }$ such that 
\begin{itemize}
\item[(1)] $\psi$ is $\mathcal{F}$-$\epsilon$-multiplicative;

\item[(2)] $\underline{K} ( \psi \circ \iota_{ \mathfrak{A} } ) \vert_{ \mathcal{P} } = \underline{K} ( \varphi ) \vert_{ \mathcal{P} }$; 

\item[(3)] $K_{0} ( \psi )  \circ \beta_{ \mathfrak{A} }^{(1)} \vert_{ \mathcal{P} } = 0$; and

\item[(4)] $K_{1} ( \psi  )\circ  \beta_{ \mathfrak{A} }^{(0)} \vert_{ \mathcal{P} } = \gamma \vert_{ \mathcal{P} }$.
\end{itemize}  
\end{lemma}        
        
\begin{proof}
Since $\mathfrak{A}$ is an infinite dimensional tracially AI algebra satisfying the UCT, we may assume that 
\begin{equation*}
\mathfrak{A} = \dirlim ( \mathfrak{A}_{n} , \varphi_{ n , n+1} )
\end{equation*}
where $\mathfrak{A}_{n} = \bigoplus_{ i  = 1}^{ k(n) } P_{ [n,i] } \mathsf{M}_{ [n, i ] } ( C( X_{ [ n , i ] } ) ) P_{ [n,i] }$, each $X_{ [n,i]}$ is a connected finite CW-complex, and $\varphi_{ n , n +1}$ is a unital, $*$-monomorphism.

Choose $n_{0} \in \N$ large enough and choose finite subsets, $\mathcal{F}_{ n_{0} }$ of $\mathfrak{A}_{ n_{0} }$ and $\mathcal{P}_{n_{0}}$ of $\mathbb{P} ( \mathfrak{A}_{n_{0}} )$, such that every element of $\mathcal{F}$ is within $\frac{ \epsilon }{ 20 }$ to an element of $( \varphi_{ n_{0} } \otimes \id_{ C(S^{1} ) }) ( \mathcal{F}_{n_{0}} )$ and for every $p \in \mathcal{P}$, there exists $e_{p} \in \mathcal{P}_{n_{0}}$ such that $[ p ] = \underline{K} ( \varphi_{ n_{0} , \infty } ) ( [ e_{p} ] )$.  

Set $\mathcal{G} = K_{0} ( \varphi_{ n_{0} , \infty } )( K_{0} ( \mathfrak{A}_{ n_{0} } ) )$.  Let $\mathfrak{B}$ be a simple, unital, tracially AI algebra.  Suppose $\ftn{ \gamma }{ \mathcal{G} }{ K_{1} ( \mathfrak{B} ) }$ is a homomorphism and suppose $\ftn{ \varphi }{ \mathfrak{A} }{ \mathfrak{B} }$ is a unital $*$-monomorphism.  Then, by Lemma \ref{l:liftkthy}, there exists $\alpha \in \mathrm{Hom}_{ \Lambda } ( \underline{K} ( \mathfrak{A}_{ n_{0} } \otimes C(S^{1}) ) , \underline{K} ( \mathfrak{B} ) )$ such that $\alpha \vert_{ K_{0} ( \mathfrak{A}_{ n_{0} } \otimes C(S^{1}) ) }$ is a positive homomorphism, $\alpha \circ \underline{K} ( \iota_{ \mathfrak{A}_{ n_{0} } } ) = \underline{ K } ( \varphi \circ \varphi_{ n_{0} , \infty } )$, $\alpha \circ  \beta_{ \mathfrak{A}_{ n_{0} } }^{(1)} \vert_{ K_{1} ( \mathfrak{A}_{ n_{0} } ) } = 0$ and $\alpha \circ \beta_{ \mathfrak{A}_{ n_{0} } }^{(0)} \vert_{ K_{0} ( \mathfrak{A}_{ n _{0} } ) } = \gamma \circ K_{0} ( \varphi_{ n_{0} } )$.

By Theorem 9.12 of \cite{linTR1} and Theorem 5.4 of \cite{hl_asyunit}, there exists a nuclear, separable, simple, unital, tracially AF algebra $\mathfrak{B}'$ and an embedding $\ftn{ \varphi }{ \mathfrak{B}' }{ \mathfrak{B} }$ such that $\underline{K}(\varphi)$ is an isomorphism.  Composing the maps obtained from Proposition 9.10 of \cite{linTR1} and Theorem 6.2.9 of \cite{booklin} with $\varphi$, we get a sequence of unital, contractive, completely positive, linear maps $\{ \ftn{ L_{ n_{0} , k  }  }{ \mathfrak{A}_{ n_{0} } \otimes C(S^{1}) }{ \mathfrak{B} } \}_{ k = 1}^{ \infty }$ such that 
\begin{equation*}
\lim_{ k \to \infty } \norm{ L_{ n_{0} , k } ( xy ) - L_{ n_{0}, k } ( x ) L_{ n_{0} , k } (y) } = 0
\end{equation*}
for all $x,y \in \mathfrak{A}_{ n_{0} } \otimes C( S^{1} )$,
\begin{equation*}
\underline{ K }( L_{ n_{0} , k } \circ \iota_{ \mathfrak{A}_{ n_{0} } } ) \vert_{ \mathcal{P}_{ n_{0} } }  =  \alpha \circ \underline{K} ( \iota_{ \mathfrak{A}_{ n_{0} } } ) \vert_{ \mathcal{P}_{ n_{0} } }, 
\end{equation*}   
$K_{0} ( L_{ n_{0} , k } ) = \alpha \vert_{ K_{0} ( \mathfrak{ A}_{n_{0} }  \otimes C(S^{1} )  ) }$, and $K_{1} ( L_{ n_{0} , k } ) = \alpha \vert_{ K_{1} ( \mathfrak{ A}_{n_{0} }  \otimes C(S^{1} ) ) }$.

Since $\mathfrak{A}_{ n_{0} } \otimes C( S^{1} )$ is nuclear, there exists a sequence of contractive, completely positive, linear maps, $\{ \ftn{  \psi_{ n_{0} , k }  }{ \mathfrak{A} \otimes C(S^{1}) }{ \mathfrak{A}_{ n_{0} } \otimes C(S^{1}) } \}_{ k = 1}^{ \infty }$ such that 
\begin{equation*}
\lim_{ k \to \infty } \norm{ ( \psi_{ n_{0} , k } \circ ( \varphi_{ n_{0} , \infty } \otimes \id_{ C(S^{1}) } ) ) (x) - x } = 0
\end{equation*}
for all $x \in \mathfrak{A}_{ n_{0} } \otimes C(S^{1})$.  Set $\beta_{ n_{0} , k } = L_{n_{0} , k } \circ \psi_{ n _{0} , k }$ for large enough $k$, $\beta_{ n_{0} , k}$ satisfies the desired property.  Hence, set $\psi = \beta_{ n_{0} , k}$.
\end{proof}         

\begin{definition}
Let $\mathfrak{A}$ be a unital $C^{*}$-algebra such that $T( \mathfrak{A} ) \neq \emptyset$.  Denote the state space of $K_{0} ( \mathfrak{A} )$ by $S( K_{0} ( \mathfrak{A} ) )$.  The canonical map from $T( \mathfrak{A} )$ to $S( K_{0} ( \mathfrak{A} ) )$ which sends $\tau$ to the function $\{ [p] \mapsto \tau (p) \}$ will be denoted $r_{ \mathfrak{A} } ( \tau )( [ p ] ) = \tau ( p )$.       

Let $\mathfrak{A}$ and $\mathfrak{B}$ be unital $C^{*}$-algebras such that $T( \mathfrak{A} )$ and $T( \mathfrak{B} )$ are nonempty sets.  Let $\kappa \in \mathrm{Hom}_{ \Lambda } ( \underline{K} ( \mathfrak{A} ) , \underline{K} ( \mathfrak{B} ) )$ such that $\kappa ( [ 1_{ \mathfrak{A} } ] ) = [ 1_{ \mathfrak{B} } ]$ in $K_{0} ( \mathfrak{B} )$.  Then an affine map 
\begin{equation*}
\ftn{ \Delta }{ T( \mathfrak{B} ) }{ T( \mathfrak{A} ) }
\end{equation*}
is said to be \emph{\textbf{compatible to $\kappa$}} if the diagram
\begin{equation*}
\xymatrix{
T( \mathfrak{B} ) \ar[r]^-{ r_{ \mathfrak{B} } } \ar[d]_{ \Delta } & S( K_{0} ( \mathfrak{B} ) ) \ar[d]^{ \kappa_{0}^{S} } \\
T( \mathfrak{A} ) \ar[r]_-{ r_{ \mathfrak{A} } } &	S( K_{0} ( \mathfrak{A} ) ) 
}
\end{equation*}
is commutative, where $\kappa_{0}$ is the homomorphism from $K_{0}( \mathfrak{A} )$ to $K_{0} ( \mathfrak{B} )$ induced by $\kappa$ and $\kappa_{0}^{S} ( f ) = f \circ \kappa_{0}$.   
\end{definition}

\begin{lemma}\label{l:closetrace}
Let $\mathfrak{C} = \mathsf{M}_{n} ( C( X ) ) \otimes C( S^{1} )$ where $X$ is either $[0,1]$ or a space with one point and let $\mathfrak{A}$ be a simple, tracially AI algebra.  Suppose $\kappa \in \Hom_{ \Lambda } ( \underline{K} ( \mathfrak{C} ) , \underline{K} ( \mathfrak{A} ) )$ such that $\kappa ( K_{0} ( \mathfrak{C} )_{+} \setminus \{ 0 \} ) \subseteq K_{0} ( \mathfrak{A} )_{+} \setminus \{ 0 \}$ and $\kappa ( [ 1_{ \mathfrak{C} } ] ) = [ 1_{ \mathfrak{A} } ]$ and $\ftn{ \gamma }{ T( \mathfrak{A} ) }{ T( \mathfrak{C} ) }$ is a continuous affine map that is compatible to $\kappa$. 

Let $\epsilon > 0$ and $\mathcal{H}$ be a finite subset of $\mathfrak{C}_{ s.a.}$.  Then there exists a unital $*$-homomorphism $\ftn{h}{ \mathfrak{C} }{ \mathfrak{A} }$ such that 
\begin{itemize}
\item[(1)] $\underline{K}( h ) = \kappa$ and

\item[(2)] $\sup\setof{ | \tau \circ h( x ) - \gamma ( \tau ) ( x ) | }{ \tau \in T( \mathfrak{A} ) } < \epsilon$ for all $x \in \mathcal{H}$.
\end{itemize} 
\end{lemma}

\begin{proof}
By Theorem 9.12 of \cite{linTR1} and Theorem 5.4 of \cite{hl_asyunit}, there exists a unital, separable, nuclear, simple, tracially AF algebra $\mathfrak{B}$ and an embedding $\ftn{ \varphi }{ \mathfrak{B} }{ \mathfrak{A} }$ such that $\underline{K}(\varphi)$ is an isomorphism.  Hence, there exists $\kappa_{0} \in \Hom_{ \Lambda } ( \underline{K} ( \mathfrak{C} ) , \underline{K} ( \mathfrak{B} ) )$ such that $\kappa_{0} ( K_{0} ( \mathfrak{C} )_{ + } \setminus \{ 0 \} ) \subseteq K_{0} ( \mathfrak{A} )_{+} \setminus \{ 0 \}$, $\kappa_{0} ( [ 1_{ \mathfrak{C} } ] ) = [ 1_{ \mathfrak{B} } ]$, and $\kappa = \underline{K} ( \varphi ) \circ \kappa_{0}$.

Using Lemma 6.2 of \cite{hl_asyunit} and applying $\varphi$, there exist a projection $p_{0}$ in $\mathfrak{A}$, a finite dimensional sub-$C^{*}$-algebra $\mathfrak{F}$ of $\mathfrak{A}$ with $1_{ \mathfrak{F} } = 1_{ \mathfrak{A} } - p_{0}$ and unital $*$-homomorphisms $\ftn{ h_{0} }{ \mathfrak{C} }{ p_{0} \mathfrak{A} p_{0} }$ and $\ftn{ h_{1} }{ \mathfrak{C} }{ \mathfrak{F} }$ such that 
\begin{equation*}
\underline{K} ( h_{0} + h_{1} ) = \kappa \quad \mathrm{and} \quad \tau ( p_{0} ) < \frac{ \epsilon }{ 3 }
\end{equation*}
for all $\tau \in T( \mathfrak{A} )$.

Since $K_{1} ( h_{1} ) = 0$, by Lemma 9.5 of \cite{linTR1}, there is a sub-$C^{*}$-algebra $\mathfrak{B}_{0}$ of $( 1_{ \mathfrak{A}} - p_{0} ) \mathfrak{A} ( 1_{ \mathfrak{A} } - p_{0} )$ where $\mathfrak{B}_{0}$ is a finite direct sum of $C^{*}$-algebras of the form $\mathsf{M}_{k}$ and $\mathsf{M}_{n} ( C ( [ 0 , 1 ] ) )$ and there exists a unital $*$-homomorphism $\ftn{ h_{2} }{ \mathfrak{C} }{ \mathfrak{B}_{0} }$ such that
\begin{equation*}
K_{0} ( h_{2} ) = K_{0} ( h_{1} ) \quad \mathrm{and} \quad | \tau \circ h_{2} ( f ) - \tau ( 1_{ \mathfrak{A} } - p_{0} ) \gamma ( \tau (f) ) | < \frac{ \epsilon }{ 3 }
\end{equation*}
for all $f \in \mathcal{H}$ and for all $\tau \in T( \mathfrak{A} )$.

Set $h = h_{0} + h_{2}$.  Then $\ftn{h}{ \mathfrak{C} }{ \mathfrak{A} }$ is a unital $*$-homomorphism such that 
\begin{equation*}
\underline{K} ( h ) = \underline{K} ( h_{0} + h_{1} ) = \kappa
\end{equation*}
and
\begin{align*}
| \tau \circ h ( f ) - \gamma ( \tau )( f ) | &< | \tau \circ h ( f ) - \tau ( 1_{ \mathfrak{A} } - p_{0}  ) \gamma ( \tau )( f ) | + \frac{ \epsilon }{ 3 } \\
			&< \frac{ \epsilon }{ 3 } + | \tau \circ h_{2} ( f ) - \tau ( 1_{ \mathfrak{A} } - p_{0} ) \gamma ( \tau )( f ) | + \frac{  \epsilon }{ 3 } \\
			&< \epsilon
\end{align*}
for all $\tau \in T( \mathfrak{A} )$ and for all $f \in \mathcal{H}$.
\end{proof}

\begin{lemma}\label{l:lifttrace}
Let $\mathfrak{A}$ be as in Lemma \ref{l:liftkthy}.  For each $\epsilon > 0$ and for each finite subsets $\mathcal{P}$, $\mathcal{F}_{1}$, and $\mathcal{F}_{2}$ of $\mathbb{P} ( \mathfrak{A} )$, $\mathfrak{A}_{s.a.}$, and $\mathfrak{A} \otimes C( S^{1} )$ respectively, there exists a finitely generated subgroup $\mathcal{G}$ of $K_{0} ( \mathfrak{A} )$ containing $[ 1_{ \mathfrak{A} } ]$ such that the following holds: for every homomorphism $\ftn{ \gamma }{ \mathcal{G} }{ K_{1} ( \mathfrak{A} ) }$ there exists a unital, contractive, completely positive, linear map $\ftn{ \psi }{ \mathfrak{A} \otimes C( S^{1} ) }{ \mathfrak{A} }$ such that 
\begin{itemize}
\item[(1)] $\psi$ is $\mathcal{F}_{2}$-$\epsilon$-multiplicative;

\item[(2)] $\underline{K} ( \psi ) \circ \underline{ K } ( \iota_{ \mathfrak{A} } ) \vert_{ \mathcal{P} } = \underline{K} ( \id_{ \mathfrak{A} } ) \vert_{ \mathcal{P} }$;

\item[(3)] $\sup\setof{ | \tau \circ \psi \circ \iota_{ \mathfrak{A} } ( a ) - \tau ( a ) | }{ \tau \in T( \mathfrak{A} ) } < \epsilon$ for all $a \in \mathcal{F}_{1}$;

\item[(4)] $K_{0} ( \psi ) \circ \beta_{ \mathfrak{A} }^{(1)} \vert_{ \mathcal{P}  } = 0$; and 

\item[(5)] $K_{1} ( \psi ) \circ \beta_{ \mathfrak{A} }^{(0)} \vert_{ \mathcal{P} } = \gamma \vert_{ \mathcal{P} }$.
\end{itemize} 
\end{lemma}

\begin{proof}
Let $\{ \mathcal{H}_{n} \}_{ n = 1}^{ \infty }$ be an increasing sequence of finite subsets of $\mathfrak{A}$ whose union is dense in $\mathfrak{A}$.  Now, for each $n \in \N$, there exist a projection $p_{n} \in \mathfrak{A}$, a sub-$C^{*}$-algebra $\mathfrak{D}_{n} = \bigoplus_{ i = 1}^{ k(n) } \mathsf{M}_{ m( i,n ) } ( C ( X_{ [i,n] } ) )$ of $\mathfrak{A}$, where $X_{ [ i , n ] }$ is either $[0,1]$ or a space with one point with $1_{ \mathfrak{D}_{n} } = p_{n}$, and a sequence of contractive, completely positive, linear maps $\{ \ftn{ L_{n} }{ \mathfrak{A} }{ \mathfrak{D}_{n} } \}_{ n = 1}^{ \infty }$ such that 
\begin{itemize}
\item[(a)] $\norm{ p_{n} x - x p_{n} } < \frac{1}{ 2^{n} }$ for all $x \in \mathcal{H}_{n}$;

\item[(b)] $\norm{ p_{n} x p_{n} - L_{n} ( x ) } < \frac{ 1 }{ 2^{n} }$ for all $x \in \mathcal{H}_{n}$;

\item[(c)] $\norm{ x - ( 1_{ \mathfrak{A} } - p_{n} )x ( 1_{ \mathfrak{A} } - p_{n} ) - L_{n} ( x ) } < \frac{ 1 }{ 2^{n} }$ for all $x \in \mathcal{H}_{n}$ with $\norm{ x } \leq 1$; and 

\item[(d)] $\tau ( 1_{ \mathfrak{A} } - p_{n} ) < \frac{1}{ 2^{n} }$ for all $\tau \in T( \mathfrak{A} )$.
\end{itemize}
Note that 
\begin{equation*}
\lim_{ n \to \infty } \norm{ L_{n} ( xy ) - L_{n} ( x ) L_{n} (y) } = 0
\end{equation*}
for all $x, y \in \mathfrak{A}$.

Denote the $i^{\mathrm{th}}$ summand of $\mathfrak{D}_{n}$ by $\mathfrak{D}_{ [ n , i ] }$ and let $d_{ [ n , i ] } = 1_{ \mathfrak{D}_{ [ n , i ] } }$.  Choose $n$ large enough such that $\frac{1}{ 2^{n} } < \frac{ \epsilon }{ 6 }$.  Let $\mathcal{P}_{1}$ be a finite subset of $\mathbb{P} ( \mathfrak{A} )$ such that $\mathcal{P}_{1}$ contains $\mathcal{P}$, $d_{ [ n , i ] } $, $p_{n}$.  Choose a finite subset $\mathcal{F}_{3}$ of $\mathfrak{A} \otimes C(S^{1})$ such that $\mathcal{F}_{3}$ contains $\iota_{ \mathfrak{A} } ( \mathcal{F}_{1} ) \cup \mathcal{F}_{2}$ and the set
\begin{equation*}
\setof{ \left(  ( 1_{ \mathfrak{A} } - p_{n} ) \otimes 1_{ C( S^{1}) } \right) x \left(  ( 1_{ \mathfrak{A} } - p_{n} ) \otimes 1_{ C( S^{1}) } \right) }{ x \in \mathcal{F} }.
\end{equation*} 
Let $\mathcal{G}$ be the finitely generated subgroup of $K_{0} ( \mathfrak{A} )$ in Lemma \ref{l:liftkthy} which corresponds to $\frac{1}{ 2^{n} }$, $\mathcal{P}_{1}$, and $\mathcal{F}_{3}$.  

Suppose $\ftn{ \gamma }{ \mathcal{G} }{ K_{1} ( \mathfrak{A} ) }$ is a homomorphism.  Then, by Lemma \ref{l:liftkthy}, there exists a contractive, completely positive, linear map $\ftn{L}{ \mathfrak{A} \otimes C(S^{1}) }{ \mathfrak{A} }$ such that 
\begin{itemize}
\item[(a)] $L$ is $\mathcal{F}_{3}$-$\frac{1}{ 2^{n} }$-multiplicative;

\item[(b)] $\underline{K} ( L \circ \iota_{ \mathfrak{A} } ) \vert_{ \mathcal{P}_{1} } = \underline{K} ( \id_{ \mathfrak{A} } ) \vert_{ \mathcal{P}_{1} }$; 

\item[(c)] $K_{0} ( L  ) \circ \beta_{ \mathfrak{A} }^{(1)} \vert_{ \mathcal{P} } = 0$; and

\item[(d)] $K_{1} ( L ) \circ \beta_{ \mathfrak{A} }^{(0)} \vert_{ \mathcal{P} } = \gamma \vert_{ \mathcal{P} }$
\end{itemize}  
Choose a projection $q_{n} \in \mathfrak{A}$ such that $[ q_{n} ] = K_{0} ( L ) \left( \sum_{ i = 1}^{ k(n) } [ d_{ [n,i] } \otimes 1_{ C( S^{1} ) } ] \right)$.  Let $\mathcal{G}_{n}$ be a finite subset of $\mathfrak{D}_{ n }$ such that $\mathcal{G}_{n}$ contains the generators of $\mathfrak{D}_{n}$.  Define $\ftn{ \eta }{ T( \mathfrak{A} ) }{ T( \mathfrak{D}_{n} \otimes C(S^{1}) ) }$ by 
\begin{equation*}
\eta ( \tau ) = \frac{ 1 }{ \tau ( p_{n} ) } \tau \circ ( \id_{ \mathfrak{A} } \otimes \mathrm{ev} ) \vert_{ \mathfrak{D}_{n } \otimes C(S^{1}) }.
\end{equation*}
Since $K_{0} ( L ) = K_{0} ( \id_{ \mathfrak{A} } \otimes \mathrm{ev} )$, we have that $K_{0} \left( L \vert_{ \mathfrak{D}_{n}  \otimes C(S^{1} ) } \right)$ and $\gamma$ are compatible.  Also, note that $K_{0} ( \mathfrak{D}_{n} \otimes C(S^{1}) ) = K_{0} ( \iota_{ \mathfrak{D}_{n} }) ( K_{0} ( \mathfrak{D}_{n} ) )$.  Therefore $K_{0} \left( L \vert_{ \mathfrak{D}_{n}  \otimes C(S^{1}) }\right)$ sends $K_{0} ( \mathfrak{D}_{n} \otimes C( S^{1} ) )_{+} \setminus \{ 0 \}$ to $K_{0} ( q_{n} \mathfrak{A} q_{n} )_{+} \setminus \{ 0 \}$.  Hence, by Lemma \ref{l:closetrace}, there exists a unital $*$-homomorphism $\ftn{h}{ \mathfrak{D}_{n} \otimes C(S^{1}) }{ q_{n} \mathfrak{A} q_{n} }$ such that $\underline{K} ( h ) = \underline{K} \left( L \vert_{ \mathfrak{D}_{n}  \otimes C(S^{1}) } \right)$ and 
\begin{equation*}
\sup \setof{ | ( \tau \circ h )( g ) - \eta ( \tau ) ( g ) | }{ \tau \in T( \mathfrak{A} ) } < \frac{ 1 }{ 2^{n} }
\end{equation*}
for all $h \in \mathcal{G}_{n}$.  

Define $\ftn{ \psi }{ \mathfrak{A} \otimes C( S^{1} ) }{ \mathfrak{A} }$ by 
\begin{equation*}
\psi ( x ) = L \left( [ ( 1_{ \mathfrak{A} } - p_{n} ) \otimes 1_{ C(S^{1}) } ] x [ ( 1_{ \mathfrak{A} } - p_{n} ) \otimes 1_{ C(S^{1}) } ] \right) + \left( h \circ ( L_{n} \otimes \id_{ C( S^{1}) } ) \right) ( x ).
\end{equation*}
By construction, $\psi$ satisfies the desired properties of the lemma.
\end{proof}

\begin{theorem}\label{t:bottunit}
Let $\mathfrak{A}$ be as in Lemma \ref{l:liftkthy}.  For each $\epsilon > 0$ and for each finite subsets $\mathcal{P}$, $\mathcal{F}_{1}$, and $\mathcal{F}_{2}$ of $\mathbb{P} ( \mathfrak{A} )$, $\mathfrak{A}$, and $\mathfrak{A} \otimes C( S^{1} )$ respectively, there exists a finitely generated subgroup $\mathcal{G}$ of $K_{0} ( \mathfrak{A} )$ containing $[ 1_{ \mathfrak{A} } ]$ such that the following holds: for every homomorphism $\ftn{ \gamma }{ \mathcal{G} }{ K_{1} ( \mathfrak{A} ) }$ there exists a unital, contractive, completely positive, linear map $\ftn{ \psi }{ \mathfrak{A} \otimes C( S^{1} ) }{ \mathfrak{A} }$ such that 
\begin{itemize}
\item[(1)] $\psi$ is $\mathcal{F}_{2}$-$\epsilon$-multiplicative;

\item[(2)] $\underline{K} ( \psi ) \circ \underline{ K } ( \iota_{ \mathfrak{A} } ) \vert_{ \mathcal{P} } = \underline{K} ( \id_{ \mathfrak{A} } ) \vert_{ \mathcal{P} }$;

\item[(3)] $\norm{ ( \psi \circ \iota_{ \mathfrak{A} } ) ( a ) - a } < \epsilon$ for all $a \in \mathcal{F}_{1}$; 

\item[(4)] $K_{0} ( \psi ) \circ \beta^{(1)}_{ \mathfrak{A} } \vert_{ \mathcal{P}  } = 0$; and 

\item[(5)] $K_{1} ( \psi ) \circ \beta_{ \mathfrak{A} }^{(0)} \vert_{ \mathcal{P} } = \gamma \vert_{ \mathcal{P} }$.
\end{itemize} 
\end{theorem}

\begin{proof}
Arguing as in the proof of Theorem 4: The tracially AI case pp.\ 439 of \cite{pr_auto} and using Lemma \ref{l:lifttrace} instead of Lemma 7 of \cite{pr_auto}, we get the desired result.
\end{proof}

\begin{corollary}\label{c:bottunit}
Let $\mathfrak{A}$ be as in Lemma \ref{l:liftkthy}.  Let $\epsilon > 0$ and let  $\mathcal{F}$ and $\mathcal{P}$ be finite subsets of $\mathfrak{A}$ and $\mathbb{P} ( \mathfrak{A} )$ respectively.  Then there exists a finitely generated subgroup $\mathcal{G}$ of $K_{0} ( \mathfrak{A} )$ containing $[ 1_{ \mathfrak{A} } ]$ such that the following holds: if $\ftn{ \gamma }{ \mathcal{G} }{ K_{1} ( \mathfrak{A} ) }$, then there exists a unitary $w \in \mathfrak{A}$ such that
\begin{itemize}
\item[(1)] $\norm{ a w - w a } < \epsilon$ for all $a \in \mathcal{F}$;

\item[(2)] $\mathrm{bott}_{1} ( \id_{ \mathfrak{A} } , w ) \vert_{ \mathcal{P} } = 0$ and $\mathrm{bott}_{0} ( \id_{ \mathfrak{A} } , w ) \vert_{ \mathcal{P} } = \gamma \vert_{ \mathcal{P} }$; and

\item[(3)] $\gamma ( [ 1_{ \mathfrak{A} } ] ) = [ w ]$.
\end{itemize}
\end{corollary}

\begin{proof}
Let $\epsilon > 0$, $\mathcal{P}$ be a finite subset of $\mathbb{P} ( \mathfrak{A} )$, and $\mathcal{F}$ be a finite subset of $\mathfrak{A}$.  Set $\mathcal{Q} = j_{ \mathfrak{A} } ( \mathcal{P} )$.  By Lemma \ref{l:closemaps}, there exist $\delta_{1} > 0$ and a finite subset $\mathcal{H}$ of $\mathfrak{A} \otimes C( S^{1} )$ corresponding to $\mathcal{Q}$.  By Lemma \ref{l:extmaplinear}, there exist $\delta_{2} > 0$, a finite subset $\mathcal{G}_{1}$ of $\mathfrak{A}$, and finite subset $\mathcal{G}_{2}$ of $\mathfrak{A} \otimes C( S^{1} )$ corresponding to $\min\{ \delta_{1} , \epsilon \}$, $\mathcal{H}$, and $\mathcal{F}$.  By Lemma \ref{l:almostunit}, there exists $\delta_{3}$ corresponding $\min\{ \delta_{1} , \delta_{2} \}$.  By Theorem \ref{t:bottunit}, there exists a finitely generated subgroup $\mathcal{G}$ of $K_{0} ( \mathfrak{A} )$ containing $[1_{ \mathfrak{A} } ]$ corresponding to $\mathcal{P}$, $\mathcal{G}_{1} \cup \mathcal{F}$, $\mathcal{G}_{2} \cup \mathcal{H}$, and $\delta = \min\{ \epsilon , \delta_{1} , \delta_{2} , \delta_{3} \}$.  

Suppose $\ftn{ \gamma }{ \mathcal{G} }{ K_{1} ( \mathfrak{A} ) }$ is a group homomorphism.  Then there exists a unital, contractive, completely positive, linear map $\ftn{ \psi }{ \mathfrak{A} \otimes C( S^{1} ) }{ \mathfrak{A} }$ such that 
\begin{itemize}
\item[(a)] $\psi$ is $(\mathcal{G}_{2} \cup \mathcal{H})$-$\delta$-multiplicative; 

\item[(b)] $\norm{ \psi ( a \otimes 1_{ C(S^{1}) } ) - a } < \delta$ for all $a \in \mathcal{G}_{1} \cup \mathcal{F}$;

\item[(c)] $K_{0} ( \psi ) \circ \beta_{ \mathfrak{A} }^{ (1) } \vert_{ \mathcal{P} } = 0$; and 

\item[(d)] $K_{1} ( \psi ) \circ \beta_{ \mathfrak{A} }^{(0)} \vert_{ \mathcal{P} } = \gamma \vert_{ \mathcal{P} }$.
\end{itemize}

Since $\norm{ L( 1_{ \mathfrak{A} } \otimes z )^{*} L( 1_{ \mathfrak{A} } \otimes z ) - 1_{ \mathfrak{A} } } < \delta < \delta_{3}$ and $\norm{ L( 1_{ \mathfrak{A} } \otimes z ) L( 1_{ \mathfrak{A} } \otimes z )^{*} - 1_{ \mathfrak{A} } } < \delta < \delta_{3}$, there exists a unitary $w \in \mathfrak{A}$ such that 
\begin{equation*}
\norm{ \psi ( 1_{ \mathfrak{A} } \otimes z ) - w } < \delta_{2}.
\end{equation*}
Therefore, $\psi$ is $( \mathcal{G}_{2} \cup \mathcal{H} )$-$\delta_{2}$-multiplicative with
\begin{equation*}
\norm{ \psi ( a \otimes 1_{ C( S^{1} ) } ) - a } < \delta_{2} \quad \mathrm{and} \quad \norm{ \psi ( 1_{ \mathfrak{A} } \otimes z ) - w } < \delta_{2}
\end{equation*}
for all $a \in \mathcal{G}_{1} \cup \mathcal{F}$.  Hence, $\norm{ \psi ( x ) - \varphi_{ \id_{ \mathfrak{A} } , w } ( x ) } < \min \{ \epsilon , \delta_{1} \}$ for all $x \in \mathcal{H}$ and $\norm{ w a - a w } < \min\{ \epsilon , \delta_{1} \}$ for all $a \in \mathcal{F}$.  Thus, $\underline{K} ( \psi ) \vert_{ \mathcal{Q} } = \underline{K} ( \varphi_{ \id_{ \mathfrak{A} } , w } ) \vert_{ \mathcal{Q} }$.  So 
\begin{equation*}
\mathrm{bott}_{1} ( \id_{ \mathfrak{A} } , w ) \vert_{ \mathcal{P} } = K_{0} ( \varphi_{ \id_{ \mathfrak{A} } , w } ) \circ \beta_{ \mathfrak{A} }^{(1)} \vert_{ \mathcal{P} } = K_{0} ( \psi ) \circ \beta_{ \mathfrak{A} }^{(1)} \vert_{ \mathcal{P} } = 0
\end{equation*}
and 
\begin{equation*}
\mathrm{bott}_{0} ( \id_{ \mathfrak{A} } , w ) \vert_{ \mathcal{P} } = K_{1} ( \varphi_{ \id_{ \mathfrak{A} } , w } ) \circ \beta_{ \mathfrak{A} }^{(0)} \vert_{ \mathcal{P} } = K_{1} ( \psi ) \circ \beta_{ \mathfrak{A} }^{(0)} \vert_{ \mathcal{P} } = \gamma \vert_{ \mathcal{P} }.
\end{equation*}
Moreover,
\begin{equation*}
\gamma ( [ 1_{ \mathfrak{A} } ] ) = K_{1} ( \psi ) \circ \beta_{ \mathfrak{A} }^{(0)} ( [ 1_{ \mathfrak{A} } ] ) = K_{1} ( \psi )( [ 1_{ \mathfrak{A} } \otimes z ] ) = [ w ].
\end{equation*}
\end{proof}

\begin{definition}
Let $\mathcal{A}$ denote the class of all nuclear, separable, simple, unital $C^{*}$-algebras $\mathfrak{A}$ such that $\mathfrak{A} \otimes \mathsf{M}_{ \mathfrak{p} }$ is a tracially AI algebra that satisfies the UCT for all supernatural numbers $\mathfrak{p}$ of infinite type.  By Theorem 2.11 of \cite{hlzn_range}, $\mathfrak{A} \otimes \mathsf{M}_{ \mathfrak{p} }$ is a tracially AI algebra that satisfies the UCT for all supernatural numbers $\mathfrak{p}$ of infinite type if and only if $\mathfrak{A} \otimes \mathsf{M}_{ \mathfrak{p} }$ is a tracially AI algebra that satisfies the UCT for some supernatural number $\mathfrak{p}$ of infinite type.  

Let $\mathcal{A}_{ \mathcal{Z} }$ denote the class of $C^{*}$-algebras $\mathfrak{A}$ in $\mathcal{A}$ such that $\mathfrak{A}$ is $\mathcal{Z}$-stable.  Note that if $\mathfrak{A} \in \mathcal{A}_{ \mathcal{Z} }$, then $\mathfrak{A} \otimes \mathsf{M}_{ \mathfrak{p} } \in \mathcal{A}_{ \mathcal{Z} }$ for all supernatural numbers $\mathfrak{p}$.
\end{definition}

\begin{notation}
Let $\mathcal{Q}$ be the UHF algebra such that $K_{0} ( \mathcal{Q} ) = \Q$ and $[ 1_{ \mathcal{Q} } ] = 1$. 
\end{notation}

\begin{definition}
Let $\mathfrak{A}$ be a unital $C^{*}$-algebra and let $\mathcal{G}$ be a subgroup of $K_{0} ( \mathfrak{A} )$.  Then $H_{ [ 1_{ \mathfrak{A} } ] } ( \mathcal{G} , K_{1} ( \mathfrak{A} ) )$ denotes the subgroup of all $x \in K_{1} ( \mathfrak{A} )$ such that there exists a homomorphism $\alpha$ from $\mathcal{G}$ to $K_{1} ( \mathfrak{A} )$ with $\alpha ( [ 1_{ \mathfrak{A} } ] ) = x$.
\end{definition}

\begin{lemma}\label{l:torsionfreebott}
Let $\mathfrak{A}$ be nuclear, separable $C^{*}$-algebra satisfying the UCT such that $K_{*} ( \mathfrak{A} )$ is torsion free.  Let $\mathcal{P}$ be a finite subset of $\mathbb{P}( \mathfrak{A} )$.  Then there exist a finite subset $\mathcal{Q}$ of $\mathbb{P} ( \mathfrak{A} )$, $\delta > 0$ and a finite subset of $\mathcal{F}$ such that if $\mathfrak{B}$ is a nuclear, simple $C^{*}$-algebra satisfying the UCT and $\ftn{ \psi }{ \mathfrak{A} }{ \mathfrak{B} }$ is a contractive, completely positive, linear map with $\psi$ a $\mathcal{F}$-$\delta$-multiplicative map, then $K_{*} ( \psi ) \vert_{ \mathcal{Q} } = 0$ implies that $\underline{K} ( \psi ) \vert_{ \mathcal{P} } = 0$.
\end{lemma}

\begin{proof}
Let $\ftn{ \alpha }{ K_{*} ( \mathfrak{A} ) }{ \underline{K} ( \mathfrak{A} ) }$ be the natural homomorphism given in \cite{multcoeff}.  Since $K_{*} ( \mathfrak{A} )$ is torsion free, by \cite{multcoeff}, there exists a finite subset $\mathcal{P}_{0}$ of $\mathbb{P}_{0}( \mathfrak{A} )$ such that for each $p \in \mathcal{P}$, there exists $q \in \mathcal{P}_{0}$ such that $\alpha ( [ q ] ) = [ p ]$ in $\underline{K} ( \mathfrak{A} )$.  
 
Set $\mathcal{Q} = \mathcal{P}_{0} \cup \mathcal{P}$.  Choose $\delta > 0$ and a finite subset $\mathcal{F}$ of $\mathfrak{A}$ such that if $\mathfrak{B}$ is a nuclear, simple $C^{*}$-algebra satisfying the UCT and $\ftn{ \psi }{ \mathfrak{A} }{ \mathfrak{B} }$ is a contractive, completely positive, linear map with $\psi$ being a $\mathcal{F}$-$\delta$-multiplicative map, then $\underline{K}( \psi ) \vert_{ \mathcal{P} }$ and $\underline{K} ( \psi ) \vert_{ \mathcal{Q} }$ are well-defined.  

Suppose $\mathfrak{B}$ is a nuclear, simple $C^{*}$-algebra satisfying the UCT and $\ftn{ \psi }{ \mathfrak{A} }{ \mathfrak{B} }$ is a contractive, completely positive, linear map with $\psi$ being a $\mathcal{F}$-$\delta$-multiplicative map.  Define $\ftn{ \iota_{1} }{ \mathfrak{A} }{ \mathfrak{A} \otimes \mathcal{O}_{\infty} }$ by $\iota_{1} ( a ) = a \otimes 1_{ \mathcal{O}_{\infty}}$ and $\ftn{ \iota_{2}}{ \mathfrak{B} }{ \mathfrak{B} \otimes \mathcal{O}_{\infty} }$ by $\iota ( b ) = b \otimes 1_{ \mathcal{O}_{ \infty } }$.  By the K{\"u}nneth Formula \cite{uct}, $\underline{K} ( \iota_{i} )$ is an isomorphism.  Note that $\iota_{2} \circ \psi = ( \psi \otimes \id_{ \mathcal{O}_{\infty } } ) \circ \iota_{1} $ is a contractive, completely positive, linear map that is $\mathcal{F}$-$\delta$-multiplicative. 

By Theorem 6.7 of \cite{sepBDF}, there exists a homomorphism $\ftn{ \varphi }{ \mathfrak{A} \otimes \mathcal{O}_{\infty} }{ \mathfrak{B} \otimes \mathcal{O}_{\infty} }$ such that 
\begin{align*}
\underline{K} ( \varphi ) \circ \underline{K} ( \iota_{1} ) \vert_{ \mathcal{Q} } = \underline{ K} ( ( \psi \otimes \id_{ \mathcal{O}_{\infty } } ) \circ \iota_{1} ) \vert_{ \mathcal{Q} } =  \underline{K} ( \iota_{2} ) \circ \underline{K}( \psi ) \vert_{ \mathcal{Q} }.
\end{align*}
Let $\ftn{ \beta_{1} }{ K_{*} ( \mathfrak{A} \otimes \mathcal{O}_{\infty} ) }{ \underline{K} ( \mathfrak{A} \otimes \mathcal{O}_{\infty}) }$ and $\ftn{ \beta_{2} }{ K_{*} ( \mathfrak{B} \otimes \mathcal{O}_{\infty} ) }{ \underline{K} ( \mathfrak{B} \otimes \mathcal{O}_{\infty}) }$ be the natural homomorphisms given in \cite{multcoeff}.  Suppose $K_{*} ( \psi ) \vert_{ \mathcal{Q} } = 0$.  Then $K_{*} ( \varphi ) \circ K_{*} ( \iota_{1} ) \vert_{ \mathcal{Q} } = 0$.  Let $p \in \mathcal{P}$.  Then there exists $q \in \mathcal{P}_{0}$ such that $ \alpha( [ q ] ) = [ p ]$ in $\underline{K} ( \mathfrak{A} )$.  Then 
\begin{align*}
 \underline{K} ( \iota_{2} ) ( \underline{K}( \psi )\vert_{ \mathcal{P} } ([p]) )&=\underline{ K} ( \varphi \circ \iota_{1} )( [ p ] ) \\
 								&= \underline{ K} ( \varphi \circ \iota_{1} )( \alpha( [ q ] ) ) \\
								&= \beta_{2} ( K_{*} ( \varphi \circ \iota_{1} )  ( [ q ] ) ) \\
								&= 0.
\end{align*}
Since $\underline{K} ( \iota_{2} )$ is an isomorphism, $\underline{K} ( \psi ) \vert_{ \mathcal{P} }( [ p ] ) =0$.
\end{proof} 

\begin{theorem}\label{t:approxcomm}
Let $\mathfrak{A}$ be in $\mathcal{A}_{ \mathcal{Z} }$.  Let $\epsilon > 0$ and $\mathcal{F}$ be a finite subset of $\mathfrak{A}$.  Then there exists a finitely generated subgroup $\mathcal{G}$ of $K_{0} ( \mathfrak{A} )$ such that the following holds:  if $[ u ] \in H_{ [ 1_{ \mathfrak{A} } ] } ( \mathcal{G} , K_{1} ( \mathfrak{A} ) )$, then there exists a continuous path of unitaries $w(t)$ in $\mathfrak{A} \otimes \mathsf{M}_{ \mathfrak{p} } \otimes \mathsf{M}_{ \mathfrak{q} }$ such that 
\begin{itemize}
\item[(1)] $w(0) \in \mathfrak{A} \otimes \mathsf{M}_{ \mathfrak{p} } \otimes 1_{ \mathsf{M}_{ \mathfrak{q} } }$ and $w(1) \in  \mathfrak{A} \otimes 1_{ \mathsf{M}_{ \mathfrak{p} } }  \otimes \mathsf{M}_{ \mathfrak{q} }$;

\item[(2)] $[ w(0) ] = [ u \otimes 1_{ \mathsf{M}_{ \mathfrak{p} } } \otimes 1_{ \mathsf{M}_{ \mathfrak{q} } } ]$ in $K_{1} ( \mathfrak{A} \otimes \mathsf{M}_{ \mathfrak{p} } \otimes 1_{ \mathsf{M}_{ \mathfrak{q} } } )$ and $[ w(1) ] = [ u \otimes 1_{ \mathsf{M}_{ \mathfrak{p} } } \otimes 1_{ \mathsf{M}_{ \mathfrak{q} } } ]$ in $K_{1} ( \mathfrak{A} \otimes 1_{ \mathsf{M}_{ \mathfrak{p} } }  \otimes \mathsf{M}_{ \mathfrak{q} } )$; and 

\item[(3)] $\norm{ w(t) ( a \otimes 1_{ \mathsf{M}_{ \mathfrak{p} } } \otimes 1_{ \mathsf{M}_{ \mathfrak{q} } }  ) - ( a \otimes 1_{ \mathsf{M}_{ \mathfrak{p} } } \otimes 1_{ \mathsf{M}_{ \mathfrak{q} } }  ) w(t) } < \epsilon$ for all $a \in \mathcal{F}$ and $t \in [ 0 , 1 ]$,
\end{itemize}
where $\mathfrak{p}$ and $\mathfrak{q}$ are supernatural numbers of infinite type with $\mathsf{M}_{ \mathfrak{p} } \otimes \mathsf{M}_{ \mathfrak{q} }$ isomorphic $\mathcal{Q}$.
\end{theorem}

\begin{proof}
Let $\{ \mathcal{F}_{1,n} \}_{ n = 1}^{ \infty }$ be an increasing sequence of finite subsets of $\mathfrak{A}$ such that $\bigcup_{ n = 1}^{ \infty } \mathcal{F}_{1,n}$ is dense in $\mathfrak{A}$, $\{ \mathcal{F}_{2,n} \}_{ n = 1}^{ \infty }$ be an increasing sequence of finite subsets of $\mathsf{M}_{ \mathfrak{p} }$ such that $\bigcup_{ n = 1}^{ \infty } \mathcal{F}_{2,n}$ is dense in $\mathsf{M}_{ \mathfrak{p} }$, $\{ \mathcal{F}_{3,n} \}_{ n = 1}^{ \infty }$ be an increasing sequence of finite subsets of $\mathsf{M}_{ \mathfrak{q} }$ such that $\bigcup_{ n = 1}^{ \infty } \mathcal{F}_{3,n}$ is dense of $\mathsf{M}_{ \mathfrak{q} }$.  Let $\{ \mathcal{P}_{1,n} \}_{ n = 1}^{ \infty }$ be an increasing sequence of finite subsets of $\mathbb{P} ( \mathfrak{A} )$, $\{ \mathcal{P}_{2,n} \}_{ n = 1}^{ \infty }$ be an increasing sequence of finite subsets of $\mathbb{P} ( \mathsf{M}_{ \mathfrak{p} } )$, and $\{ \mathcal{P}_{ 3, n } \}_{ n = 1}^{\infty}$ be an increasing sequence of finite subsets of $\mathbb{P} ( \mathsf{M}_{ \mathfrak{q} } )$ such that 
\begin{align*}
\bigcup_{ n = 1}^{ \infty } \mathcal{P}_{1,n} = \mathbb{P} ( \mathfrak{A} ), \
\bigcup_{ n = 1}^{ \infty } \mathcal{P}_{2,n} = \mathbb{P} ( \mathsf{M}_{ \mathfrak{p} } ), \ \text{and} \ 
\bigcup_{ n = 1}^{ \infty } \mathcal{P}_{3,n} = \mathbb{P} ( \mathsf{M}_{ \mathfrak{q} } ).
\end{align*}

For each $n$, let $\mathcal{G}_{1,n}$ be the finitely generated subgroup of $K_{0} ( \mathfrak{A} \otimes \mathsf{M}_{ \mathfrak{p} } )$ provided by Corollary \ref{c:bottunit} that corresponds to $\mathcal{F}_{1, n} \otimes \mathcal{F}_{2,n }$, $\mathcal{P}_{1,n} \otimes \mathcal{P}_{2,n}$, and $\frac{1}{ 2^{n} }$ and let $\mathcal{G}_{2,n}$ be the finitely generated subgroup of $K_{0} ( \mathfrak{A} \otimes \mathsf{M}_{ \mathfrak{q} } )$ provided by Corollary \ref{c:bottunit} that corresponds to $\mathcal{F}_{1,n} \otimes \mathcal{F}_{3,n}$, $\mathcal{P}_{1,n} \otimes \mathcal{P}_{3,n}$, and $\frac{1}{ 2^{n} }$.  Note that we may assume that $\mathcal{G}_{1,n} \subseteq \mathcal{G}_{1,n+1}$, $\mathcal{G}_{2,n} \subseteq \mathcal{G}_{2,n+1}$, $\bigcup_{ n = 1}^{ \infty } \mathcal{G}_{1,n} = K_{0} ( \mathfrak{A} \otimes \mathsf{M}_{ \mathfrak{p} } )$, and $\bigcup_{ n = 1}^{ \infty } \mathcal{G}_{2,n} = K_{0} ( \mathfrak{A} \otimes \mathsf{M}_{ \mathfrak{q} } )$.

Let $\{ \mathcal{G}_{n} \}_{ n = 1}^{ \infty }$ be an increasing sequence of finitely generated subgroups of $K_{0} ( \mathfrak{A} )$ containing $[ 1_{ \mathfrak{A} } ]$ such that if $\sum_{ i = 1}^{n} x_{i} \otimes y_{i}$ is a generator of $\mathcal{G}_{1,n}$, then $x_{i} \in \mathcal{G}_{n}$ and if $\sum_{ i = 1}^{n} x_{i} \otimes y_{i}$ is a generator of $\mathcal{G}_{2,n}$, then $x_{i} \in \mathcal{G}_{n}$.  (Note that we are identifying $K_{0} ( \mathfrak{A} \otimes \mathsf{M}_{ \mathfrak{p} } )$ with $K_{0} ( \mathfrak{A} ) \otimes K_{0} ( \mathsf{M}_{ \mathfrak{q} } )$ and $K_{0} ( \mathfrak{A} \otimes \mathsf{M}_{ \mathfrak{q} } )$ with $K_{0} ( \mathfrak{A} ) \otimes K_{0} (\mathsf{M}_{ \mathfrak{q} }).)$

Since $K_{0} ( \mathfrak{A} \otimes \mathsf{M}_{ \mathfrak{p} } \otimes \mathsf{M}_{ \mathfrak{q} } )$ and $K_{1} ( \mathfrak{A} \otimes \mathsf{M}_{ \mathfrak{p} } \otimes \mathsf{M}_{ \mathfrak{q} } )$ are torsion free groups, by Theorem 8.4 of \cite{hl_homotopyunit} and Lemma \ref{l:torsionfreebott}, there exist $\delta > 0$, a finite subset $\mathcal{G}$ of $\mathfrak{A} \otimes \mathsf{M}_{ \mathfrak{p} } \otimes \mathsf{M}_{ \mathfrak{q} }$, and a finite subset $\mathcal{H}$ of $\mathbb{P} ( \mathfrak{A} \otimes \mathsf{M}_{ \mathfrak{p} } \otimes \mathsf{M}_{ \mathfrak{q} } )$ such that if $v$ is a unitary in $\mathfrak{A} \otimes \mathsf{M}_{ \mathfrak{p} } \otimes \mathsf{M}_{ \mathfrak{q} }$ with 
\begin{equation*}
\norm{ x v - v x } < \delta  \ \mbox{for all} \ x \in \mathcal{G} 
\end{equation*}
and
\begin{equation*}
 \mathrm{bott}_{0} ( \id_{ \mathfrak{A} \otimes \mathsf{M}_{ \mathfrak{p} } \otimes \mathsf{M}_{ \mathfrak{q} } } , v ) \vert_{ \mathcal{H} } = 0 \ \mathrm{and} \ \mathrm{bott}_{1} ( \id_{ \mathfrak{A} \otimes \mathsf{M}_{ \mathfrak{p} } \otimes \mathsf{M}_{ \mathfrak{q} } } , v ) \vert_{ \mathcal{H} } = 0
\end{equation*}
then there exists a continuous path of unitaries $v(t)$ in $\mathfrak{A} \otimes \mathsf{M}_{ \mathfrak{p} } \otimes \mathsf{M}_{ \mathfrak{q} }$ such that $v(0) = 1_{  \mathfrak{A} \otimes \mathsf{M}_{ \mathfrak{p} } \otimes \mathsf{M}_{ \mathfrak{q} } }$, $v(1) = v$, and 
\begin{equation*}
\norm{ x v(t) - v(t) x } < \frac{ \epsilon }{ 2 }
\end{equation*}
for all $x \in \mathcal{F} \otimes 1_{ \mathsf{M}_{ \mathfrak{p} } } \otimes 1_{ \mathsf{M}_{ \mathfrak{q} } }$ and $t \in [0,1]$.

Choose $n$ large enough such that $\frac{1}{ 2^{n-1} } < \frac{ \epsilon }{ 2 }$ and if $v$ is a unitary in $\mathfrak{A} \otimes \mathsf{M}_{ \mathfrak{p} } \otimes \mathsf{M}_{ \mathfrak{q} }$ with 
\begin{equation*}
\norm{ v x - x v } < \frac{1}{ 2^{n-1} }
\end{equation*}
for all $x \in \mathcal{F}_{1,n} \otimes \mathcal{F}_{2,n} \otimes \mathcal{F}_{3,n}$, then
\begin{equation*}
\norm{ v x - x v } < \frac{ \delta }{ 2 }
\end{equation*}
for all $x \in \mathcal{G}$ and if
\begin{equation*}
 \mathrm{bott}_{0} ( \id_{ \mathfrak{A} \otimes \mathsf{M}_{ \mathfrak{p} } \otimes \mathsf{M}_{ \mathfrak{q} } } , v ) \vert_{ \mathcal{P}_{1,n} \otimes \mathcal{P}_{2,n} \otimes \mathcal{P}_{3,n} } = 0 \ \mathrm{and} \ \mathrm{bott}_{1} ( \id_{ \mathfrak{A} \otimes \mathsf{M}_{ \mathfrak{p} } \otimes \mathsf{M}_{ \mathfrak{q} } } , v ) \vert_{ \mathcal{P}_{1,n} \otimes \mathcal{P}_{2,n} \otimes \mathcal{P}_{3,n} } = 0
\end{equation*}
then 
\begin{equation*}
 \mathrm{bott}_{0} ( \id_{ \mathfrak{A} \otimes \mathsf{M}_{ \mathfrak{p} } \otimes \mathsf{M}_{ \mathfrak{q} } } , v ) \vert_{ \mathcal{H} } = 0 \ \mathrm{and} \ \mathrm{bott}_{1} ( \id_{ \mathfrak{A} \otimes \mathsf{M}_{ \mathfrak{p} } \otimes \mathsf{M}_{ \mathfrak{q} } } , v ) \vert_{ \mathcal{H} } = 0.
\end{equation*}
Note that the last part of the statement can be done since we are identifying $K_{i} ( \mathfrak{A} \otimes \mathsf{M}_{ \mathfrak{p} } \otimes \mathsf{M}_{ \mathfrak{q} } )$ with $K_{i} ( \mathfrak{A} ) \otimes K_{0} ( \mathsf{M}_{ \mathfrak{p} } ) \otimes K_{0} ( \mathsf{M}_{ \mathfrak{q} } )$. 

Let $[ u ] \in H_{ [ 1_{ \mathfrak{A} } ] } ( \mathcal{G}_{n} , K_{1} ( \mathfrak{A} ) )$.  Then there exists $\ftn{ \gamma }{ \mathcal{G}_{n} }{ K_{1} ( \mathfrak{A} ) }$ such that $\gamma( [ 1_{ \mathfrak{A} } ] ) = [ u ]$.  By Corollary \ref{c:bottunit}, there exists a unitary $v_{0} \in \mathfrak{A} \otimes \mathsf{M}_{ \mathfrak{p} } $ such that 
\begin{equation*}
\norm{ v_{0} ( x_{1} \otimes x_{2} ) - ( x_{1} \otimes x_{2}) v_{0} } < \frac{1}{ 2^{n} } 
\end{equation*}  
for all $x_{1} \in \mathcal{F}_{1,n}$ and $x_{2} \in \mathcal{F}_{2,n}$,
\begin{equation*}
\mathrm{bott}_{1} ( \id_{ \mathfrak{A} \otimes \mathsf{M}_{ \mathfrak{p} } } , v_{0} ) \vert_{ \mathcal{P}_{1,n} \otimes \mathcal{P}_{2,n} } = 0,
\end{equation*}
and 
\begin{equation*}
\mathrm{bott}_{0} ( \id_{ \mathfrak{A}  \otimes \mathsf{M}_{ \mathfrak{p} } } , v_{0} ) \vert_{ \mathcal{P}_{1,n} \otimes \mathcal{P}_{2,n} } = \gamma \otimes K_{0} ( \id_{ \mathsf{M}_{ \mathfrak{p} } } )  \vert_{ \mathcal{P}_{1,n} \otimes \mathcal{P}_{2,n} }.
\end{equation*}

By Corollary \ref{c:bottunit}, there exists a unitary $v_{1} \in \mathfrak{A} \otimes \mathsf{M}_{ \mathfrak{q} }$ such that 
\begin{equation*}
\norm{ v_{1} ( x_{1}  \otimes x_{3} ) - ( x_{1}  \otimes x_{3} ) v_{1} } < \frac{1}{ 2^{n} } 
\end{equation*}  
for all $x_{1} \in \mathcal{F}_{1,n}$ and $x_{3} \in \mathcal{F}_{3,n}$,
\begin{equation*}
\mathrm{bott}_{1} ( \id_{ \mathfrak{A} \otimes \mathsf{M}_{ \mathfrak{q} } } , v_{1} ) \vert_{ \mathcal{P}_{1,n}  \otimes \mathcal{P}_{3,n} } = 0,
\end{equation*}
and 
\begin{equation*}
\mathrm{bott}_{0} ( \id_{ \mathfrak{A}  \otimes \mathsf{M}_{ \mathfrak{q} } }, v_{1} ) \vert_{ \mathcal{P}_{1,n}  \otimes \mathcal{P}_{3,n}} = \gamma \otimes K_{0} ( \id_{ \mathsf{M}_{ \mathfrak{q} } } ) \vert_{ \mathcal{P}_{1,n} \otimes \mathcal{P}_{3,n}}.
\end{equation*}

Set $w_{0} = v_{0} \otimes 1_{ \mathsf{M}_{ \mathfrak{q} } } \in \mathfrak{A} \otimes \mathsf{M}_{ \mathfrak{p} } \otimes 1_{ \mathsf{M}_{ \mathfrak{q} } }$ and set $w_{1} = \id^{ [ 1,3,2 ] } ( v_{1} \otimes 1_{ \mathsf{M}_{ \mathfrak{p} } } ) \in \mathfrak{A} \otimes 1_{ \mathsf{M}_{ \mathfrak{p} } } \otimes \mathsf{M}_{ \mathfrak{q} }$.  Then
\begin{equation*}
\norm{ x w_{i} - w_{i} x } < \frac{1}{ 2^{n} }
\end{equation*}
for all $x \in \mathcal{F}_{1,n} \otimes \mathcal{F}_{2,n} \otimes \mathcal{F}_{3,n}$, for each $i = 0,1$,
\begin{equation*}
\mathrm{bott}_{1} ( \id_{ \mathfrak{A} \otimes \mathsf{M}_{ \mathfrak{p} } \otimes \mathsf{M}_{ \mathfrak{q} } }  , w_{i} ) \vert_{ \mathcal{P}_{1,n} , \mathcal{P}_{2,n} \otimes \mathcal{P}_{3,n} } = 0 
\end{equation*}
and 
\begin{equation*}
\mathrm{bott}_{0} ( \id_{ \mathfrak{A} \otimes \mathsf{M}_{ \mathfrak{p} } \otimes \mathsf{M}_{ \mathfrak{q} } }  , w_{i} ) \vert_{ \mathcal{P}_{1,n} \otimes \mathcal{P}_{2,n} \otimes \mathcal{P}_{3,n} } = \gamma \otimes K_{0} ( \id_{ \mathsf{M}_{ \mathfrak{p} } } ) \otimes K_{0} ( \id_{ \mathsf{M}_{ \mathfrak{q} } } )   \vert_{ \mathcal{P}_{1,n} \otimes \mathcal{P}_{2,n} \otimes \mathcal{P}_{3,n} }.
\end{equation*}

Note that 
\begin{equation*}
\norm{ x w_{0}^{*} w_{1} - w_{0}^{*} w_{1} x } < \frac{1}{ 2^{n-1} }
\end{equation*}
for all $x \in \mathcal{F}_{1,n} \otimes \mathcal{F}_{2,n} \otimes \mathcal{F}_{3,n}$.  By Lemma 3.11, 
\begin{align*}
&\mathrm{bott}_{0} ( \id_{ \mathfrak{A} \otimes \mathsf{M}_{ \mathfrak{p} } \otimes \mathsf{M}_{ \mathfrak{q} } }  , w_{0}^{*} w_{1} ) \vert_{ \mathcal{P}_{1,n} \otimes \mathcal{P}_{2,n} \otimes \mathcal{P}_{3,n} }  \\
&\qquad = - \mathrm{bott}_{0} ( \id_{ \mathfrak{A} \otimes \mathsf{M}_{ \mathfrak{p} } \otimes \mathsf{M}_{ \mathfrak{q} } }  , w_{0} ) \vert_{ \mathcal{P}_{1,n} \otimes \mathcal{P}_{2,n} \otimes \mathcal{P}_{3,n} } + \mathrm{bott}_{0} ( \id_{ \mathfrak{A} \otimes \mathsf{M}_{ \mathfrak{p} } \otimes \mathsf{M}_{ \mathfrak{q} } }  , w_{1} ) \vert_{ \mathcal{P}_{1,n} \otimes \mathcal{P}_{2,n} \otimes \mathcal{P}_{3,n} }   \\
&\qquad = 0 
\end{align*}
\begin{align*}
&\mathrm{bott}_{1} ( \id_{ \mathfrak{A} \otimes \mathsf{M}_{ \mathfrak{p} } \otimes \mathsf{M}_{ \mathfrak{q} } }  , w_{0}^{*} w_{1} ) \vert_{ \mathcal{P}_{1,n} \otimes \mathcal{P}_{2,n} \otimes \mathcal{P}_{3,n} }  \\
&\qquad = - \mathrm{bott}_{1} ( \id_{ \mathfrak{A} \otimes \mathsf{M}_{ \mathfrak{p} } \otimes \mathsf{M}_{ \mathfrak{q} } }  , w_{0} ) \vert_{ \mathcal{P}_{1,n} \otimes \mathcal{P}_{2,n} \otimes \mathcal{P}_{3,n} } + \mathrm{bott}_{1} ( \id_{ \mathfrak{A} \otimes \mathsf{M}_{ \mathfrak{p} } \otimes \mathsf{M}_{ \mathfrak{q} } }  , w_{1} ) \vert_{ \mathcal{P}_{1,n} \otimes \mathcal{P}_{2,n} \otimes \mathcal{P}_{3,n} }   \\
&\qquad = 0. 
\end{align*}
Therefore, 
\begin{align*}
\norm{ x w_{0}^{*} w_{1} - w_{0}^{*} w_{1} x } < \frac{ \delta} { 2 }
\end{align*}
for all $x \in \mathcal{G}$,
\begin{equation*}
\mathrm{bott}_{0} ( \id_{ \mathfrak{A} \otimes \mathsf{M}_{ \mathfrak{p} } \otimes \mathsf{M}_{ \mathfrak{q} } }  , w_{0}^{*} w_{1} ) \vert_{ \mathcal{H} } = 0,
\end{equation*}
and 
\begin{equation*}
\mathrm{bott}_{1} ( \id_{ \mathfrak{A} \otimes \mathsf{M}_{ \mathfrak{p} } \otimes \mathsf{M}_{ \mathfrak{q} } }  , w_{0}^{*} w_{1} ) \vert_{ \mathcal{H} } = 0.
\end{equation*}
Hence, there exists a continuous path of unitaries $v(t)$ in $\mathfrak{A} \otimes \mathsf{M}_{ \mathfrak{p} } \otimes \mathsf{M}_{ \mathsf{q} }$ such that 
\begin{equation*}
\norm{ v(t) ( a \otimes 1_{ \mathsf{M}_{ \mathfrak{p} } } \otimes 1_{ \mathsf{M}_{ \mathfrak{q} } } ) - ( a \otimes 1_{ \mathsf{M}_{ \mathfrak{p} } } \otimes 1_{ \mathsf{M}_{ \mathfrak{q} } } ) v(t) } < \frac{ \epsilon }{ 2 }
\end{equation*}
for all $a \in \mathcal{F}$ and for all $t \in [ 0 ,1 ]$, $v( 0 ) = 1_{ \mathfrak{A} } \otimes 1_{ \mathsf{M}_{ \mathfrak{p} } } \otimes 1_{ \mathsf{M}_{ \mathfrak{q} } }$ and $v(1) = w_{0}^{*} w_{1}$.  

Set $w(t) = w_{0} v(t)$.  Then $w(t)$ is a continuous path of unitaries in $\mathfrak{A} \otimes \mathsf{M}_{ \mathfrak{p} } \otimes \mathsf{M}_{ \mathfrak{q} }$ such that $w(0) = w_{0}$, $w(1) = w_{1}$, and 
\begin{equation*}
\norm{ w(t) ( a \otimes 1_{ \mathsf{M}_{ \mathfrak{p} } } \otimes 1_{ \mathsf{M}_{ \mathfrak{q} } } ) - ( a \otimes 1_{ \mathsf{M}_{ \mathfrak{p} } } \otimes 1_{ \mathsf{M}_{ \mathfrak{q} } } ) w(t) } < \epsilon
\end{equation*}
for all $a \in \mathcal{F}$.

Note that $[ v_{0} ] = \gamma \otimes K_{0} ( \id_{ \mathsf{M} _{ \mathfrak{p} } } ) ( [ 1_{ \mathfrak{A} } \otimes 1_{ \mathsf{M}_{ \mathfrak{p} } } ] ) = \gamma ( [ 1_{ \mathfrak{A} } ] ) \otimes [ 1_{ \mathsf{M}_{ \mathfrak{p} } } ]  = [ u \otimes 1_{ \mathsf{M}_{ \mathfrak{p} } } ]$ in $K_{1} ( \mathfrak{A} \otimes \mathsf{M}_{ \mathfrak{p} } )$ and $[ v_{1} ] = \gamma \otimes K_{1} ( \id_{ \mathsf{M} _{ \mathfrak{q} } } ) ( [ 1_{ \mathfrak{A} } \otimes 1_{ \mathsf{M}_{ \mathfrak{q} } } ] ) = \gamma ( [ 1_{ \mathfrak{A} } ] ) \otimes [ 1_{ \mathsf{M}_{ \mathfrak{q} } } ] = [ u \otimes 1_{ \mathsf{M}_{ \mathfrak{q} } } ]$ in $K_{0} ( \mathfrak{A} \otimes \mathsf{M}_{ \mathfrak{q} } )$.  Hence, $[ w(0) ] = [ u \otimes 1_{ \mathsf{M}_{ \mathfrak{p} } } \otimes 1_{ \mathsf{M}_{ \mathfrak{q} } } ]$ in $K_{1} ( \mathfrak{A} \otimes \mathsf{M}_{ \mathfrak{p} } \otimes 1_{ \mathsf{M}_{ \mathfrak{q} } } )$ and $[ w(1) ] = [ u \otimes 1_{ \mathsf{M}_{ \mathfrak{p} } } \otimes 1_{ \mathsf{M}_{ \mathfrak{q} } } ]$ in $K_{1} ( \mathfrak{A} \otimes 1_{ \mathsf{M}_{ \mathfrak{p} } }  \otimes \mathsf{M}_{ \mathfrak{p} } )$.
\end{proof}

\section{The automorphism group of a $\mathcal{Z}$-stable $C^{*}$-algebra}\label{results}

Let $\mathfrak{A}$ be a separable, unital $C^{*}$-algebra and let $\{ G_{n} \}_{ n = 1}^{ \infty }$ be an increasing sequence of finitely generated subgroups of $K_{0} ( \mathfrak{A} )$ containing $[ 1_{ \mathfrak{A} } ]$.  Note that there exists a sequence of surjective homomorphisms 
\begin{align*}
\frac{ K_{1} ( \mathfrak{A} ) }{ H_{ [ 1_{\mathfrak{A} } ] } ( G_{1} , K_{1} ( \mathfrak{A} ) ) } \leftarrow \frac{ K_{1} ( \mathfrak{A} ) }{ H_{ [ 1_{\mathfrak{A} } ] } ( G_{2} , K_{1} ( \mathfrak{A} ) ) } \leftarrow \cdots
\end{align*}
which gives an inverse limit group $\displaystyle \lim_{ \leftarrow } \frac{ K_{1} ( \mathfrak{A} ) }{ H_{ [ 1_{\mathfrak{A} } ] } ( G_{n} , K_{1} ( \mathfrak{A} ) ) }$.  For each $n \in \N$, equip $\frac{ K_{1} ( \mathfrak{A} ) }{ H_{ [ 1_{\mathfrak{A} } ] } ( G_{n} , K_{1} ( \mathfrak{A} ) ) }$ with the discrete topology and give the inverse limit the inverse limit topology.  

For each $g \in K_{1} ( \mathfrak{A} )$, denote the image of $g$ in the inverse limit by $\check{g}$.  For each $g \in U ( \mathfrak{A} )$, denote the image of $g$ in $\frac{ \overline{\mathrm{Inn} }( \mathfrak{A} ) }{  \overline{\mathrm{Inn} }_{0}( \mathfrak{A} ) }$ by $\hat{g} = \mathrm{Ad} ( g )+ \overline{ \mathrm{Inn} }_{0} ( \mathfrak{A} )$.  The authors in \cite{pr_auto} showed that the map $\mu : \hat{g} \mapsto \check{g}$ extends uniquely to a continuous group homomorphism from $\frac{ \overline{\mathrm{Inn} }( \mathfrak{A} ) }{  \overline{\mathrm{Inn} }_{0}( \mathfrak{A} ) }$ to $\displaystyle \lim_{ \leftarrow } \frac{ K_{1} ( \mathfrak{A} ) }{ H_{ [ 1_{\mathfrak{A} } ] } ( G_{n} , K_{1} ( \mathfrak{A} ) ) }$.  With an abuse of notation, we denote this unique extension by $\mu$.  Moreover, by Theorem 2 of \cite{pr_auto}, if $\mathfrak{A}$ satisfies Property (C) (see Definition \ref{d:propertyC}), then $\mu$ is surjective and if, in addition, the natural map from $U( \mathfrak{A} ) / U_{0} ( \mathfrak{A} )$ to $K_{1} ( \mathfrak{A} )$ is an isomorphism, then $\mu$ is an isomorphism of topological abelian groups.  Consequently,  $\frac{ \overline{\mathrm{Inn} }( \mathfrak{A} ) }{ \overline{\mathrm{Inn} }_{0}( \mathfrak{A} ) }$ is a totally disconnected group.  Moreover, the authors showed that if $\mathfrak{A}$ is a unital $C^{*}$-algebra satisfying the UCT and if $\mathfrak{A}$ is either a Kirchberg algebra or simple tracially AI algebra, then $\mathfrak{A}$ satisfies Property (C).  We use the results from the previous section to show that every element in $\mathcal{A}_{ \mathcal{Z} }$ satisfies Property (C).   

\subsection{$\mathcal{Z}$-stable $C^{*}$-algebras and Property (C)}\label{propertyC}

\begin{definition}\label{d:propertyC}
Let $\mathfrak{A}$ be a unital $C^{*}$-algebra.  $\mathfrak{A}$ is said to satisfy \emph{\textbf{Property (C)}} if for each $\epsilon > 0$ and for each finite subset $\mathcal{F}$ of $\mathfrak{A}$, there exists a finitely generated subgroup $\mathcal{G}$ of $K_{0} ( \mathfrak{A} )$ containing $[ 1_{ \mathfrak{A} } ]$ such that the following holds: for every $u \in U( \mathfrak{A} )$ with $[ u ] \in H_{ [ 1_{ \mathfrak{A} } ] } ( \mathcal{G} , K_{1} ( \mathfrak{A} ) )$, there exists $w \in U( \mathfrak{A} )$ such that 
\begin{equation*}
\norm{ w a - a w } < \epsilon
\end{equation*}
for all $a \in \mathcal{F}$ and $[ w ] = [ u ]$ in $U( \mathfrak{A} ) / U( \mathfrak{A} )_{0}$.
\end{definition}

The following is an equivalent definition of Property (C) which will be easier to check in our context.  

\begin{lemma}\label{l:equivalentstate}
Let $\mathfrak{A}$ be a unital $C^{*}$-algebra.  Then the following are equivalent:
\begin{itemize}
\item[(1)] $\mathfrak{A}$ satisfies Property (C).

\item[(2)] For each $\epsilon > 0$ and for each finite subset $\mathcal{F}$ of $\mathfrak{A}$, there exists a finitely generated subgroup $\mathcal{G}$ of $K_{0} ( \mathfrak{A} )$ containing $[ 1_{ \mathfrak{A} } ]$ such that the following holds: for every $u \in U( \mathfrak{A} )$ with $[ u ] \in H_{ [ 1_{ \mathfrak{A} } ] } ( \mathcal{G} , K_{1} ( \mathfrak{A} ) )$, there exists $w \in U ( \mathfrak{A} )_{0}$ such that
\begin{equation*}
\norm{ u a u^{*} - w a w^{*} } < \epsilon
\end{equation*}
for all $a \in \mathcal{F}$.
\end{itemize}
\end{lemma}

\begin{proof}
$(1) \implies (2):$  Let $\epsilon > 0$ and $\mathcal{F}$ be a finite subset of $\mathfrak{A}$.  Let $\mathcal{G}$ be the finitely generated subgroup of $K_{0} ( \mathfrak{A} )$ given in Property (C) that corresponds to $\epsilon$ and $\mathcal{F}$.  Let $u \in U( \mathfrak{A} )$ with $[ u ] \in H_{ [ 1_{ \mathfrak{A} } ] } ( \mathcal{G} , K_{1} ( \mathfrak{A} ) )$.  Then, by our assumption, there exists $v \in U( \mathfrak{A} )$ such that 
\begin{equation*}
\norm{ v a - a v } < \epsilon
\end{equation*}
for all $a \in \mathcal{F}$ and $u v^{*} \in U( \mathfrak{A} )_{0}$.  Set $w = u v^{*}$.  Then $w \in U( \mathfrak{A} )_{0}$ and 
\begin{equation*}
\norm{ u a u^{*} - w a w^{*} } = \norm{ a - v^{*} a v } = \norm{ v a - a v }  < \epsilon
\end{equation*}
for all $a \in \mathcal{F}$. 

$(2) \implies (1):$  Let $\epsilon > 0$ and $\mathcal{F}$ be a finite subset of $\mathfrak{A}$.  Let $\mathcal{G}$ be the finitely generated subgroup of $K_{0} ( \mathfrak{A} )$ given in (2) that corresponds to $\epsilon$ and $\mathcal{F}$.  Suppose $u \in U( \mathfrak{A} )$ such that $[ u ] \in H_{ [ 1_{ \mathfrak{A} } ] } ( \mathcal{G} , K_{1} ( \mathfrak{A} ) )$. Then, by our assumption, there exists $v \in U ( \mathfrak{A} )_{0}$ such that 
\begin{equation*}
\norm{ u a u^{*} - v a v^{*} } < \epsilon
\end{equation*}  
for all $a \in \mathcal{F}$.  Set $w = v^{*} u$.  Since $[ v ] \in U( \mathfrak{A} )_{ 0 }$, $[ w ] = [ u ]$ in $U( \mathfrak{A} ) / U( \mathfrak{A} )_{0}$.  Moreover, 
\begin{equation*}
\norm{ w a - a w} = \norm{ v^{*} u a - a v^{*} u } = \norm{ u a u^{*} - v a v^{*} } < \epsilon
\end{equation*}
for all $a \in \mathcal{F}$. 
\end{proof}

\begin{lemma}\label{l:approxunit}
Let $\mathfrak{A}$ and $\mathfrak{B}$ be separable $C^{*}$-algebras, with $\mathfrak{B}$ unital and $\mathcal{Z}$-stable and let $\mathfrak{p}$ and $\mathfrak{q}$ be supernatural numbers which are relatively prime.  Suppose $\ftn{ \varphi_{1} , \varphi_{2} }{ \mathfrak{A} }{ \mathfrak{B} }$ are $*$-homomorphisms such that 
\begin{equation*}
\ftn{ \varphi_{1} \otimes 1_{ \mathcal{Z}_{ \mathfrak{p} , \mathfrak{q} } }, \varphi_{2} \otimes 1_{ \mathcal{Z}_{ \mathfrak{p} , \mathfrak{q} } } }{ \mathfrak{A} }{ \mathfrak{B} \otimes \mathcal{Z}_{ \mathfrak{p} , \mathfrak{q} } }
\end{equation*}
are approximately unitarily equivalent via unitaries $\{ w_{n} \}_{ n \in \N}$ with $[ w_{n} ] = 0$ in $K_{1} ( \mathfrak{B} \otimes \mathcal{Z}_{ \mathfrak{p} , \mathfrak{q} } )$.  Then $\varphi_{1}, \varphi_{2}$ are approximately unitarily equivalent via unitaries $\{ v_{n} \}_{ n \in \N }$ in $\mathfrak{B}$ such that $[ v_{n} ] = 0$ in $K_{1} ( \mathfrak{B} )$.

Consequently, for every $\epsilon > 0$ and for every finite subset $\mathcal{F}$ of $\mathfrak{A}$, there exists $\delta > 0$ such that if $w$ is a unitary in $\mathfrak{B} \otimes \mathcal{Z}_{ \mathfrak{p} , \mathfrak{q} }$ such that 
\begin{equation*}
\norm{ ( \varphi_{1} \otimes 1_{ \mathcal{Z}_{ \mathfrak{p} , \mathfrak{q} } }  )( a ) - w ( \varphi_{2} \otimes 1_{ \mathcal{Z}_{ \mathfrak{p} , \mathfrak{q} } } ) ( a ) w^{*} } < \delta 
\end{equation*}
for all $a \in \mathcal{F}$ and $[ w ] = 0$  in $K_{1} ( \mathfrak{B} \otimes \mathcal{Z}_{ \mathfrak{p} , \mathfrak{q} } )$, then there exists a unitary $v$ in $\mathfrak{B}$ with $[ v ] = 0$ in $K_{1} ( \mathfrak{B} )$ such that
\begin{equation*}
\norm{ \varphi_{1} ( a ) - v \varphi_{2} ( a ) v^{*} } < \epsilon
\end{equation*}
for all $a \in \mathcal{F}$.
\end{lemma}

\begin{proof}
Since $\mathfrak{B} \cong \mathfrak{B} \otimes \mathcal{Z} \cong \mathfrak{B} \otimes \mathcal{Z} \otimes \mathcal{Z}$, by Lemma \ref{l:isoz}, there exist a $*$-isomorphism $\ftn{ \nu_{ \mathfrak{B} } }{ \mathfrak{B} }{ \mathfrak{B} \otimes \mathcal{Z} }$ and a sequence of unitaries $\{ u_{n} \}_{ n \in \N }$ in $\mathfrak{B} \otimes \mathcal{Z}$ such that
\begin{equation*}
\lim_{ n \to \infty } \norm{ \nu_{ \mathfrak{B} } ( b ) - u_{n} ( b \otimes 1_{ \mathcal{Z} } ) u_{n}^{*} } = 0
\end{equation*}
for all $b \in \mathfrak{B}$.  Let $\ftn{ \overline{ \sigma }_{ \mathfrak{p} , \mathfrak{q} } }{ \mathcal{Z}_{ \mathfrak{p} , \mathfrak{q} } }{ \mathcal{Z} }$ be the unital embedding in Proposition 3.4 of \cite{ww_localelliott}.
Note that 
\begin{equation*}
\lim_{ n \to \infty } \norm{ \varphi_{i} ( a ) - \nu_{ \mathfrak{B} }^{-1} ( u_{n} ) \nu_{ \mathfrak{B} }^{-1}( \varphi_{i} ( a ) \otimes 1_{ \mathcal{Z} } ) \nu_{ \mathfrak{B} }^{-1} ( u_{n}^{*} ) } = 0
\end{equation*}
and 
\begin{align*}
\varphi_{1} ( a  ) \otimes 1_{ \mathcal{Z} } &= \varphi_{1} ( a ) \otimes \overline{\sigma}_{ \mathfrak{p} , \mathfrak{q} } ( 1_{ \mathcal{Z}_{ \mathfrak{p} , \mathfrak{q} } } ) \\
						&= ( \id_{ \mathfrak{B} } \otimes \overline{\sigma}_{ \mathfrak{p} , \mathfrak{q} } ) ( \varphi_{1} ( a ) \otimes 1_{ \mathcal{Z}_{ \mathfrak{p} , \mathfrak{q} } } ) \\   
\varphi_{2} ( a  ) \otimes 1_{ \mathcal{Z} } &= \varphi_{1} ( a ) \otimes \overline{\sigma}_{ \mathfrak{p} , \mathfrak{q} } ( 1_{ \mathcal{Z}_{ \mathfrak{p} , \mathfrak{q} } } ) \\
					&= ( \id_{ \mathfrak{B} } \otimes \overline{\sigma}_{ \mathfrak{p} , \mathfrak{q} } )( \varphi_{2} ( a ) \otimes 1_{ \mathcal{Z}_{ \mathfrak{p} , \mathfrak{q} } } )
\end{align*}
for all $a \in \mathfrak{A}$.  Therefore,
\begin{align*}
&\lim_{ n \to \infty } \norm{ \varphi_{1} ( a ) \otimes 1_{ \mathcal{Z} } - ( \id_{ \mathfrak{B} } \otimes \overline{\sigma}_{ \mathfrak{p} , \mathfrak{q} } )( w_{n} ) ( \varphi_{2} ( a ) \otimes 1_{ \mathcal{Z} } ) ( \id_{ \mathfrak{B} } \otimes \overline{\sigma}_{ \mathfrak{p} , \mathfrak{q} } )( w_{n}^{*} ) } \\
&=\lim_{ n \to \infty } \norm{ (\id_{ \mathfrak{B} } \otimes \overline{\sigma}_{ \mathfrak{p} , \mathfrak{q} } ) (  \varphi_{1} ( a ) \otimes 1_{ \mathcal{Z}_{ \mathfrak{p} , \mathfrak{q} } } ) - (\id_{ \mathfrak{B} } \otimes \overline{\sigma}_{ \mathfrak{p} , \mathfrak{q} } ) ( w_{n}  (\varphi_{2} ( a ) \otimes 1_{ \mathcal{Z}_{ \mathfrak{p} , \mathfrak{q} } } )w_{n}^{*} ) } \\
&= 0
\end{align*}
for all $a \in \mathfrak{ \mathfrak{A} }$.  Set $v_{n} = \nu_{ \mathfrak{B} }^{-1} \left( u_{n}  \left[\id_{ \mathfrak{B} } \otimes \overline{\sigma}_{ \mathfrak{p} , \mathfrak{q} } \right]( w_{n} )  u_{n}^{*} \right)$.  Let $a \in \mathfrak{A}$.  Then 
\begin{align*}
&\norm{ \varphi_{1} ( a ) - v_{n} \varphi_{2} ( a ) v_{n}^{*} } \\
&\quad \leq \norm{ \varphi_{1} ( a ) - \nu_{ \mathfrak{B} }^{-1} ( u_{n} ) \nu_{ \mathfrak{B} }^{-1} ( \varphi_{1} ( a ) \otimes 1_{ \mathcal{Z} } ) \nu_{ \mathfrak{B} }^{-1} ( u_{n}^{*} ) } \\
&\qquad + \| \nu_{ \mathfrak{B} }^{-1} ( u_{n} ) \nu_{ \mathfrak{B} }^{-1} ( \varphi_{1} ( a ) \otimes 1_{ \mathcal{Z} } ) \nu_{ \mathfrak{B} }^{-1} ( u_{n}^{*} )  \\
&\qquad \qquad  - \nu_{ \mathfrak{B} }^{-1}  ( u_{n} ( \id_{ \mathfrak{B} } \otimes \overline{ \sigma }_{ \mathfrak{p} , \mathfrak{q} } )( w_{n} ) ) \nu_{ \mathfrak{B} }^{-1} ( \varphi_{2} ( a ) \otimes 1_{ \mathcal{Z} } ) \nu_{ \mathfrak{B} }^{-1}  ( ( \id_{ \mathfrak{B} } \otimes \overline{ \sigma }_{ \mathfrak{p} , \mathfrak{q} } )( w_{n}^{*} ) u_{n}^{*}  ) \| \\
&\qquad + \norm{ \nu_{ \mathfrak{B} }^{-1}  ( u_{n} ( \id_{ \mathfrak{B} } \otimes \overline{ \sigma }_{ \mathfrak{p} , \mathfrak{q} } )( w_{n} ) ) \nu_{ \mathfrak{B} }^{-1} ( \varphi_{2} ( a ) \otimes 1_{ \mathcal{Z} } ) \nu_{ \mathfrak{B} }^{-1}  ( ( \id_{ \mathfrak{B} } \otimes \overline{ \sigma }_{ \mathfrak{p} , \mathfrak{q} } )( w_{n}^{*} ) u_{n}^{*}  ) - v_{n} \varphi_{2} ( a ) v_{n}^{*}  } \\
&\quad \leq 
\norm{ \varphi_{1} ( a ) - \nu_{ \mathfrak{B} }^{-1} ( u_{n} ) \nu_{ \mathfrak{B} }^{-1} ( \varphi_{1} ( a ) \otimes 1_{ \mathcal{Z} } ) \nu_{ \mathfrak{B} }^{-1} ( u_{n}^{*} ) } \\
&\qquad + \norm{ \varphi_{1} ( a ) \otimes 1_{ \mathcal{Z} } - ( \id_{ \mathfrak{B} } \otimes \overline{ \sigma }_{ \mathfrak{p} , \mathfrak{q} })( w_{n} ) (\varphi_{2} ( a ) \otimes 1_{ \mathcal{Z} } )( \id_{ \mathfrak{B} } \otimes \overline{ \sigma }_{ \mathfrak{p} , \mathfrak{q} })( w_{n}^{*} ) }  \\
&\qquad + \norm{ \varphi_{2} ( a ) - \nu_{ \mathfrak{B} }^{-1} ( u_{n} ) \nu_{ \mathfrak{B} }^{-1} ( \varphi_{2} ( a ) \otimes 1_{ \mathcal{Z} } ) \nu_{ \mathfrak{B} }^{-1} ( u_{n}^{*} ) }. 
\end{align*}
Therefore, for each $a \in \mathfrak{A}$, 
\begin{equation*}
\lim_{ n \to \infty } \norm{ \varphi_{1} ( a ) - v_{n} \varphi_{2} ( a ) v_{n}^{*} } = 0.
\end{equation*}
Since $[ w_{n} ] = 0$ in $K_{1} ( \mathfrak{B} \otimes \mathcal{Z}_{ \mathfrak{p} , \mathfrak{q} } )$, then 
\begin{align*}
[ v_{n} ] &= \left( K_{1} ( \nu_{ \mathfrak{B} }^{-1} ) \circ K_{1} ( \mathrm{Ad} ( u_{n} ) ) \circ K_{1} ( \id_{ \mathfrak{B} } \otimes \overline{ \sigma }_{ \mathfrak{p} , \mathfrak{q} } ) \right) ( [ w_{n} ] )\\
		&= 0
\end{align*}
in $K_{1} ( \mathfrak{B} )$.
\end{proof}

\begin{lemma}\label{l:k1iso}
Let $\mathfrak{A}$ be in $\mathcal{A}_{ \mathcal{Z} }$.  Then the canonical homomorphism from $U( \mathfrak{A} ) / U ( \mathfrak{A} )_{0}$ to $K_{1} ( \mathfrak{A} )$ is an isomorphism.
\end{lemma}

\begin{proof}
Let $\mathfrak{p}$ be a supernatural number of infinite type.  Then $\mathfrak{A} \otimes \mathsf{M}_{ \mathfrak{p} }$ is a tracially AI algebra.  Thus $\mathfrak{A}$ is a finite $\mathcal{Z}$-stable $C^{ * }$-algebra.  Hence, by \cite{MRrrZabs}, $\mathfrak{A}$ has stable rank one.  Since $\mathfrak{A}$ is a simple unital $C^{*}$-algebra with stable rank one, by Theorem 10.12 of \cite{mr_dimstablerk} and Corollary 7.14 of \cite{BlackComp}, the canonical homomorphism from $U( \mathfrak{A} ) / U ( \mathfrak{A} )_{0}$ to $K_{1} ( \mathfrak{A} )$ is an isomorphism.
\end{proof}

\begin{theorem}\label{t:strongapprox}
Let $\mathfrak{A}$ be in $\mathcal{A}_{ \mathcal{Z} }$.  Then $\mathfrak{A}$ satisfies Property (C).
\end{theorem}

\begin{proof}
By Lemma \ref{l:equivalentstate}, it is enough to prove (2) of Lemma \ref{l:equivalentstate}.  Let $\epsilon > 0$ and let $\mathcal{F}$ be a finite subset of $\mathfrak{A}$.  Let $\delta > 0$ be the quantity given in Lemma \ref{l:approxunit} corresponding to $\mathcal{F}$ and $\epsilon$.  Let $\mathcal{G}$ be the finitely generated subgroup of $K_{0} ( \mathfrak{A} )$ given in Theorem \ref{t:approxcomm} corresponding to $\mathcal{F}$ and $\frac{ \delta }{ 2 }$.

Let $u \in U( \mathfrak{A} )$ with $[u] \in H_{ [ 1_{ \mathfrak{A} } ] } ( \mathcal{G} , K_{1} ( \mathfrak{A} ) )$.  Let $\mathfrak{p}$ and $\mathfrak{q}$ be supernatural numbers of infinite type such that $\mathsf{M}_{ \mathfrak{p} } \otimes \mathsf{M}_{ \mathfrak{q} }$ is isomorphic to $\mathcal{Q}$.  By Theorem \ref{t:approxcomm}, then there exists a continuous path of unitaries $w(t)$ in $\mathfrak{A} \otimes \mathsf{M}_{ \mathfrak{p} } \otimes \mathsf{M}_{ \mathfrak{q} }$ such that 
\begin{itemize}
\item[(1)] $w(0) \in \mathfrak{A} \otimes \mathsf{M}_{ \mathfrak{p} } \otimes 1_{ \mathsf{M}_{ \mathfrak{q} } }$ and $w(1) \in  \mathfrak{A} \otimes 1_{ \mathsf{M}_{ \mathfrak{p} } }  \otimes \mathsf{M}_{ \mathfrak{q} }$;

\item[(2)] $[ w(0) ] = [ u \otimes 1_{ \mathsf{M}_{ \mathfrak{p} } } \otimes 1_{ \mathsf{M}_{ \mathfrak{q} } } ]$ in $K_{1} ( \mathfrak{A} \otimes \mathsf{M}_{ \mathfrak{p} } \otimes 1_{ \mathsf{M}_{ \mathfrak{q} } } )$ and $[ w(1) ] = [ u \otimes 1_{ \mathsf{M}_{ \mathfrak{p} } } \otimes 1_{ \mathsf{M}_{ \mathfrak{q} } } ]$ in $K_{1} ( \mathfrak{A} \otimes 1_{ \mathsf{M}_{ \mathfrak{p} } }  \otimes \mathsf{M}_{ \mathfrak{q} } )$; and

\item[(3)] $\norm{ w(t) ( a \otimes 1_{ \mathsf{M}_{ \mathfrak{p} } } \otimes 1_{ \mathsf{M}_{ \mathfrak{q} } }  ) - ( a \otimes 1_{ \mathsf{M}_{ \mathfrak{p} } } \otimes 1_{ \mathsf{M}_{ \mathfrak{q} } }  ) w(t) } < \frac{ \delta }{ 2 }$ for all $a \in \mathcal{F}$ and $t \in [ 0 , 1 ]$.
\end{itemize}
Since 
\begin{equation*}
\mathcal{Z}_{ \mathfrak{p} , \mathfrak{q} } = \setof{ f \in C\left( [ 0 , 1 ], \mathsf{M}_{ \mathfrak{p} } \otimes \mathsf{M}_{ \mathfrak{q} } \right) }{ f(0) \in \mathsf{M}_{ \mathfrak{p} } \otimes 1_{ \mathsf{M}_{ \mathfrak{q} } } \ \text{and} \ f(1) \in 1_{ \mathsf{M}_{ \mathfrak{p} } } \otimes \mathsf{M}_{ \mathfrak{q} } },   
\end{equation*}
$\mathfrak{A} \otimes \Z_{ \mathfrak{p} , \mathfrak{q} }$ can be identified with the $C^{*}$-algebra 
\begin{align*}
\setof{ f \in C\left( [ 0 , 1 ], \mathfrak{A} \otimes \mathsf{M}_{ \mathfrak{p} } \otimes \mathsf{M}_{ \mathfrak{q} } \right) }{ f(0) \in \mathfrak{A} \otimes \mathsf{M}_{ \mathfrak{p} } \otimes 1_{ \mathsf{M}_{ \mathfrak{q} } } \ \text{and} \ f(1) \in \mathfrak{A} \otimes 1_{ \mathsf{M}_{ \mathfrak{p} } } \otimes \mathsf{M}_{ \mathfrak{q} } }.
\end{align*}
With this identification, $w$ is an element of $U( \mathfrak{A} \otimes \mathcal{Z}_{ \mathfrak{p} , \mathfrak{q} } )$ such that
\begin{equation*}
\norm{ w ( a \otimes 1_{ \mathcal{Z}_{ \mathfrak{p} , \mathfrak{q} } }  ) - ( a \otimes 1_{ \mathcal{Z}_{ \mathfrak{p} , \mathfrak{q} } } ) w } < \delta
\end{equation*}
for all $a \in \mathcal{F}$.  Since  $[ w(0) ] = [ u \otimes 1_{ \mathsf{M}_{ \mathfrak{p} } } \otimes 1_{ \mathsf{M}_{ \mathfrak{q} } } ]$ in $K_{1} ( \mathfrak{A} \otimes \mathsf{M}_{ \mathfrak{p} } \otimes 1_{ \mathsf{M}_{ \mathfrak{q} } } )$ and $[ w(1) ] = [ u \otimes 1_{ \mathsf{M}_{ \mathfrak{p} } } \otimes 1_{ \mathsf{M}_{ \mathfrak{q} } } ]$ in $K_{1} ( \mathfrak{A} \otimes 1_{ \mathsf{M}_{ \mathfrak{p} } }  \otimes \mathsf{M}_{ \mathfrak{q} } )$, we have that 
\begin{equation*}
[ \left( ( u \otimes 1_{ \mathcal{Z}_{ \mathfrak{p}, \mathfrak{q} } }) w^{*} \right) (0) ] = 0 
\end{equation*}
in $K_{1} ( \mathfrak{A} \otimes \mathsf{M}_{ \mathfrak{p} } \otimes 1_{ \mathsf{M}_{ \mathfrak{q} } } )$ and 
\begin{equation*}
[ \left( (u \otimes 1_{ \mathcal{Z}_{ \mathfrak{p}, \mathfrak{q} } } )w^{*} \right) (1) ] = 0 
\end{equation*}
in $K_{1} ( \mathfrak{A} \otimes 1_{ \mathsf{M}_{ \mathfrak{p} } }  \otimes \mathsf{M}_{ \mathfrak{p} } )$.  By Proposition 5.2 of \cite{ww_localelliott}, $[  ( u \otimes 1_{ \mathcal{Z}_{ \mathfrak{p}, \mathfrak{q} } }) w^{*}  ] = 0$ in $K_{1} ( \mathfrak{A} \otimes \mathcal{Z}_{ \mathfrak{p} , \mathfrak{q} } )$.  Let $a \in \mathcal{F}$.  Then
\begin{equation*}
\norm{ ( \mathrm{Ad} ( u ) \otimes 1_{ \mathcal{Z}_{ \mathfrak{p} , \mathfrak{q} } } )( a ) - \mathrm{Ad} \left( (u \otimes 1_{ \mathcal{Z}_{ \mathfrak{p} , \mathfrak{q} }  })  w^{*} \right)( a \otimes 1_{ \mathcal{Z}_{ \mathfrak{p} , \mathfrak{q}  } } ) } < \delta.
\end{equation*}
Hence, by Lemma \ref{l:approxunit}, there exists a unitary $v \in U( \mathfrak{A} )$ such that $[ v ] = 0$ in $K_{1} ( \mathfrak{A} )$ and 
\begin{equation*}
\norm{ u a u^{*} - v a v^{*} } < \epsilon
\end{equation*}
for all $a \in \mathcal{F}$.  By Lemma \ref{l:k1iso}, $v \in U ( \mathfrak{A} )_{0}$.  We have just shown that (2) of Lemma  \ref{l:equivalentstate} holds.  Therefore, $\mathfrak{A}$ satisfies Property (C).
\end{proof}

\subsection{The structure of the automorphism group of a $C^{*}$-algebra in $\mathcal{A}_{ \mathcal{Z} }$.}

Recall from \cite{multcoeff}, there is a natural order structure on $\underline{K}(\mathfrak{A} )$.  Let $\mathrm{Aut}(\underline{K}(\mathfrak{A}))$ be the topological group of all ordered group isomorphisms of $\underline{K}(\mathfrak{A})$, which preserve the grading and the Bockstein operations in \cite{multcoeff}.  For a unital $C^*$-algebra $\mathfrak{A}$, let $\mathrm{Aut} ( \underline{K} ( \mathfrak{A} ) )_{1}$ be the subgroup of $\mathrm{Aut} ( \underline{K}(  \mathfrak{A} ) )$ which sends $[ 1_{ \mathfrak{A} } ]$ to $[ 1_{ \mathfrak{A} } ]$.  Set
\begin{align*}
J(\mathfrak{A} ) = (\underline{K}(\mathfrak{A}), T(\mathfrak{A}) , U( \mathfrak{A} ) / CU( \mathfrak{A} ) ).
\end{align*}
Then an element of $\mathrm{Aut} (J(\mathfrak{A} ) )_1$ is an ordered tuple
$(\alpha, \lambda_T, \lambda_{U} )$ where 
$\alpha \in \mathrm{Aut} ( \underline{K} ( \mathfrak{A} ) )_{1}$, $\ftn{ \lambda_T }{ T( \mathfrak{A} ) }{ T( \mathfrak{A} ) }$ is an affine homeomorphism, and $\ftn{ \lambda_{U} }{ U( \mathfrak{A} ) / CU( \mathfrak{A} ) }{ U( \mathfrak{A} ) / CU( \mathfrak{A} ) }$ is a homeomorphism such that $\alpha$, $\lambda_{T}$, and $\lambda_{U}$ are compatible is the sense of Section 2, pp.~5 of \cite{hlzn:hom}.
  
\begin{theorem}
Let $\mathfrak{A}$ be a $C^{*}$-algebra in $\mathcal{A}_{ \mathcal{Z} }$ satisfying the UCT.  Then 
\begin{itemize}
\item[(a)] $\overline{\mathrm{Inn}}_{0} ( \mathfrak{A} )$ is a simple topological group;

\item[(b)] $\frac{ \overline{ \mathrm{Inn} } ( \mathfrak{A} ) }{ \overline{ \mathrm{Inn} }_{0} ( \mathfrak{A} ) }$ is totally disconnected; and

\item[(c)] the sequence
\begin{align*}
0 \to \overline{\mathrm{Inn}} ( \mathfrak{A} ) \to \mathrm{Aut} ( \mathfrak{A} ) \to \mathrm{Aut} ( J( \mathfrak{A} ) )_{1}
\end{align*}
is exact.
\end{itemize}
In addition, if $K_{1} ( \mathfrak{A} )$ is torsion free, then the sequence in (c) becomes a short exact sequence
\begin{align*}
0 \to \overline{\mathrm{Inn}} ( \mathfrak{A} ) \to \mathrm{Aut} ( \mathfrak{A} ) \to \mathrm{Aut} ( J( \mathfrak{A} ) )_{1} \to 0.
\end{align*}
\end{theorem}  

\begin{proof}
(a) follows from Theorem \ref{t:simpletop}.  (b) follows from Theorem 2 of \cite{pr_auto} and Theorem \ref{t:strongapprox} and Lemma \ref{l:k1iso}.  (c) follows from Theorem 4.8 of \cite{hlzn:hom}.

Suppose $K_{1} ( \mathfrak{A} )$ is torsion free.  Let $( \alpha, \lambda_{T}, \lambda_{U} )$ be an element of $\mathrm{Aut} ( J( \mathfrak{A} ) )_{1}$.  By Theorem 5.9 of \cite{hlzn:hom}, there exist unital homomorphisms $\ftn{ \phi, \psi }{ \mathfrak{A} }{ \mathfrak{A} }$ such that 
\begin{itemize}
\item[(i)] $( \alpha, \lambda_{T}, \lambda_{U} )$ is induced by $\phi$ and 

\item[(ii)] $( \alpha^{-1} , \lambda_{T}^{-1}, \lambda_{U}^{-1} )$ is induced by $\psi$.  
\end{itemize}
Let $\{ \mathcal{F}_{n} \}_{n = 1}^{ \infty }$ be an increase sequence of finite subsets of $\mathfrak{A}$ whose union is dense in $\mathfrak{A}$.  Using Theorem 4.8 of \cite{hlzn:hom}, we get a sequence of unital homomorphisms $\ftn{ \phi_{n}, \psi_{n} }{ \mathfrak{A} }{ \mathfrak{A} }$ such that 
\begin{align*}
\norm{ (\psi_{n} \circ \phi_{n})( x )  - x } < \frac{1}{ 2^{n} } \quad \text{and} \quad \norm{ ( \phi_{n+1} \circ \psi_{n} )( a ) - a } < \frac{1}{2^{n} }
\end{align*}
for all $a \in \mathcal{F}_{n}$, and $\phi_{n}$ is unitarily equivalent to $\phi$ for each $n \in \N$ and $\psi_{n}$ is unitarily equivalent to $\psi$ for each $n \in \N$.  Therefore, we get an isomorphism $\ftn{ \beta }{ \mathfrak{A} }{ \mathfrak{A} }$ such that $\beta$ and $\phi$ induce the same element in $\mathrm{Aut} ( J( \mathfrak{A} ) )_{1}$, i.e., $\beta$ induces $( \alpha, \lambda_{T}, \lambda_{U} )$.
\end{proof}

\begin{remark}
Note that if $\mathfrak{A}$ is a nuclear, separable, unital, tracially AI algebra satisfying the UCT, then by \cite{linTR1}, the sequence
\begin{align*}
0 \to \overline{\mathrm{Inn}} ( \mathfrak{A} ) \to \mathrm{Aut} ( \mathfrak{A} ) \to \mathrm{Aut} ( J( \mathfrak{A} ) )_{1} \to 0
\end{align*}
is exact even when $K_{1} ( \mathfrak{A} )$ has an element with finite order.  Moreover, $\mathfrak{A}$ is an element of $\mathcal{A}_{ \mathcal{Z} }$.  The authors believe that the above sequence should be exact for an arbitrary $\mathfrak{A}$ in $\mathcal{A}_{\mathcal{Z}}$ satisfying the UCT but the proof has eluded the authors.  
\end{remark}

\section{Acknowledgement}  The authors are grateful to the referee for a careful reading of the paper and useful suggestions.

\end{document}